\documentclass[11pt, oneside, reqno]{article}

\usepackage{amsmath}
\usepackage{tabularx}
\usepackage{makecell}
\usepackage{multicol,graphicx,color}
\usepackage{pslatex}
\usepackage{authblk}
\usepackage{amssymb}
\usepackage{latexsym}
\usepackage{lscape}
\usepackage{epsfig}
\usepackage{pstricks}
\usepackage{amsfonts}
\usepackage{graphicx}
\usepackage{amsthm}
\usepackage{longtable}
\usepackage{enumerate}
\usepackage{exscale}
\usepackage{relsize}
\usepackage{subfigure}
\usepackage{cite}

\usepackage[dvipdfm]{hyperref}
\numberwithin{equation}{section}
\theoremstyle{plain}
\newtheorem{theorem}{Theorem}[section]
\newtheorem{corollary}[theorem]{Corollary}
\newtheorem{lemma}[theorem]{Lemma}

\newtheorem{remark}{Remark}
\begin{document}

\title{Action minimizers under topological constraints in the planar equal-mass four-body problem }
\author[1]{Duokui Yan \thanks{duokuiyan@buaa.edu.cn}}
\affil[1]{School of Mathematics and System Sciences \\
Beihang University, Beijing 100191, China}

\date{}

\maketitle
\begin{abstract}
It is shown that in the planar equal-mass four-body problem, there exist two sets of action minimizers connecting two planar boundary configurations with fixed symmetry axes and specific order constraints: a double isosceles configuration (Fig.~\ref{fig01}) and an isosceles trapezoid configuration (Fig.~\ref{fig02}). By applying the level estimate method, these minimizers are shown to be collision-free and they can be extended to two sets of periodic or quasi-periodic orbits. 
\end{abstract}

\section{Introduction}
After the the celebrated work of the figure-eight orbit \cite{CM}, variational method with topological constraints has been applied to show the existence of many new periodic orbits in the N-body problem. One of the main difficulties is to show that an action minimizer under topological constraints is collision-free. In the last two decades, much progress \cite{CM, CA, CV, CH, CH1, CH2, CH3, FT, FU, LO, Mar, Mon, Yu, Yu1, Yu2, Zh} has been made in this direction. There are mainly two methods in excluding possible collisions in an action minimizer under topological constraints. The first is called the local deformation method \cite{CV, CH3, FT, FU, Yu, Yu1}, in which one introduces a small deformation near an isolated collision, such that  the deformed path has a smaller action value than the one with collision. This method can be applied to show the existence of many periodic orbits, such as the Hip-Hop orbit in the spatial four-body problem \cite{CV, FT}, the planar choreographies in the N-body problem \cite{CM, FT, Yu1}, etc. The second is called the level estimate method. Basically, after obtaining a lower bound estimate of the action functional among all the collision paths in the admissible set, one can define a test path within the admissible set such that its action value is strictly less than the previous lower bound estimate. When applying this method, the main challenges are to find a good lower bound of action of collision paths and to define an appropriate test path with action value smaller than the lower bound. Crucial contributions have been made in \cite{CH, CH1, CH2, Zhang} and references therein. In \cite{CH2}, Chen introduced a braid group constraint \cite{Mon} and successfully showed the existence of retrograde and prograde orbits in the planar three-body problem. He introduced a binary decomposition method \cite{CH1}, such that the standard action functional can be decomposed to the sum of several Keplerian action functionals. A nice lower bound of action of all collision paths can be obtained by applying an estimate of the Keplerian action functionals \cite{CH, CH2}. Indeed, this binary decomposition method works not only for the three-body problem, but for all the N-body problems.

In this paper, we apply the level estimate method and the binary decomposition method to show the existence of two sets of minimizers under topological constraints in the planar equal-mass four-body problem. In fact, after the work of the figure-eight orbit \cite{CM}, many new orbits \cite{Simo, BR, Van, Ou1, Yan, Yan1} have been found numerically by searching for possible local action minimizers in appropriate loop spaces. Some of them can be characterized as local action minimizers connecting two given configurations. For example, Broucke \cite{BR} found an interesting set of choreographic solutions, which are so far the only known stable nontrivial choreographies other than the figure-eight orbit. In this set, there are basically two types of orbits: retrograde choreographic orbit and prograde choreographic orbit. Sample pictures are shown in Fig.~\ref{fig00}.  The highlighted paths in Fig.~\ref{fig00} (a) and (b) are the paths for $t \in [0,1]$ and both of them connect a double isosceles configuration (in dots) and an isosceles trapezoid configuration (in crosses) with fixed symmetry axes. Furthermore, the highlighted path in Fig.~\ref{fig00} (a) has a smaller action value than the highlighted path in Fig.~\ref{fig00} (b). In other words, the orbits in Fig.~\ref{fig00} (b) can be characterized as a local action minimizer connecting a double isosceles configuration and an isosceles trapezoid configuration with fixed symmetry axes. The existence of retrograde choreographic solutions has been shown in \cite{Ou1}. However, the existence of the prograde choreographies as Fig.~\ref{fig00} (b) is still open.
\begin{figure}[!htbp]
 \begin{center}
\subfigure[Retrograde Choreography]{\includegraphics[width=2.38in]{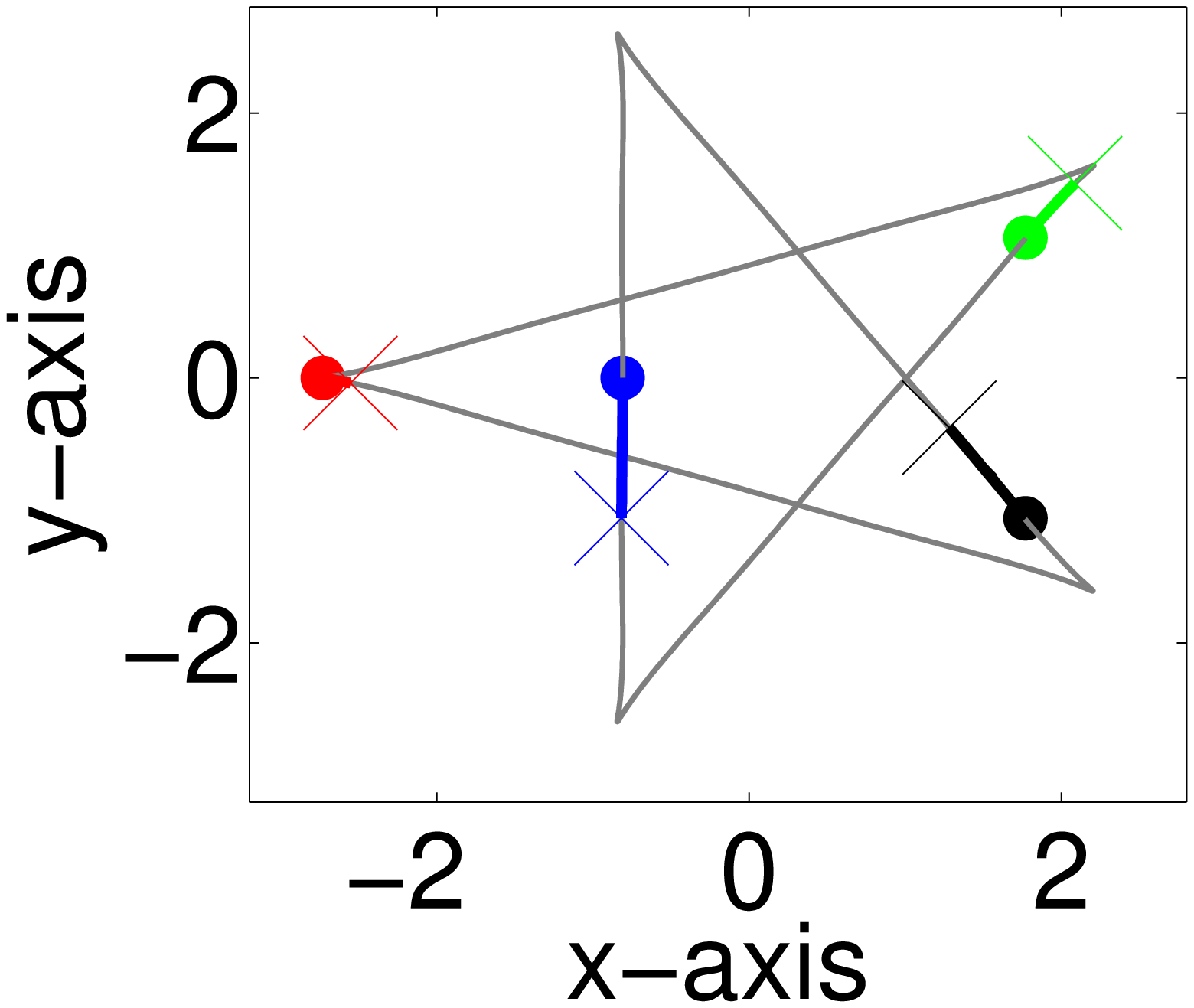}}
\subfigure[Prograde Choreography]{\includegraphics[width=2.38in]{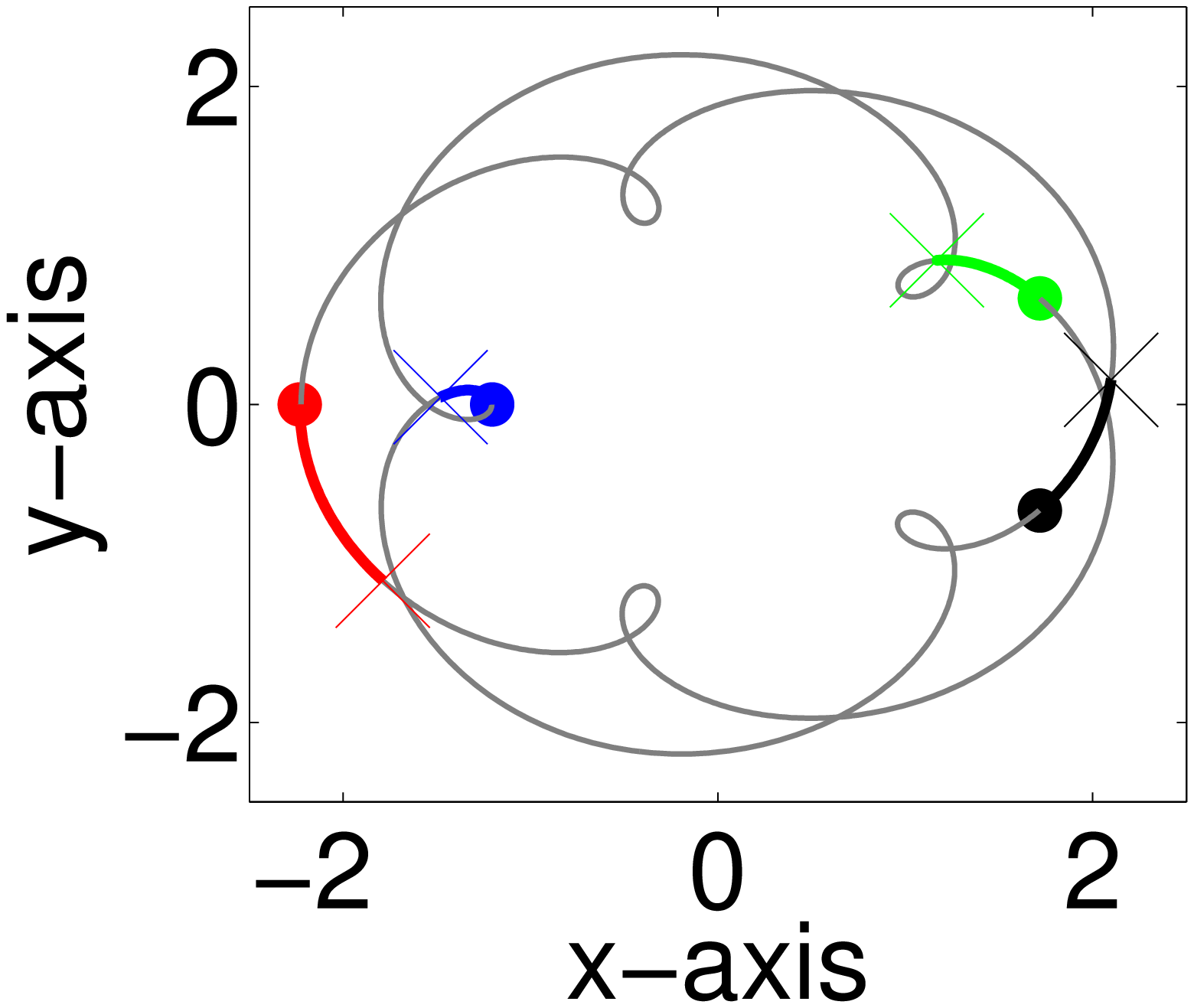}}
 \end{center}
 \caption{\label{fig00} A retrograde choreography is shown in (a) and a prograde choreography is shown in (b). The four dots represent the starting configuration: a double isosceles configuration, while the four crosses represent an isosceles trapezoid configuration. Both orbits can be found by searching local action minimizers connecting a double isosceles configuration and an isosceles trapezoid configuration with fixed symmetry axes. }
  \end{figure} 

Instead of studying the set of the prograde choreographies, we consider a wider set of orbits in the planar equal-mass four-body problem: action minimizers starting at a double isosceles configuration in \eqref{Q_s} and ending at one of the two isosceles trapezoid configurations in \eqref{Q_e1}, while order constraints of the four bodies are introduced on both boundary configurations. Let the masses be $M=[m_1, \, m_2, \, m_3, \, m_4]=[1,\, 1,\, 1,\, 1]$. Let $q_i(t)=(q_{ix}(t), \, q_{iy}(t))$ be the position coordinate of the $i$-th body $\, (i=1,2,3,4)$. The position matrix is denoted by 
$q= \begin{bmatrix}
q_1 \\
q_2 \\
q_3 \\
q_4 
\end{bmatrix}.$ The center of mass is set to be at the origin. That is $q \in \chi$, where
\[ \chi= \left \{ q\in \mathbb{R}_{4\times 2} \, \bigg{|} \, \sum_{i=1}^4 m_iq_i=0 \right \}. \] At $t=0$, the starting configuration $Q_s$ is defined as follows:
\begin{equation}\label{Q_s}
Q_s=\begin{bmatrix}
-a_1-c_1 & 0 \\
-a_1 &   0\\
(2a_1+c_1)/2 & b_1 \\
(2a_1+c_1)/2 & -b_1 
\end{bmatrix},
\end{equation}
where $a_1 \in \mathbb{R}$, $b_1 \geq 0$ and $c_1 \geq 0$. The configuration $Q_s$ is referred to as a double isosceles configuration with order constraints. In fact, as in Fig.~\ref{fig01}, at $t=0$, $q_1$ and $q_2$ are on the $x$-axis with an order constraint $q_{1x}(0) \leq q_{2x}(0)$, while $q_3$ and $q_4$ are on a vertical line with another order constraint $q_{3y}(0) \geq q_{4y}(0)$. Since there is no restriction on $a_1$, the vertical line connecting $q_3$ and $q_4$ could be on the left of body 1, on the right of body 2 or inbetween bodies 1 and 2. In Fig.~\ref{fig01}, a sample picture shows the case when this vertical line is on the right of body 2.
\vspace{0.15in}
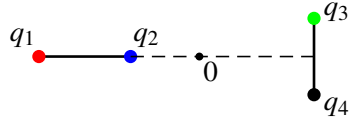
\begin{figure}[!htbp]
\begin{center}
\psset{xunit=1.2in,yunit=0.4in}
\begin{pspicture}(-1, -0.2)(1.1, 1)
\psline[linewidth=1pt](-0.8, 0.3)(-0.4, 0.3)
%\psline[linestyle=dashed, linewidth=0.6pt](-0.4, 0.3)(0.5, 0.3)

\psline[linewidth=1pt](0.4, -0.2)(0.4, 0.8)

\rput(-0.87,  0.55){$q_1$}
\rput(-0.33,  0.55){$q_2$}
\rput(0.50, -0.36){$q_4$}
\rput(0.50, 0.88){$q_3$}
\rput(-0.05, 0.1){$0$}
\psdot[linecolor=red,dotstyle=*, dotsize=5pt](-0.8,  0.3)
\psdot[linecolor=blue,dotstyle=*, dotsize=5pt](-0.4, 0.3)
\psdot[linecolor=black,dotstyle=*, dotsize=5pt](0.4, -0.2)
\psdot[linecolor=green,dotstyle=*, dotsize=5pt](0.4, 0.8)
\psdot[linecolor=black,dotstyle=*, dotsize=3pt](-0.1,  0.3)

\psline[linestyle=dashed, linewidth=0.6pt](-0.4, 0.3)(0.4, 0.3)
\rput(-1.7, 1.2){$Q_s$:}
\end{pspicture}
\end{center}
\caption{\label{fig01} A possible shape of the configuration $Q_s$ is shown, where the red dot represents $q_1$, the blue dot represents $q_2$, the green dot represents $q_3$ and the black dot represents $q_4$. In $Q_s$, $q_1$ and $q_2$ are on the $x$-axis with an order constraint $q_{1x} \leq q_{2x}$, while $q_3$ and $q_4$ are on a vertical line with another order constraint $q_{3y} \geq q_{4y}$. }
\end{figure}

At $t=1$, the other boundary configuration $Q_e$ is defined in two different ways. The definitions of configurations $Q_{e_1}$ and $Q_{e_2}$ are as follows: 
\begin{equation}\label{Q_e1}
Q_{e_1}=\begin{bmatrix}
-b_2 & -a_2 \\
-c_2 &   a_2\\
c_2 & a_2 \\
b_2 & -a_2 
\end{bmatrix}R(\theta),   \qquad Q_{e_2}= \begin{bmatrix}
-a_2 & -b_2 \\
-a_2 &   b_2\\
a_2 & c_2 \\
a_2 & -c_2 
\end{bmatrix}R(\theta),
\end{equation}
where $a_2 \in \mathbb{R}$, \, $b_2 \geq 0$,  \, $c_2 \geq 0$ and $R(\theta)=\begin{bmatrix}
 \cos(\theta)& \sin(\theta)\\
 -\sin(\theta)& \cos(\theta)
 \end{bmatrix}$.
Sample pictures of the two configurations are shown in Fig.~\ref{fig02}, where one can see the geometric meaning of the order constraints in $Q_{e_i} \, (i=1,2)$. The isosceles trapezoid configuration $Q_{e_1} \cdot R(-\theta)$ satisfies the following order constraints: $q_1$ and $q_4$ are on a horizontal line with $q_1$ on the left of $q_4$, while $q_2$ and $q_3$ are on another horizontal line with $q_2$ on the left of $q_3$. The isosceles trapezoid configuration $Q_{e_2} \cdot R(-\theta)$ satisfies that $q_1$ and $q_2$ are on a vertical line with $q_2$ above $q_1$, while $q_3$ and $q_4$ are on another vertical line with $q_3$ above $q_4$. Similar to the graph of $Q_s$, Fig.~\ref{fig02} only shows one possible shape of the two isosceles trapezoids.

 \begin{figure}[!htbp]
\begin{center}
\psset{xunit=1.2in,yunit=0.4in}
\begin{pspicture}(-1, -0.5)(3, 2)
\rput(-0.68, 2){$Q_{e_1}\cdot R(-\theta)$:}
%\rput(-0.05, 0.1){$0$}

\psline[linewidth=1pt](-0.2, 0.8)(0.2, 0.8)
%\psline[linestyle=dashed, linewidth=0.6pt](-0.4, 0.3)(0.5, 0.3)

\psline[linewidth=1pt](-0.5, -0.4)(0.5, -0.4)

\rput(-0.6, -0.5){$q_1$}
\rput(-0.3, 0.93){$q_2$}
\rput(0.6, -0.5){$q_4$}
\rput(0.3, 0.93){$q_3$}
\rput(0.08, 0.03){$0$}

\psline[linestyle=dashed, linewidth=0.6pt](-0.5, -0.4)(-0.2, 0.8)
\psline[linestyle=dashed, linewidth=0.6pt](0.5, -0.4)(0.2, 0.8)
\psline[linestyle=dashed, linewidth=0.6pt](0, -0.8)(0, 1.2)

\psdot[linecolor=red,dotstyle=*, dotsize=5pt](-0.5, -0.4)
\psdot[linecolor=blue,dotstyle=*, dotsize=5pt](-0.2, 0.8)
\psdot[linecolor=black,dotstyle=*, dotsize=5pt](0.5, -0.4)
\psdot[linecolor=green,dotstyle=*, dotsize=5pt](0.2, 0.8)
\psdot[linecolor=black,dotstyle=*, dotsize=3pt](0,  0.2)

\rput(1.3, 2){$Q_{e_2}\cdot R(-\theta)$:}
\psline[linestyle=dashed, linewidth=0.6pt](1.5, 0.3)(2.7, 0.3)
\psline[linewidth=1pt](2.5, -0.6)(2.5, 1.2)
%\psline[linestyle=dashed, linewidth=0.6pt](-0.4, 0.3)(0.5, 0.3)

\psline[linewidth=1pt](1.7,  -0.1)(1.7,  0.7)
\psline[linestyle=dashed, linewidth=0.6pt](1.7,  -0.1)(2.5, -0.6)
\psline[linestyle=dashed, linewidth=0.6pt](1.7, 0.7)(2.5, 1.2)

\rput(1.6,  -0.3){$q_1$}
\rput(1.6,  0.93){$q_2$}
\rput(2.6, -0.8){$q_4$}
\rput(2.6, 1.4){$q_3$}
\rput(2.05,  0.1){$0$}
\psdot[linecolor=red,dotstyle=*, dotsize=5pt](1.7,  -0.1)
\psdot[linecolor=blue,dotstyle=*, dotsize=5pt](1.7, 0.7)
\psdot[linecolor=black,dotstyle=*, dotsize=5pt](2.5, -0.6)
\psdot[linecolor=green,dotstyle=*, dotsize=5pt](2.5, 1.2)

\psdot[linecolor=black,dotstyle=*, dotsize=3pt](2.1, 0.3)

\end{pspicture}

\end{center}
\caption{\label{fig02}One possible shape of the configurations $Q_{e_1}\cdot R(-\theta)$ and $Q_{e_2}\cdot R(-\theta)$ is shown, in which $q_1$ is the red dot, $q_2$ is the blue dot, $q_3$ is the green dot and $q_4$ is the black dot. In $Q_{e_1} \cdot R(-\theta)$, the positions satisfy $q_{4x} \geq q_{1x}$ and $q_{3x} \geq q_{2x}$. While in $Q_{e_2}\cdot R(-\theta)$, the positions satisfy $q_{2y} \geq q_{1y}$ and $q_{3y} \geq q_{4y}$. }
\end{figure}
The topological constraints in $Q_s$, $Q_{e_1}$ and $Q_{e_2}$ have their own advantages. First, they have simple geometric meanings. The topological constraints are only related to the shapes and orders of the four bodies on the boundaries. Their geometric meanings can be easily seen from their matrix definitions. Second, these constraints are very helpful in obtaining a lower bound estimate of actions of all paths with boundary collisions in the admissible set. A detailed explanation can be found in Section \ref{lowerbddcollision}.

We start our analysis by showing the existence of action minimizers under the topological constraints. For each given $\theta$, we denote $Q_S$ and $Q_{E_i} \, (i=1,2)$ to be the boundary configuration sets:
\begin{equation}\label{QS}
 Q_S= \left\{ Q_s \, \bigg| \, a_1 \in \mathbb{R}, \, b_1 \geq 0, \, c_1 \geq 0   \right\},
\end{equation}
\begin{equation}\label{QEi}
 Q_{E_i}= \left\{ Q_{e_i} \, \bigg| \, a_2 \in \mathbb{R}, \, b_2 \geq 0, \, c_2 \geq 0   \right\},
\end{equation}
where $Q_s$ and $Q_{e_i} \, (i=1,2)$ are defined in \eqref{Q_s} and \eqref{Q_e1}. We set $P(Q_S, Q_{E_i}) \, (i=1,2)$ to be the set of paths in $H^1([0,1], \chi)$ which have boundaries in $Q_S$ and $Q_{E_i}$:
\[ P(Q_S, Q_{E_i}):=\left\{q(t) \in H^1([0,1], \chi) \, \bigg| \, q(0) \in  Q_S,  \, q(1) \in Q_{E_i} \right\}, \quad (i=1,2).  \]
For each given $\theta \in (0, \frac{\pi}{4})$, a standard variational argument implies that there exist action minimizers $\mathcal{P}_{Q_i}= \mathcal{P}_{Q_i}([0,1]) \, (i=1,2)$, such that
\begin{equation}\label{existenceofminimizer}
\mathcal{A}(\mathcal{P}_{Q_i}) =\inf_{ q(t) \in P(Q_S, Q_{E_i})} \mathcal{A},
\end{equation}
where $\mathcal{A}=\mathcal{A}(q)=\int_0^1 \left[ K(\dot{q}(t)) + U(q(t)) \right] dt$ is the standard Lagrange action functional. 
The main difficulty is to exclude possible boundary collisions in $\mathcal{P}_{Q_i}= \mathcal{P}_{Q_i}([0,1]) \, (i=1,2)$. A binary decomposition method \cite{CH, CH1, CH2} is introduced to obtain a lower bound of action for all paths with boundary collisions in each case. Then we are left to find appropriate test paths such that the level estimate method works. Our idea is to use a linear approximation of the minimizer $\mathcal{P}_{Q_i}$, which should have an action value close to the infimum $\mathcal{A}(\mathcal{P}_{Q_i}), \, (i=1,2)$. It turns out that for small angle $\theta$, we can eliminate collisions in the minimizer $\mathcal{P}_{Q_i} \, (i=1,2)$. The main results of the paper are as follows.
\begin{theorem}\label{mainthm1}
When $\theta \in (0, 0.0539 \pi]$, the action minimizer $\mathcal{P}_{Q_1}= \mathcal{P}_{Q_1}([0,1])$, which connects the two boundary configuration sets $Q_{S}$ and $Q_{E_1}$, is collision-free and it can be extended to a periodic or quasi-periodic orbit.
\end{theorem} 
The proof of Theorem \ref{mainthm1} can be found in Theorem \ref{Q_sQ_e1orbitext} of Section \ref{extperiodic}. Pictures of the minimizer $\mathcal{P}_{Q_1}= \mathcal{P}_{Q_1}([0,1])$ with $\theta=\frac{\pi}{20}$ and its periodic extension are given in Fig.~\ref{fig03}. 
\begin{figure}[!htbp]
 \begin{center}
\subfigure[Action minimizer $\mathcal{P}_{Q_1}$]{\includegraphics[width=2.38in]{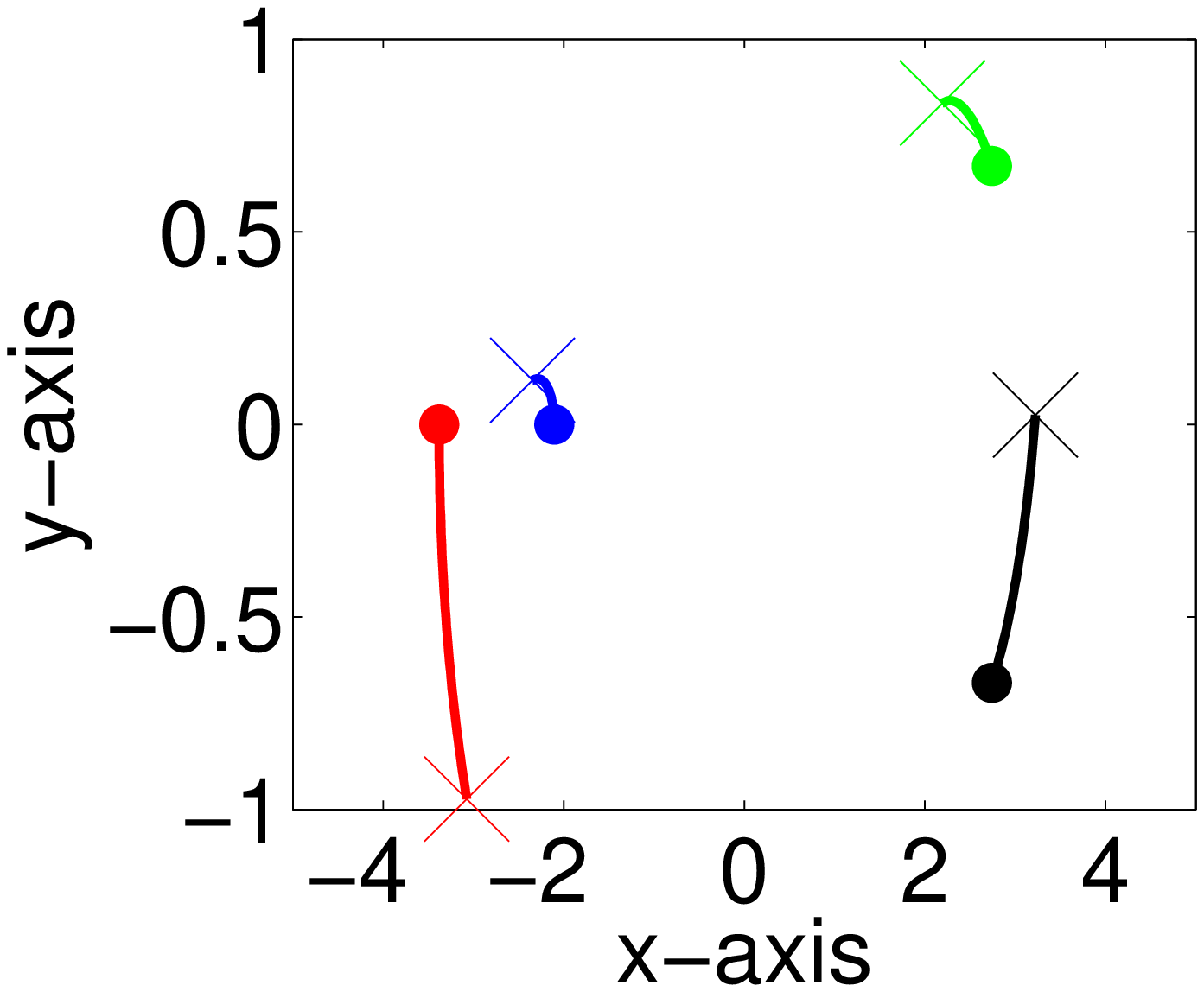}}
\subfigure[ Periodic extension of $\mathcal{P}_{Q_1}$]{\includegraphics[width=2.38in]{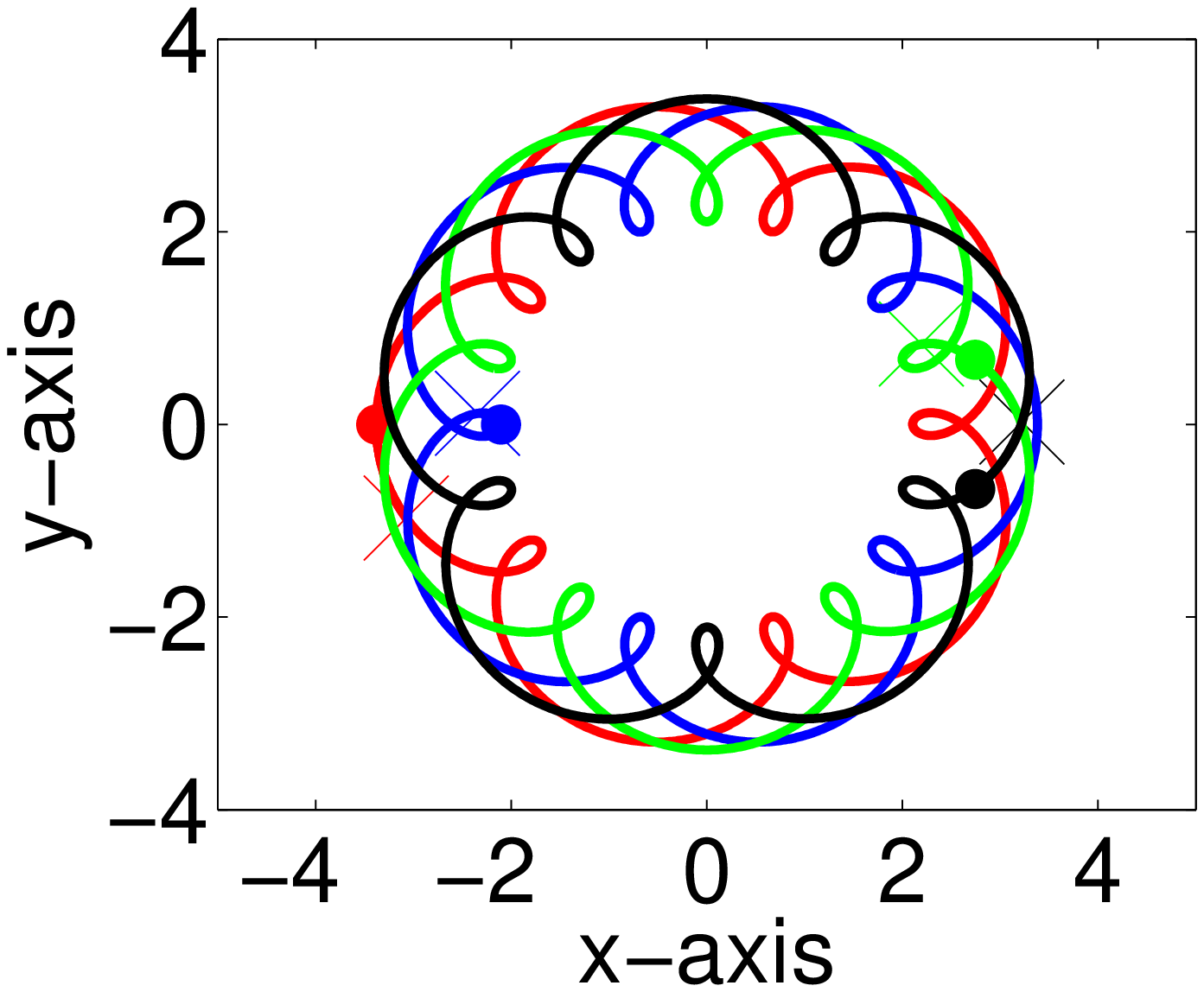}}
 \end{center}
 \caption{\label{fig03} When $\theta=\frac{\pi}{20}$, the trajectories of the action minimizer $\mathcal{P}_{Q_1}= \mathcal{P}_{Q_1}([0,1])$ and its periodic extension are shown. The configuration $Q_s$ is in dots and the configuration $Q_{e_1}$ is in crosses. }
  \end{figure}

\begin{theorem}\label{mainthm2}
When $\theta \in (0, 0.0664 \pi]$, the action minimizer $\mathcal{P}_{Q_2}= \mathcal{P}_{Q_2}([0,1])$, which connects the two boundary configuration sets $Q_{S}$ and $Q_{E_2}$, is collision-free and it can be extended to a periodic or quasi-periodic orbit.
\end{theorem} 
The proof of Theorem \ref{mainthm2} can be found in Theorem \ref{Q_sQ_e2orbitext} of Section \ref{extperiodic}. Pictures of the minimizer $\mathcal{P}_{Q_2}= \mathcal{P}_{Q_2}([0,1])$ with $\theta=\frac{\pi}{20}$ and its periodic extension are presented in Fig.~\ref{fig04}. 
\begin{figure}[!htbp]
 \begin{center}
\subfigure[Action minimizer $\mathcal{P}_{Q_2}$]{\includegraphics[width=2.38in]{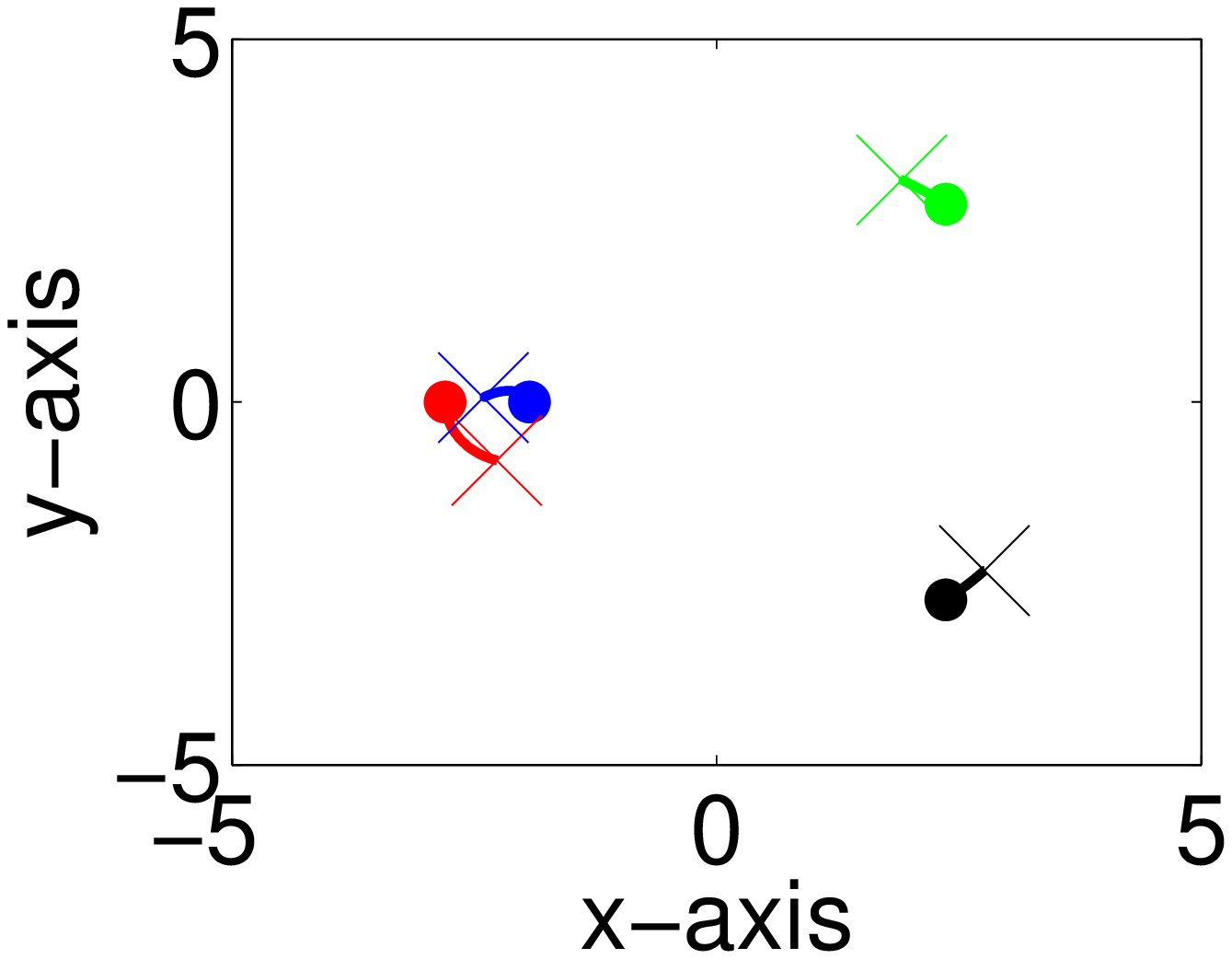}}
\subfigure[Periodic extension of $\mathcal{P}_{Q_2}$]{\includegraphics[width=2.38in]{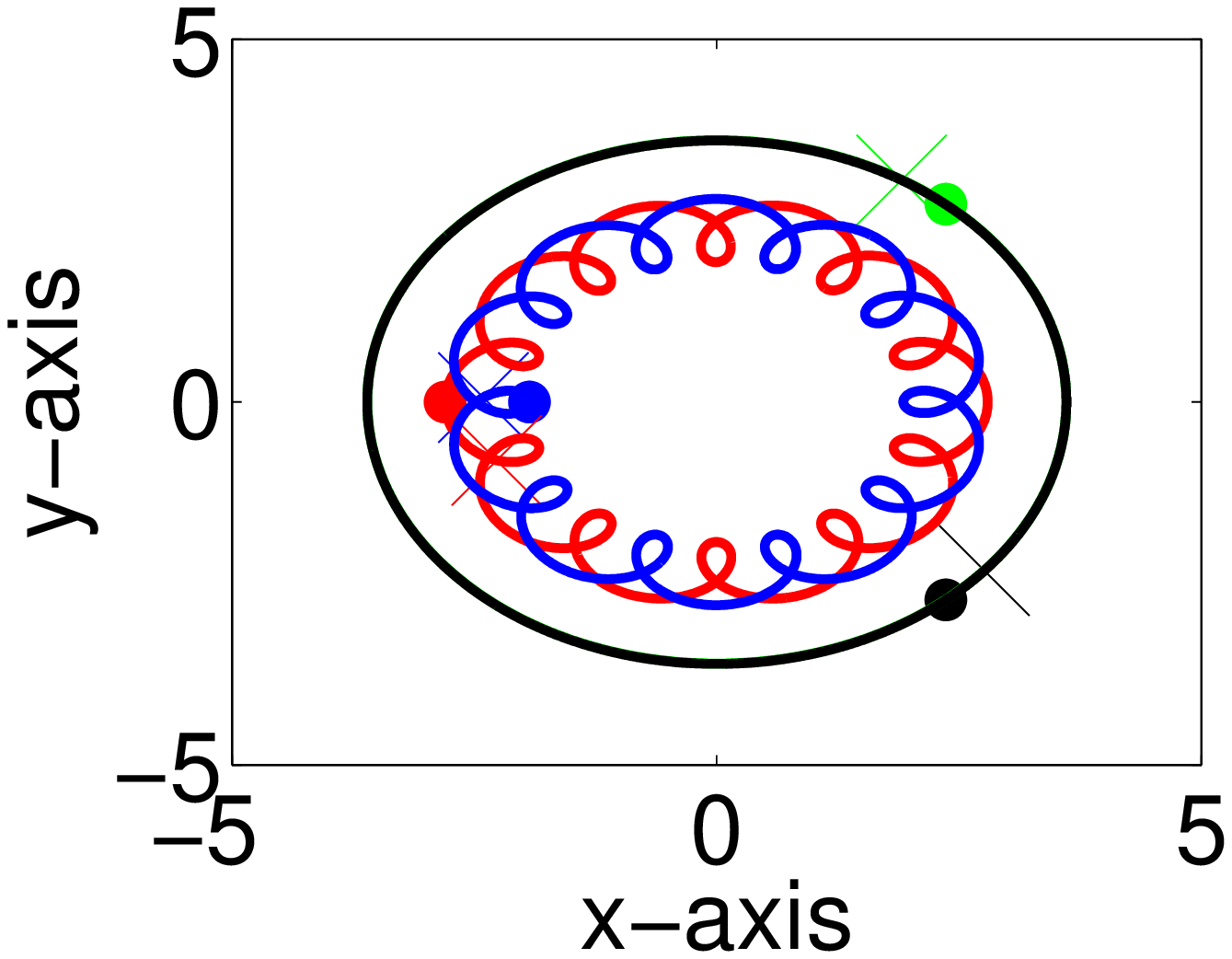}}
 \end{center}
 \caption{\label{fig04} When $\theta=\frac{\pi}{20}$, the trajectories of the action minimizer $\mathcal{P}_{Q_2}= \mathcal{P}_{Q_2}([0,1])$ and its periodic extension are shown. The configuration $Q_s$ is in dots and the configuration $Q_{e_2}$ is in crosses. }
  \end{figure}

 \begin{remark}
The action minimizers $\mathcal{P}_{Q_i} \, (i=1,2)$ are called as local action minimizers in the following sense. If we consider the minimizing problem with all six variables in $\mathbb{R}$: 
 \[\mathcal{A}_{i, \inf} =\inf_{ \{ a_j, \, b_j,  \, c_j\in \mathbb{R} \, (j=1,2)\} } \inf_{\{ q(0)= Q_s, \, q(1)= Q_{e_i}, \, q(t) \in H^1([0, 1], \chi) \} }  \mathcal{A},\]
 the corresponding action minimizer is not $\mathcal{P}_{Q_i} \, (i=1,2)$. In fact, numerical results imply that for given $\theta \in (0, \pi/10]$, $\mathcal{A}_{i, \inf}< \mathcal{A} \left(\mathcal{P}_{Q_i}\right)  \, (i=1,2)$. 
 \end{remark}
% 
%  \begin{remark}
%Different from the braid group constraints in \cite{CH, CH2}, the topological constraints in Q_s, Q_e1 and Q_e2 are mainly about the shapes and orders of the four bodies. One of the advantages is that they have clear geometric meanings such that ChenÕs estimates of Kepler action \cite{CH, CH1, CH2}, can be easily applied here to estimate a lower action bound of collision paths. Actually, similar ideas are applied to show the existence of prograde double-double orbits \cite{Yan2} in the planar equal-mass four-body problem. 
%\end{remark}
  
 The paper is organized as follows. Section \ref{coercivityandext} introduces the existence of action minimizers $\mathcal{P}_{Q_i}= \mathcal{P}_{Q_i}([0,1]) \in H^1([0,1], \chi) \, (i=1,2)$. Section \ref{lowerbddcollision} estimates the lower bound of action of paths with boundary collisions in each case. Section \ref{testpath} defines the test paths with actions smaller than the lower bounds of action in Section \ref{lowerbddcollision}. It implies that both action minimizers $\mathcal{P}_{Q_i}= \mathcal{P}_{Q_i}([0,1]) \, (i=1,2)$ are collision-free. In the end, Section \ref{extperiodic} extends the minimizing paths to periodic or quasi-periodic orbits. 

\section{Coercivity and possible extensions}\label{coercivityandext}
A general coercivity result (Theorem \ref{coercivitygeneral}) of the Lagrangian action functional $\mathcal{A}$ is introduced in this section. Let $\displaystyle \chi= \left \{ q\in \mathbb{R}_{N\times d} \bigg{|} \sum_{i=1}^Nm_iq_i=0 \right \}$. We set
\[ Q_s= \begin{bmatrix}
q_1(a_1, \dots a_k) \\
\dots \\
q_N(a_1, \dots a_k)
\end{bmatrix}, \qquad  Q_e= \begin{bmatrix}
q_1(b_1, \dots b_s) \\
\dots \\
q_N(b_1, \dots b_s)
\end{bmatrix},
\]
where $q_i \in \mathbb{R}^d \, (i=1,2, \dots, N, \, d=1, \, 2, \, \text{or} \, 3)$ are row vectors, and $Q_s, Q_e \in \chi$. Our variational argument is a two-step minimizing procedure. First, we consider a fixed-end boundary value problem, which is also known as the Bolza problem. For given values of $a_1, \dots a_k$ and $b_1, \dots b_s$, the two matrices $Q_s$ and $Q_e$ are fixed. There exists an action minimizer $\mathcal{P}$, such that
\[ \mathcal{A}\left(  \mathcal{P}\right) =   \inf_{ \{ q(t) \in  P(Q_s,Q_e) \} } \mathcal{A} = \inf_{ \{ q(t) \in  P(Q_s,Q_e) \} } \int_0^1 \left[ K(\dot{q}(t)) + U(q(t)) \right] dt,\]
where $\mathrm P(Q_s,Q_e)$ is defined as follows:
\begin{equation*}
 \mathrm P(Q_s,Q_e):=\left  \{q(t)\in H^1([0,1],\chi) \,\big{|}  \,q(0)=Q_s, \, q(1)=Q_e \right \}.
\end{equation*}

If one wants $\mathcal{P}$ to be a part of a periodic solution, the boundaries must be special and they should satisfy certain structural prescribed boundary conditions. Hence, we introduce a second minimizing procedure. Instead of fixing the boundaries, we free several parameters on the boundaries $q(0)=Q_s$ and $q(1)=Q_e$. The Lagrangian action functional is then minimized over these parameters. A general coercivity theorem \cite{Ou2} is introduced here to show the existence of minimizers connecting the two free boundaries. A proof of Theorem \ref{coercivitygeneral} is given in \cite{Ou2}. Actually, similar coercivity results can also be 
found in \cite{CH, FU}.
\begin{theorem}\label{coercivitygeneral}
Let
\begin{equation}\label{Q_sQ_e}
Q_s= \begin{bmatrix}
q_1(a_1, \dots a_k) \\
\dots \\
q_N(a_1, \dots a_k)
\end{bmatrix}, \qquad  Q_e= \begin{bmatrix}
q_1(b_1, \dots b_s) \\
\dots \\
q_N(b_1, \dots b_s)
\end{bmatrix},
\end{equation}
where $Q_s, Q_e \in \chi$, $q_i \in \mathbb{R}^d, (i=1, \dots, N)$ and $a_1, \, \dots, \, a_k, \, b_1, \,  \dots,  \, b_s$ are independent variables. Assume that $Q_s$ is linear with respect to $a_i \, (1 \leq i \leq k)$ and $Q_e$ is linear with respect to $b_j \, (1 \leq j \leq s)$. Let $(a_1, \dots, a_k) \in \mathcal{S}_1, \, (b_1, \dots, b_s) \in  \mathcal{S}_2$, where $\mathcal{S}_1 \in \mathbb{R}^k$ and $\mathcal{S}_2 \in \mathbb{R}^s$ are closed subsets. $\mathcal{S}_1 \cup \mathcal{S}_2=\mathcal{S}.$ The intersection of the following two linear spaces satisfies:
\[ \left\{ Q_s \bigg| (a_1, \dots, a_k)  \in   \mathbb{R}^k \right\}   \cap \left\{ Q_e  \bigg| (b_1, \dots, b_s)  \in   \mathbb{R}^s \right\}=\{\vec{0} \}.\] 
Then there exist a path sequence $\{ \mathcal{P}_{n_l} \}$ and a minimizer $\mathcal{P}_0$ in $H^1([0, 1], \chi)$, such that for each $n_l$,
\[ \mathcal{A}(\mathcal{P}_{n_l})= \inf_{\{ q(0)= Q_s, \, q(1)= Q_e, \, q(t) \in H^1([0, 1], \chi), \, a_i=a_{{i}_{n_l}}, \, b_j=b_{{j}_{n_l}}, (i=1, \dots, k; j=1, \dots, s) \} }  \mathcal{A},\]
\begin{equation*}
\begin{split}
\mathcal{A}(\mathcal{P}_0) &= \inf_{ \{  (a_1,\dots, a_k, \, b_1, \dots, b_s) \in  \mathcal{S} \} } \inf_{\{ q(0)= Q_s, \, q(1)= Q_e, \, q(t) \in H^1([0, 1], \chi) \} }  \mathcal{A} \\
&= \inf_{\{ q(0)= Q_s, \, q(1)= Q_e, \, q(t) \in H^1([0, 1], \chi), \, a_i=a_{{i}_{0}}, \, b_j=b_{{j}_{0}}, (i=1, \dots, k; j=1, \dots, s) \} }  \mathcal{A} .
\end{split}
\end{equation*}
 For $t \in [0, 1]$,  $\displaystyle  \mathcal{P}_{n_l}(t)$ converges to $\mathcal{P}_0 (t)$ uniformly. In particular,
\[ \lim_{n_l \to \infty} a_{{i}_{n_l}} = a_{{i}_{0}}, \qquad  \lim_{n_l \to \infty} b_{{j}_{n_l}} = b_{{j}_{0}}, \quad i=1, \dots, k; \, \,  j=1, \dots, s.\]
\end{theorem}

As its applications, we consider the following free boundary value problems. The starting configuration $Q_s$ is defined as follows, which is a double-isosceles with order constraints. A sample picture of it can be found in Fig.~\ref{fig01}. 
\begin{equation}\label{Q_s01}
Q_s=\begin{bmatrix}
-a_1-c_1 & 0 \\
-a_1 &   0\\
(2a_1+c_1)/2 & b_1 \\
(2a_1+c_1)/2 & -b_1 
\end{bmatrix}, \quad a_1 \in \mathbb{R}, \,  \, b_1\geq 0, \, \, c_1 \geq 0. 
\end{equation}
The two ending configurations $Q_{e_1}$ and $Q_{e_2}$ are defined by
\begin{equation}\label{Q_e01}
Q_{e_1}=\begin{bmatrix}
-b_2 & -a_2 \\
-c_2 &   a_2\\
c_2 & a_2 \\
b_2 & -a_2 
\end{bmatrix}R(\theta),  \qquad Q_{e_2}= \begin{bmatrix}
-a_2 & -b_2 \\
-a_2 &   b_2\\
a_2 & c_2 \\
a_2 & -c_2 
\end{bmatrix}R(\theta), 
\end{equation}
where $a_2 \in \mathbb{R}$, \, $b_2\geq 0$, \, $c_2 \geq 0$ and $R(\theta)=\begin{bmatrix}
 \cos(\theta)& \sin(\theta)\\
 -\sin(\theta)& \cos(\theta)
 \end{bmatrix}$. Sample pictures of them are given in Fig.~\ref{fig02}. For $\theta \in (0, \pi/4)$, it is clear that 
 \[\{ Q_s \, \big| \, a_1, b_1, c_1 \in \mathbb{R}\} \cap \{Q_{e_i} \, \big| \, a_2, b_2, c_2 \in \mathbb{R}\}= \{0\},  \, \,  (i=1,2).  \]
By Theorem \ref{coercivitygeneral}, there exist two action minimizers $\mathcal{P}_{Q_1}$ and $\mathcal{P}_{Q_2}$, such that
\begin{equation}
\mathcal{A}(\mathcal{P}_{Q_i}) = \inf_{ \{ a_j\in \mathbb{R}, \, b_j \geq 0, \, c_j \geq 0 \, (j=1,2)\} } \inf_{\{ q(0)= Q_s, \, q(1)= Q_{e_i}, \, q(t) \in H^1([0, 1], \chi) \} }  \mathcal{A}, \, (i=1,2).
\end{equation}
For each given $\theta$, we denote $Q_S$ and $Q_{E_i} \, (i=1,2)$ to be the boundary configuration sets:
\begin{equation}\label{2QS}
 Q_S= \left\{ Q_s \, \bigg| \, a_1 \in \mathbb{R}, \, b_1 \geq 0, \, c_1 \geq 0   \right\},
\end{equation}
\begin{equation}\label{2QEi}
 Q_{E_i}= \left\{ Q_{e_i} \, \bigg| \, a_2 \in \mathbb{R}, \, b_2 \geq 0, \, c_2 \geq 0   \right\},
\end{equation}
Where $Q_s$ is defined by \eqref{Q_s01}, and $Q_{e_i} \, (i=1,2)$ are defined by \eqref{Q_e01}. By the celebrated results of Marchal \cite{Mar} and Chenciner \cite{CA}, it is known that $\mathcal{P}_{Q_1}$ and $\mathcal{P}_{Q_2}$ are free of collision when $t \in (0,1)$. However, the two action minimizers could have boundary collisions in the sets $Q_{S}$ or $Q_{E_i} \, (i=1,2)$. If $\mathcal{P}_{Q_1}$ and $\mathcal{P}_{Q_2}$ are free of collision in the boundary configuration sets $Q_S$ or $Q_{E_i} \, (i=1,2)$, we show that they can be extended.

\begin{lemma}\label{extensionformula1}
If $\mathcal{P}_{Q_i} \equiv\mathcal{P}_{Q_i}([0,1])  \, (i=1,2)$ has no collision at $t=0$, then the path  $\mathcal{P}_{Q_i}([0,1]) \, (i=1,2)$ can be smoothly extended to $\mathcal{P}_{Q_i}([-1,1])$. 
\end{lemma}
\begin{proof}
We only show the extension for $\mathcal{P}_{Q_1}$ here. The extension for $\mathcal{P}_{Q_2}$ follows by the same argument. The proof is based on the first variation formulas. 
Let $Q_s=\begin{bmatrix}
-a_{11}-c_{11} & 0 \\
-a_{11} &   0\\
(2a_{11}+c_{11})/2 & b_{11} \\
(2a_{11}+c_{11})/2 & -b_{11} 
\end{bmatrix}$  be the position matrix at $t=0$ in the minimizing path $\mathcal{P}_{Q_1}([0,1])$. Let
\[ q(t) =\begin{bmatrix}
q_{1}(t) \\
q_{2}(t) \\
q_{3}(t) \\
q_{4}(t)  
\end{bmatrix} =\begin{bmatrix}
q_{1x}(t) & q_{1y}(t) \\
q_{2x}(t) & q_{2y}(t) \\
q_{3x}(t) & q_{3y}(t) \\
q_{4x}(t) & q_{4y}(t) 
\end{bmatrix}\]
be the position matrix path of $\mathcal{P}_{Q_1}([0,1])$. By assumption, $q(0)=Q_s$ has no collision. It implies that $b_{11} >0$ and $c_{11} >0$. By applying the first variation formulas to $a_{11}$, $b_{11}$ and $c_{11}$, we have
\begin{eqnarray}\label{firstvariation}
-\dot{q}_{1x}(0)-\dot{q}_{2x}(0)+\dot{q}_{3x}(0)+\dot{q}_{4x}(0)&=&0, \nonumber \\
\dot{q}_{3y}(0)- \dot{q}_{4y}(0)&=&0, \\
-\dot{q}_{1x}(0)+\frac{1}{2}\dot{q}_{3x}(0)+\frac{1}{2}\dot{q}_{4x}(0)&=&0.\nonumber 
\end{eqnarray}
Note that $q(t) \in \chi$, it follows that 
\begin{equation}\label{linearmomentum}
  \sum_{i=1}^4 \dot{q}_{ix}(0)=\sum_{i=1}^4 \dot{q}_{iy}(0)=0. 
  \end{equation}
 Then  identities \eqref{firstvariation} and \eqref{linearmomentum} imply that
\begin{equation}\label{velocityrelation1}
\dot{q}_{1x}(0)=\dot{q}_{2x}(0)=0, \quad \dot{q}_{3x}(0)=-\dot{q}_{4x}(0), \quad \dot{q}_{3y}(0)=\dot{q}_{4y}(0).
\end{equation}
The definition of a path $\mathcal{P}_{Q_1}([-1,0])$ is given as follows.
\begin{eqnarray}
q_1(t)&=&(q_{1x}(-t), \, \, -q_{1y}(-t)), \quad t \in [-1, 0], \nonumber \\
q_2(t)&=& (q_{2x}(-t), \, \, -q_{2y}(-t)), \quad t \in [-1, 0], \nonumber \\
q_3(t)&=& (q_{4x}(-t), \, \, -q_{4y}(-t)), \quad t \in [-1, 0], \\
q_4(t)&=& (q_{3x}(-t), \, \, -q_{3y}(-t)), \quad t \in [-1, 0]. \nonumber 
\end{eqnarray}
By \eqref{velocityrelation1}, it is easy to check that $\mathcal{P}_{Q_1}([-1,0])$ and $\mathcal{P}_{Q_1}([0,1])$ is $C^1$ smooth at $t=0$. Note that for $t \in [0, 1)$, $\mathcal{P}_{Q_1}$ satisfies the Newtonian equations. Hence, by the uniqueness of solution of ODE system, $\mathcal{P}_{Q_1}([-1,0])$ and $\mathcal{P}_{Q_1}([0,1])$ is smoothly connected at $t=0$. The proof is complete.
\end{proof}

\begin{corollary}\label{extensionformula2}
If $\mathcal{P}_{Q_i}\equiv \mathcal{P}_{Q_i}([0,1])$ has no collision at $t=1$, then the path  $\mathcal{P}_{Q_i}([0,1])$ can be smoothly extended to  $\mathcal{P}_{Q_i}([0,2]), \, (i=1,2).$
\end{corollary}
\begin{proof}
Similar to the proof in Lemma \ref{extensionformula1}, we only show the extension for $\mathcal{P}_{Q_1}([0,1])$. The formulas of extension for $\mathcal{P}_{Q_2}([0,1])$ will be given at the end of the proof. \\
 In $\mathcal{P}_{Q_1}([0,1])$, we assume the position matrix to be 
\[ q(t) =\begin{bmatrix}
q_{1}(t) \\
q_{2}(t) \\
q_{3}(t) \\
q_{4}(t)  
\end{bmatrix} =\begin{bmatrix}
q_{1x}(t) & q_{1y}(t) \\
q_{2x}(t) & q_{2y}(t) \\
q_{3x}(t) & q_{3y}(t) \\
q_{4x}(t) & q_{4y}(t) 
\end{bmatrix}. \]
Let the matrix $q(1)=Q_{e_1}$ in $\mathcal{P}_{Q_1}([0,1])$ be $\begin{bmatrix}
-b_{21} & -a_{21} \\
-c_{21} &   a_{21}\\
c_{21} & a_{21} \\
b_{21} & -a_{21} 
\end{bmatrix}R(\theta)$, where $\theta \in (0, \pi/4)$. By assumption, $q(1)= Q_{e_1}$ has no collision. It implies that $b_{21}>0$ and $c_{21}>0$. 
By applying the first variation formulas to the boundary at $t=1$ and noting that $ \displaystyle \sum_{i=1}^4 \dot{q}_{ix}(1)=\sum_{i=1}^4 \dot{q}_{iy}(1)=0$, we have
\begin{eqnarray}\label{firstvariation2}
\dot{q}_{1x}(1) \sin \theta - \dot{q}_{1y}(1) \cos \theta &=& -\dot{q}_{4x}(1) \sin \theta + \dot{q}_{4y}(1) \cos \theta, \nonumber \\
\dot{q}_{2x}(1) \sin \theta - \dot{q}_{2y}(1) \cos \theta &=& -\dot{q}_{3x}(1) \sin \theta + \dot{q}_{3y}(1) \cos \theta,  \nonumber\\
\dot{q}_{1x}(1) \cos \theta + \dot{q}_{1y}(1) \sin \theta &=& \dot{q}_{4x}(1) \cos \theta + \dot{q}_{4y}(1) \sin \theta, \\
\dot{q}_{2x}(1) \cos \theta + \dot{q}_{2y}(1) \sin \theta &=& \dot{q}_{3x}(1) \cos \theta + \dot{q}_{3y}(1) \sin \theta.\nonumber 
\end{eqnarray}
The matrix form of identities \eqref{firstvariation2} is
\begin{equation}\label{matrixeqnpq1}
\dot{q}_1(1)= \dot{q}_4(1)B R(2 \theta), \qquad  \dot{q}_2(1)= \dot{q}_3(1)B R(2 \theta),
\end{equation}
where $B=\begin{bmatrix}
1 & 0\\
0  &  -1
\end{bmatrix}$ and $R(2\theta)=\begin{bmatrix}
\cos(2\theta) & \sin(2\theta)\\
-\sin(2\theta) & \cos(2\theta)
\end{bmatrix}$. 
Hence, we can define the extended path $\mathcal{P}_{Q_1}([1, 2])$ as follows
\begin{equation}\label{extensionpq1}
 \begin{aligned}
 q_1(t)&= -q_4(2-t) B\, R(2\theta), \quad  q_2(t)=- q_3 (2-t) B\, R(2\theta),  \\
 q_3 (t)&= -q_2 (2-t) B\, R(2\theta), \quad  q_4 (t)=- q_1 (2-t) B\, R(2\theta),\\
\end{aligned} \qquad t \in [1,2].
\end{equation}
It is easy to check that the path $\mathcal{P}_{Q_1}([1, 2])$ is smoothly connected with the path $\mathcal{P}_{Q_1}([0, 1])$ at $t=1$. 

Similarly, for the path $\mathcal{P}_{Q_2}([0, 1])$, the smoothly extened path $\mathcal{P}_{Q_2}([1, 2])$ is defined as follows
\begin{equation}\label{extensionpq2}
 \begin{aligned}
 \tilde{q}_1(t)&=  \tilde{q}_2(2-t) B\, R(2\theta), \quad   \tilde{q}_2(t)= \tilde{q}_1 (2-t) B\, R(2\theta),  \\
  \tilde{q}_3 (t)&=  \tilde{q}_4 (2-t) B\, R(2\theta), \quad   \tilde{q}_4 (t)=  \tilde{q}_3 (2-t) B\, R(2\theta),\\
\end{aligned} \qquad t \in [1,2],
\end{equation}
where $\tilde{q}(t) =\begin{bmatrix}
\tilde{q}_{1}(t) \\
\tilde{q}_{2}(t) \\
\tilde{q}_{3}(t) \\
\tilde{q}_{4}(t)  
\end{bmatrix} =\begin{bmatrix}
\tilde{q}_{1x}(t) & \tilde{q}_{1y}(t) \\
\tilde{q}_{2x}(t) & \tilde{q}_{2y}(t) \\
\tilde{q}_{3x}(t) & \tilde{q}_{3y}(t) \\
\tilde{q}_{4x}(t) & \tilde{q}_{4y}(t) 
\end{bmatrix}$ is the position matrix path of $\mathcal{P}_{Q_2}([0, 1])$. The proof is complete.
\end{proof}

\section{Lower bounds of action of paths with boundary collisions}\label{lowerbddcollision}
In this section, we obtain lower bounds of actions when $\mathcal{P}_{Q_i}([0, 1]) \, (i=1,2)$ have boundary collisions. We first introduce a theorem of Chen \cite{CH, CH2, Gor} which estimates the Keplerian action functional.  

Given $\theta \in (0, \pi]$, $T>0$, consider the following path spaces:
\begin{align*}
\Gamma_{T, \theta} := &\left\{\vec{r} \in H^1([0,T], \mathbb{R}^2):  \, \, <\vec{r}(0), \, \vec{r}(T)>= |\vec{r}(0)||\vec{r}(T)| \cos \theta  \right\},  \\
\Gamma^{*}_{T, \theta}:=& \left\{ \vec{r} \in \Gamma_{T, \theta}: \,   |\vec{r}(t)|=0 \, \,  \, \text{for some}  \, \, t \in [0,T]\right\}.  
\end{align*}
The symbol $<\cdot, \cdot>$ stands for the standard scalar product in $\mathbb{R}^2$ and $|\cdot |$ represents the standard norm in $\mathbb{R}^2$. Define the Keplerian action functional $I_{\mu, \alpha, T}: H^1([0,T], \mathbb{R}^2) \to \mathbb{R}\cup \{+\infty\}$ by
\[ I_{\mu, \alpha, T}(\vec{r}) := \int_0^{T} \frac{\mu}{2} |\dot{\vec{r}}|^2 + \frac{\alpha}{|\vec{r}|} \, dt. \]
\begin{theorem}\label{keplerestimate}
Let $\theta \in (0, \pi]$, $T>0$, $\mu>0$, $\alpha>0$ be constants. Then
\begin{equation}\label{keplernocollision}
\inf_{\vec{r} \in \Gamma_{T, \theta}}  I_{\mu, \alpha, T}(\vec{r}) = \frac{3}{2}\left(  \mu \alpha^2 \theta^2 T\right)^{\frac{1}{3}},
\end{equation}
\begin{equation}\label{keplercollision}
\inf_{\vec{r} \in \Gamma^{*}_{T, \theta}}  I_{\mu, \alpha, T}(\vec{r}) = \frac{3}{2}\left(  \mu \alpha^2 \pi^2 T\right)^{\frac{1}{3}}.
\end{equation}
\end{theorem}
Recall that the masses of the four bodies are $M=[1 , \, 1, \, 1, \, 1]$ and the position matrix path is denoted by $q(t)= \begin{bmatrix}
q_{1}(t) \\
q_{2}(t) \\
q_{3}(t) \\
q_{4}(t)  
\end{bmatrix}$. The action functional $\mathcal{A}: H^1([0,1], \chi) \to \mathbb{R}\cup \{+\infty\}$ is defined by
\[ \mathcal{A}= \mathcal{A}(q):= \int_0^1 K(\dot{q})+ U(q) \, dt,\]
where 
\[  K(\dot{q}) = \frac{1}{2} \left(|\dot{q}_1|^2+|\dot{q}_2|^2+|\dot{q}_3|^2+|\dot{q}_4|^2 \right) \]
is the kinetic energy and 
\[ U(q)= \frac{1}{|q_1-q_2|}+ \frac{1}{|q_1-q_3|}+\frac{1}{|q_1-q_4|}+\frac{1}{|q_2-q_3|}+\frac{1}{|q_2-q_4|}+\frac{1}{|q_3-q_4|}\]
is the potential energy. The action functional $\mathcal{A}(q)$ can be written as 
\begin{equation}\label{actioninkeplerform}
\mathcal{A}(q)= \frac{1}{4} \sum_{1\leq i<j \leq 4} \int_0^1 \frac{1}{2} |\dot{q}_i-\dot{q}_j|^2+ \frac{4}{|q_i-q_j|} \, dt.
\end{equation}
Let $\mathcal{A}_{ij}:= \int_0^1 \frac{1}{2} |\dot{q}_i-\dot{q}_j|^2+ \frac{4}{|q_i-q_j|} \, dt$. It follows that 
\begin{equation}\label{formulaofaction}
\mathcal{A}= \mathcal{A}(q)= \frac{1}{4} \sum_{1\leq i<j \leq 4} \mathcal{A}_{ij}. 
\end{equation}
Actually, if $\mathcal{P}_{Q_i}([0,1])$ has a total collision at $t=0$ or $t=1$,  by Theorem \ref{keplernocollision}, we have
\begin{equation}\label{estimatetotalcollision}
\mathcal{A} \geq \frac{6}{4} \cdot\frac{3}{2}\left(  16 \pi^2 \right)^{\frac{1}{3}}\geq 12.16.
\end{equation}
 In the following two subsections, we discuss the lower bounds of action for all paths connecting $Q_S$ and $Q_{E_i} \, (i=1,2)$ with boundary collisions. These estimates are mainly based on Theorem \ref{keplerestimate} \cite{CH, CH2} and the topological constraints in $Q_S$ and $Q_{E_i} \, (i=1,2)$.

\subsection{Lower bound of action of $\mathcal{P}_{Q_1}$ with boundary collisions}
In the minimizing path $\mathcal{P}_{Q_1}=\mathcal{P}_{Q_1}([0,1])$, the two boundary configurations are 
\begin{equation}\label{boundarypq1}
q(0)=\begin{bmatrix}
-a_{11}-c_{11} & 0 \\
-a_{11} &   0\\
(2a_{11}+c_{11})/2 & b_{11} \\
(2a_{11}+c_{11})/2 & -b_{11} 
\end{bmatrix}, \, \, \, q(1)=\begin{bmatrix}
-b_{21} & -a_{21} \\
-c_{21} &   a_{21}\\
c_{21} & a_{21} \\
b_{21} & -a_{21} 
\end{bmatrix}R(\theta),
\end{equation}
where $q(t)= \begin{bmatrix}
q_{1}(t) \\
q_{2}(t) \\
q_{3}(t) \\
q_{4}(t)  
\end{bmatrix}$ is the position matrix path of $\mathcal{P}_{Q_1}([0,1])$, $a_{11}, a_{21} \in \mathbb{R}$ and the other four constants $b_{11}, b_{21}, c_{11}, c_{21} \geq 0$.  In this subsection, we show that
\begin{lemma}\label{lowerbdd1}
Let $\theta \in (0, \pi/10]$. If the minimizing path $\mathcal{P}_{Q_1}([0,1])$ has a boundary collision, then its action $\mathcal{A}$ satisfies
\begin{equation}\label{estlowerbddpq1}
\mathcal{A}=\mathcal{A}(q) \geq \frac{3}{8} 16^{\frac{1}{3}} \left[  \left( 2 \pi^2\right)^{\frac{1}{3}} + 2 (2 \theta)^{\frac{2}{3}} \right].
\end{equation}
\end{lemma}
\begin{proof}
According to the results of Marchal \cite{Mar} and Chenciner \cite{CA}, the only possible collisions in $\mathcal{P}_{Q_1}([0,1])$ are on the boundaries: $q(0)$ and $q(1)$. By the definitions of the boundary configuration sets $Q_S$ and $Q_{E_1}$ in \eqref{2QS} and \eqref{2QEi}, the possible boundary collisions could be the following:
\begin{enumerate}
\item At $t=0$: \\
$q_1=q_2$; \, $q_3=q_4$; \, $q_2=q_3=q_4$; \, $q_1=q_3=q_4$; \ $q_1=q_2=q_3=q_4$;
\item At $t=1$: \\
$q_2=q_3$; \, $q_1=q_4$; \, $q_1=q_2=q_3=q_4$.
\end{enumerate}
We will discuss them case by case under the assumption: $\theta \in (0, \pi/10]$. The analysis only considers possible binary collisions since triple collisions and total collisions can be treated as several binary collisions happening simultaneously in the action functional $\mathcal{A}_{ij}\, (1\leq i<j\leq 4)$. Five different cases are studied as follows.
 \vspace{0.2in} 
 
\textbf{Case 1:} $q_2$ and $q_3$ collide at $t=1$ $\left(q_2(1)=q_3(1) \right)$ and there is no collision at $t=0$ in $\mathcal{P}_{Q_1}([0,1])$.\\

By Lemma \ref{extensionformula1}, the path $\mathcal{P}_{Q_1}([0,1])$ can be extended to  $\mathcal{P}_{Q_1}([-1, 1])$. The binary collision $q_2(1)=q_3(1)$ implies another binary collision at $t=-1$:  $q_2(-1)= q_4(-1)$. Then by Theorem \ref{keplerestimate}, 
\begin{align*}
 & 2 \left(\mathcal{A}_{23}+\mathcal{A}_{24} \right) \\
 =& \int_{-1}^{1} \frac{1}{2} |\dot{q}_2-\dot{q}_3|^2+ \frac{4}{|q_2-q_3|} \, dt +\int_{-1}^{1} \frac{1}{2} |\dot{q}_2-\dot{q}_4|^2+ \frac{4}{|q_2-q_4|} \, dt\\
 \geq & 3\left( 32 \pi^2\right)^{\frac{1}{3}}. 
 \end{align*}
It follows that 
\begin{equation}\label{a23a24t1est1}
\mathcal{A}_{23}+\mathcal{A}_{24}\geq \frac{3}{2} \left( 32 \pi^2\right)^{\frac{1}{3}}.
\end{equation}
Let $\beta_1\in [0, \pi]$ be the polar angle of the vector $\overrightarrow{q_1q_3}$ at $t=0$. Let $\beta_2 \in [0, \pi]$ be the polar angle of the vector $\overrightarrow{q_2q_3}$ at $t=0$. Let $\alpha \in [-\pi/2, \pi/2]$ be the angle from the vector $\overrightarrow{q_1q_4}$ at $t=1$ rotating counterclockwisely to the vector $\overrightarrow{q_1q_2}$ at $t=1$. Note that $\theta \in (0, \pi/10]$ and $q_2(1)=q_3(1)$. When  $\frac{\pi}{2}-\alpha + \theta \leq \pi$ , $\theta+\beta_1  \leq \pi$ and $\beta_1-\alpha - \theta \leq \pi$, the following inequalities hold:
\[  \mathcal{A}_{12} \geq \frac{3}{2}16^{\frac{1}{3}} \left( \alpha + \theta \right)^{\frac{2}{3}} ,  \qquad  \,\,    \mathcal{A}_{34} \geq \frac{3}{2}16^{\frac{1}{3}} \left( \frac{\pi}{2}-\alpha + \theta \right)^{\frac{2}{3}}, \]   
\[   \mathcal{A}_{13} \geq \frac{3}{2}16^{\frac{1}{3}} \left(  \beta_1-\alpha - \theta \right)^{\frac{2}{3}} ,   \qquad  \,\, \mathcal{A}_{14} \geq \frac{3}{2}16^{\frac{1}{3}} \left( \theta+\beta_1 \right)^{\frac{2}{3}}. \]
The function $f(x)=x^{\frac{2}{3}}$ is concave. It implies that for any $a, b \in \mathbb{R}$, we have
\[a^{\frac{2}{3}}+ b^{\frac{2}{3}} \geq (|a|+|b|)^{\frac{2}{3}} \geq \max\left\{(a+b)^{\frac{2}{3}}, (a-b)^{\frac{2}{3}} \right\}. \]
Hence, 
\begin{equation}\label{a13a14t1est1}
 \mathcal{A}_{13}+\mathcal{A}_{14} \geq \frac{3}{2}16^{\frac{1}{3}}  \left(\alpha + 2 \theta \right)^{\frac{2}{3}}, 
\end{equation}
\begin{equation}\label{a12a13a14t1est1}
 \mathcal{A}_{12}+\mathcal{A}_{13}+\mathcal{A}_{14}  \geq  \frac{3}{2}16^{\frac{1}{3}}  \theta^{\frac{2}{3}}, 
\end{equation}
\begin{equation}\label{a13a14a34t1est1}
 \mathcal{A}_{34}+\mathcal{A}_{13}+\mathcal{A}_{14}  \geq \frac{3}{2}16^{\frac{1}{3}}  \left( \frac{\pi}{2} + 3 \theta \right)^{\frac{2}{3}},
\end{equation}
\begin{equation}\label{a13a14t1est1}
 \mathcal{A}_{12}+\mathcal{A}_{34} \geq \frac{3}{2}16^{\frac{1}{3}}  \left(\frac{\pi}{2} + 2 \theta\right)^{\frac{2}{3}}. 
\end{equation}
It follows that when  $\frac{\pi}{2}-\alpha + \theta \leq \pi$ , $\theta+\beta_1  \leq \pi$ and $\beta_1-\alpha - \theta \leq \pi$,
\begin{equation}\label{a12a13a14a34t1est1}
\mathcal{A}_{12}+\mathcal{A}_{34}+\mathcal{A}_{13}+\mathcal{A}_{14} \geq  \frac{3}{4}16^{\frac{1}{3}} \left[  \left(\frac{\pi}{2} + 2 \theta\right)^{\frac{2}{3}}+  \left( \frac{\pi}{2} + 3 \theta \right)^{\frac{2}{3}}  + \theta^{\frac{2}{3}}  \right].
\end{equation}
If  $\frac{\pi}{2}-\alpha + \theta \geq \pi$, note that $\theta \in (0, \pi/10]$, it follows that 
\begin{equation}\label{a34estimate}
\mathcal{A}_{34} \geq  \frac{3}{2}16^{\frac{1}{3}}  \left(  \frac{3}{2} \pi +\alpha- \theta \right)^{\frac{2}{3}} \geq\frac{3}{2}16^{\frac{1}{3}}  \left(  \frac{9}{10} \pi \right)^{\frac{2}{3}}.
\end{equation}
Similarly, if $\theta+\beta_1  \geq \pi$, 
\begin{equation}\label{a14estimate}
\mathcal{A}_{14} \geq\frac{3}{2}16^{\frac{1}{3}}  \left(  \frac{9}{10} \pi \right)^{\frac{2}{3}}.
\end{equation}
If $\beta_1-\alpha - \theta \geq \pi$ and $\theta+\beta_1  < \pi$, then 
\[\mathcal{A}_{13} \geq \frac{3}{2}16^{\frac{1}{3}} \left( 2\pi+ \alpha + \theta- \beta_1 \right)^{\frac{2}{3}}.\]
In this case,
\begin{equation}\label{a1314estimate}
\mathcal{A}_{13}+ \mathcal{A}_{14} \geq  \frac{3}{2}16^{\frac{1}{3}}  \left( 2\pi+ \alpha + 2\theta \right)^{\frac{2}{3}} \geq \frac{3}{2}16^{\frac{1}{3}}  \left(  \alpha + 2\theta \right)^{\frac{2}{3}}. 
\end{equation}
If $\beta_1-\alpha - \theta \geq \pi$ and $\theta+\beta_1  \geq \pi$, then
\[\mathcal{A}_{13} \geq \frac{3}{2}16^{\frac{1}{3}} \left( 2\pi+ \alpha + \theta- \beta_1 \right)^{\frac{2}{3}}, \quad \mathcal{A}_{14} \geq \frac{3}{2}16^{\frac{1}{3}} \left( 2 \pi-\theta-\beta_1 \right)^{\frac{2}{3}}. \]
It implies that
\begin{equation}\label{a1314estimate02}
 \mathcal{A}_{13}+ \mathcal{A}_{14}\geq \frac{3}{2}16^{\frac{1}{3}}  \left(  \alpha + 2\theta \right)^{\frac{2}{3}}.
 \end{equation}
By the estimates in \eqref{a34estimate}, \eqref{a14estimate}, \eqref{a1314estimate} and \eqref{a1314estimate02}, it follows that if $\frac{\pi}{2}-\alpha + \theta > \pi$, or $\theta+\beta_1 > \pi$, or $\beta_1-\alpha - \theta>\pi$, \eqref{a12a13a14t1est1}, \eqref{a13a14a34t1est1} and \eqref{a13a14t1est1} still hold.  Hence, the estimate \eqref{a12a13a14a34t1est1} holds.

Therefore, \eqref{formulaofaction},  \eqref{a23a24t1est1} and \eqref{a12a13a14a34t1est1} imply that 
\begin{equation}\label{est1t1}
\mathcal{A} \geq \frac{3}{16} 16^{\frac{1}{3}}  \left[ 2\left( 2 \pi^2\right)^{\frac{1}{3}}  + \left(\frac{\pi}{2} + 2 \theta\right)^{\frac{2}{3}}+  \left( \frac{\pi}{2} + 3 \theta \right)^{\frac{2}{3}}  + \theta^{\frac{2}{3}} \right].
\end{equation}

 \vspace{0.2in} 
 
\textbf{Case 2:} $q_1$ and $q_4$ collide at $t=1$ $\left(q_1(1)=q_4(1) \right)$ and there is no collision at $t=0$ in $\mathcal{P}_{Q_1}([0,1])$. \\

By Lemma \ref{extensionformula1}, the path $\mathcal{P}_{Q_1}([0,1])$ can be extended to  $\mathcal{P}_{Q_1}([-1, 1])$. The binary collision $q_1(1)=q_4(1)$ implies another binary collision at $t=-1$:  $q_1(-1)= q_3(-1)$. Then by Theorem \ref{keplerestimate}, 
\begin{align*}
 & 2 \left(\mathcal{A}_{13}+\mathcal{A}_{14} \right) \\
 =& \int_{-1}^{1} \frac{1}{2} |\dot{q}_1-\dot{q}_3|^2+ \frac{4}{|q_1-q_3|} \, dt +\int_{-1}^{1} \frac{1}{2} |\dot{q}_1-\dot{q}_4|^2+ \frac{4}{|q_1-q_4|} \, dt\\
 \geq & 3\left( 32 \pi^2\right)^{\frac{1}{3}}. 
 \end{align*}
It implies that 
\begin{equation}\label{a13a14t1est2}
\mathcal{A}_{13}+\mathcal{A}_{14}\geq \frac{3}{2} \left( 32 \pi^2\right)^{\frac{1}{3}}. 
\end{equation}
Let $\beta_2\in [0, \pi]$ be the polar angle of the vector $\overrightarrow{q_2q_3}$ at $t=0$. Let $\alpha \in [-\pi/2, \pi/2]$ be the angle from the vector $\overrightarrow{q_2q_3}$ at $t=1$ rotating counterclockwisely to the vector $\overrightarrow{q_1q_3}$ at $t=1$. Similar to the argument in \textbf{Case 1},  if $\pi- \alpha+ \theta \leq \pi$ and $\theta-\alpha+\beta_2\leq \pi$, we have
\[  \mathcal{A}_{12} \geq \frac{3}{2}16^{\frac{1}{3}} \left( \pi- \alpha + \theta \right)^{\frac{2}{3}},  \qquad  \,\,    \mathcal{A}_{34} \geq \frac{3}{2}16^{\frac{1}{3}} \left( \frac{\pi}{2}-\alpha - \theta \right)^{\frac{2}{3}}, \]   
\[   \mathcal{A}_{23} \geq \frac{3}{2}16^{\frac{1}{3}} \left( \theta- \beta_2 \right)^{\frac{2}{3}} ,   \qquad  \,\, \mathcal{A}_{24} \geq \frac{3}{2}16^{\frac{1}{3}} \left( \theta-\alpha +\beta_2 \right)^{\frac{2}{3}}. \]
It follows that
\begin{eqnarray}\label{a12a34a23a24t1est2}
&& \mathcal{A}_{12}+ \mathcal{A}_{34}+  \mathcal{A}_{23}+  \mathcal{A}_{24} \nonumber\\
 &\geq& \frac{3}{2}16^{\frac{1}{3}}  \left[ \left( \pi- \alpha + \theta \right)^{\frac{2}{3}} + \left( \frac{\pi}{2}-\alpha - \theta \right)^{\frac{2}{3}}  +\left( \theta- \beta_2 \right)^{\frac{2}{3}}+ \left( \theta-\alpha +\beta_2 \right)^{\frac{2}{3}}\right] \nonumber\\
 &\geq & \frac{3}{2}16^{\frac{1}{3}}  \left[   \left( \pi- \alpha + \theta \right)^{\frac{2}{3}} +  \left( \frac{\pi}{2}-\alpha - \theta \right)^{\frac{2}{3}} +  \left( 2\theta-\alpha  \right)^{\frac{2}{3}}\right] \nonumber\\
 &\geq & \frac{3}{4}16^{\frac{1}{3}} \left[  \left( \frac{\pi}{2}+2 \theta \right)^{\frac{2}{3}} + \left( \frac{\pi}{2}-3 \theta \right)^{\frac{2}{3}}+\left( \pi- \theta \right)^{\frac{2}{3}}   \right].
\end{eqnarray}
If $\pi- \alpha+ \theta >\pi$, then 
\begin{equation}\label{a12estimate001}
 \mathcal{A}_{12} \geq \frac{3}{2}16^{\frac{1}{3}} \left( \pi+ \alpha - \theta \right)^{\frac{2}{3}}.
 \end{equation}
If $\theta-\alpha+\beta_2> \pi$, then
\begin{equation}\label{a24estimate001}
\mathcal{A}_{24} \geq \frac{3}{2}16^{\frac{1}{3}} \left( 2\pi- \theta+\alpha -\beta_2 \right)^{\frac{2}{3}}.
\end{equation}
Note that $\theta \in (0, \pi/10]$. It implies that 
\begin{equation}\label{a23a23estimate001}
\mathcal{A}_{23}+  \mathcal{A}_{24} \geq   \left( 2 \pi -2\theta+ \alpha  \right)^{\frac{2}{3}} >  \left( 2\theta- \alpha  \right)^{\frac{2}{3}}. \end{equation}
If $\pi- \alpha+ \theta \leq \pi$ and $\theta-\alpha+\beta_2> \pi$, by \eqref{a23a23estimate001}, the inequality \eqref{a12a34a23a24t1est2} holds. If $\pi- \alpha+ \theta >\pi$ and $\theta-\alpha+\beta_2> \pi$, by \eqref{a12estimate001} and \eqref{a23a23estimate001}, the inequality \eqref{a12a34a23a24t1est2} becomes
\begin{eqnarray*}
&& \mathcal{A}_{12}+ \mathcal{A}_{34}+  \mathcal{A}_{23}+  \mathcal{A}_{24} \nonumber\\
 &\geq & \frac{3}{2}16^{\frac{1}{3}}  \left[   \left( \pi+\alpha - \theta \right)^{\frac{2}{3}} +  \left( \frac{\pi}{2}-\alpha - \theta \right)^{\frac{2}{3}} +  \left( 2\theta-\alpha  \right)^{\frac{2}{3}}\right] \nonumber\\
 &\geq & \frac{3}{4}16^{\frac{1}{3}} \left[  \left( \frac{3\pi}{2}-2 \theta \right)^{\frac{2}{3}} + \left( \frac{\pi}{2}-3 \theta \right)^{\frac{2}{3}}+\left( \pi+ \theta \right)^{\frac{2}{3}}   \right]\\
 &> & \frac{3}{4}16^{\frac{1}{3}} \left[  \left( \frac{\pi}{2}+2 \theta \right)^{\frac{2}{3}} + \left( \frac{\pi}{2}-3 \theta \right)^{\frac{2}{3}}+\left( \pi- \theta \right)^{\frac{2}{3}}   \right].
\end{eqnarray*}
Hence, if $\theta \in (0, \pi/10]$, inequality \eqref{a12a34a23a24t1est2} always holds for all values of $\alpha \in [-\pi/2, \pi/2]$ and $\beta_2 \in [0, \pi]$. Then, inequalities \eqref{a13a14t1est2} and \eqref{a12a34a23a24t1est2} imply that
\begin{equation}\label{est2t1}
\mathcal{A} \geq \frac{3}{16} 16^{\frac{1}{3}}  \left[ 2\left( 2 \pi^2\right)^{\frac{1}{3}}  + \left(\frac{\pi}{2} + 2 \theta\right)^{\frac{2}{3}}+  \left( \frac{\pi}{2} - 3 \theta \right)^{\frac{2}{3}}  + \left( \pi-  \theta \right)^{\frac{2}{3}} \right].
\end{equation}
 \vspace{0.2in} 
 
\textbf{Case 3:} $q_1$ and $q_2$ collide at $t=0$ $\left(q_1(0)=q_2(0) \right)$ and there is no collision at $t=1$ in $\mathcal{P}_{Q_1}([0,1])$. \\

By Corollary \ref{extensionformula2}, the path $\mathcal{P}_{Q_1}([0,1])$ can be extended to $\mathcal{P}_{Q_1}([0,2])$. The binary collision $q_1(0)=q_2(0)$ implies another binary collision at $t=2$: $q_3(2)=q_4(2)$. By Theorem \ref{keplerestimate}, 
\begin{align*}
 & 2 \left(\mathcal{A}_{12}+\mathcal{A}_{34} \right) \\
 =& \int_{-1}^{1} \frac{1}{2} |\dot{q}_1-\dot{q}_2|^2+ \frac{4}{|q_1-q_2|} \, dt +\int_{-1}^{1} \frac{1}{2} |\dot{q}_3-\dot{q}_4|^2+ \frac{4}{|q_3-q_4|} \, dt\\
 \geq & 3\left( 32 \pi^2\right)^{\frac{1}{3}}. 
 \end{align*}
It follows that
\begin{equation}\label{a12a34t0est3}
\mathcal{A}_{12}+\mathcal{A}_{34}  \geq  \frac{3}{2} \left( 32 \pi^2\right)^{\frac{1}{3}}.
\end{equation}
Let $\beta \in [0, \pi]$ be the polar angle of the vector $\overrightarrow{q_2q_3}$ at $t=0$.  Let $\alpha \in (-\pi/2, \pi/2)$ be the angle from the vector $\overrightarrow{q_1q_4}$ at $t=1$ rotating counterclockwisely to the vector $\overrightarrow{q_1q_3}$ at $t=1$. By Theorem \ref{keplerestimate}, if $ \beta-\alpha- \theta \leq \pi$, \, $\beta + \theta -\alpha \leq \pi$ and $\theta + \beta \leq \pi$, we have 
\begin{align*}
 \mathcal{A}_{13} \geq \frac{3}{2}16^{\frac{1}{3}} \left(\alpha + \theta- \beta \right)^{\frac{2}{3}} ,  & \qquad  \,\,    \mathcal{A}_{24} \geq \frac{3}{2}16^{\frac{1}{3}} \left( \beta + \theta -\alpha\right)^{\frac{2}{3}}, \nonumber \\
 \mathcal{A}_{14} \geq \frac{3}{2}16^{\frac{1}{3}} \left(\theta + \beta \right)^{\frac{2}{3}} ,  & \qquad  \,\,    \mathcal{A}_{23} \geq \frac{3}{2}16^{\frac{1}{3}} \left( \theta -\beta\right)^{\frac{2}{3}}, 
\end{align*}
It follows that if $ \beta-\alpha- \theta \leq \pi$, \, $\beta + \theta -\alpha \leq \pi$ and $\theta + \beta \leq \pi$,
\begin{equation}\label{a13a24a14a23t0est3}
\mathcal{A}_{13}+ \mathcal{A}_{24}+\mathcal{A}_{14}+\mathcal{A}_{23} \geq  3 (16)^{\frac{1}{3}}(2 \theta)^{\frac{2}{3}}.
\end{equation}
Note that $ \beta-\alpha- \theta < \beta-\alpha+ \theta$. If $ \beta-\alpha- \theta > \pi$ \, or \, $\beta + \theta -\alpha > \pi$, we have
\[\mathcal{A}_{13}+ \mathcal{A}_{24} \geq \min \left\{ \frac{3}{2}16^{\frac{1}{3}}\left( 2\pi-2 \theta \right)^{\frac{2}{3}},  \frac{3}{2}16^{\frac{1}{3}}\left(2 \theta \right)^{\frac{2}{3}} \right\}  = \frac{3}{2}16^{\frac{1}{3}}\left(2 \theta \right)^{\frac{2}{3}}. \]
If $\theta + \beta > \pi$, we have 
\[\mathcal{A}_{14}+\mathcal{A}_{23} \geq \frac{3}{2}16^{\frac{1}{3}} \left( 2\pi-2 \theta \right)^{\frac{2}{3}} > \frac{3}{2}16^{\frac{1}{3}}\left(2 \theta \right)^{\frac{2}{3}}. \]
Hence, inequality \eqref{a13a24a14a23t0est3} always holds. By  \eqref{formulaofaction}, \eqref{a12a34t0est3} and \eqref{a13a24a14a23t0est3}, it follows that
\begin{equation}\label{est3t0}
\mathcal{A} \geq \frac{3}{8} 16^{\frac{1}{3}} \left[  \left( 2 \pi^2\right)^{\frac{1}{3}} + 2 (2 \theta)^{\frac{2}{3}} \right].
\end{equation}

 \vspace{0.2in} 
 
\textbf{Case 4:} $q_3$ and $q_4$ collide at $t=0$ $\left(q_3(0)=q_4(0) \right)$ and there is no collision at $t=1$ in $\mathcal{P}_{Q_1}([0,1])$. \\

By Corollary \ref{extensionformula2}, the path $\mathcal{P}_{Q_1}([0,1])$ can be extended to $\mathcal{P}_{Q_1}([0,2])$. The binary collision $q_3(0)=q_4(0)$ implies another binary collision at $t=2$: $q_1(2)=q_2(2)$. By Theorem \ref{keplerestimate}, 
\begin{align*}
 & 2 \left(\mathcal{A}_{12}+\mathcal{A}_{34} \right) \\
 =& \int_{-1}^{1} \frac{1}{2} |\dot{q}_1-\dot{q}_2|^2+ \frac{4}{|q_1-q_2|} \, dt +\int_{-1}^{1} \frac{1}{2} |\dot{q}_3-\dot{q}_4|^2+ \frac{4}{|q_3-q_4|} \, dt\\
 \geq & 3\left( 32 \pi^2\right)^{\frac{1}{3}}. 
 \end{align*}
 It follows that
\begin{equation}\label{a12a34t0est4}
\mathcal{A}_{12}+\mathcal{A}_{34}  \geq  \frac{3}{2} \left( 32 \pi^2\right)^{\frac{1}{3}}.
\end{equation}
Let $\alpha \in (-\pi/2, \pi/2)$ be the angle from the vector $\overrightarrow{q_1q_4}$ at $t=1$ rotating counterclockwisely to the vector $\overrightarrow{q_1q_3}$ at $t=1$. By Theorem \ref{keplerestimate}, the following estimates hold if the collision point at $t=0$ satisfies $q_{3x}(0)=q_{4x}(0) \geq q_{2x}(0)$: 
\begin{align*}
 \mathcal{A}_{13} \geq \frac{3}{2}16^{\frac{1}{3}} \left(\alpha + \theta \right)^{\frac{2}{3}} ,  & \qquad  \,\,    \mathcal{A}_{24} \geq \frac{3}{2}16^{\frac{1}{3}} \left(  \theta - \alpha\right)^{\frac{2}{3}}, \nonumber \\
 \mathcal{A}_{14} \geq \frac{3}{2}16^{\frac{1}{3}}  \theta^{\frac{2}{3}} ,  & \qquad  \,\,    \mathcal{A}_{23} \geq \frac{3}{2}16^{\frac{1}{3}} \theta^{\frac{2}{3}}. 
\end{align*}
It follows that 
\begin{equation}\label{a13a24a14a23t0est4}
\mathcal{A}_{13}+ \mathcal{A}_{24}+\mathcal{A}_{14}+\mathcal{A}_{23} \geq  \frac{3}{2}16^{\frac{1}{3}} \left[ 2(\theta)^{\frac{2}{3}}+ (2 \theta)^{\frac{2}{3}} \right].
\end{equation}

If the collision point at $t=0$ satisfies $q_{1x}(0) \leq q_{3x}(0)=q_{4x}(0) <q_{2x}(0)$, by Theorem \ref{keplerestimate}, we have
\begin{align*}
 \mathcal{A}_{13} \geq \frac{3}{2}16^{\frac{1}{3}} \left(\alpha + \theta \right)^{\frac{2}{3}} ,  & \qquad  \,\,    \mathcal{A}_{24} \geq \frac{3}{2}16^{\frac{1}{3}} \left( \pi- \theta +\alpha\right)^{\frac{2}{3}}, \nonumber \\
 \mathcal{A}_{14} \geq \frac{3}{2}16^{\frac{1}{3}}  \left( \theta\right)^{\frac{2}{3}} ,  & \qquad  \,\,    \mathcal{A}_{23} \geq \frac{3}{2}16^{\frac{1}{3}}  \left( \pi-\theta\right)^{\frac{2}{3}}.
\end{align*}
It is easy to check that inequality \eqref{a13a24a14a23t0est4} holds for $\theta \in (0, \pi/10]$. 

If the collision point at $t=0$ satisfies $q_{3x}(0)=q_{4x}(0) < q_{1x}(0)$, by Theorem \ref{keplerestimate}, we have
\begin{align*}
 \mathcal{A}_{13} \geq \frac{3}{2}16^{\frac{1}{3}} \left(\pi-\alpha - \theta \right)^{\frac{2}{3}} ,  & \qquad  \,\,    \mathcal{A}_{24} \geq \frac{3}{2}16^{\frac{1}{3}} \left( \pi- \theta +\alpha\right)^{\frac{2}{3}}, \nonumber \\
 \mathcal{A}_{14} \geq \frac{3}{2}16^{\frac{1}{3}}  \left( \pi-\theta\right)^{\frac{2}{3}} ,  & \qquad  \,\,    \mathcal{A}_{23} \geq \frac{3}{2}16^{\frac{1}{3}}  \left( \pi-\theta\right)^{\frac{2}{3}}.
\end{align*}
It is clear that inequality \eqref{a13a24a14a23t0est4} holds for $\theta \in (0, \pi/10]$. Therefore, inequalities \eqref{a12a34t0est4} and \eqref{a13a24a14a23t0est4} imply that
\begin{equation}\label{est4t0}
\mathcal{A} \geq \frac{3}{8} 16^{\frac{1}{3}} \left[  \left( 2 \pi^2\right)^{\frac{1}{3}} + 2 (\theta)^{\frac{2}{3}}+ (2 \theta)^{\frac{2}{3}} \right].
\end{equation}

 \vspace{0.2in} 
 
\textbf{Case 5:}  collisions happen on both ends of $\mathcal{P}_{Q_1}([0,1])$. \\

Note that at $t=0$, there are two possible binary collisions: $q_1(0)=q_2(0)$ and $q_3(0)=q_4(0)$. Triple collisions and total collision at $t=0$ will cause at least one of the two binary collisions. Similarly, at $t=1$, we need to consider the other two binary collisions: $q_1(1)=q_4(1)$ and $q_2(1)=q_3(1)$. It implies that there are at least two different pairs of binary collisions in $[0,1]$. By Theorem \ref{keplerestimate}, the action $\mathcal{A}$ satisfies
\begin{equation}\label{est5t0}
\mathcal{A} \geq \frac{3}{4} 16^{\frac{1}{3}} (\pi)^{\frac{2}{3}}.
\end{equation}

Therefore, for $\theta \in (0, \frac{\pi}{10}]$, the estimates in the 5 cases (\eqref{est1t1}, \eqref{est2t1}, \eqref{est3t0}, \eqref{est4t0}, \eqref{est5t0}) imply that 
\begin{equation}\label{actionboundpq1}
\mathcal{A} \geq \frac{3}{8} 16^{\frac{1}{3}} \left[  \left( 2 \pi^2\right)^{\frac{1}{3}} + 2 (2 \theta)^{\frac{2}{3}} \right].
\end{equation}
The proof is complete.
\end{proof}

\subsection{Lower bound of action of $\mathcal{P}_{Q_2}$ with boundary collisions}
Let $\tilde{q}(t)= \begin{bmatrix}
\tilde{q}_{1}(t) \\
\tilde{q}_{2}(t) \\
\tilde{q}_{3}(t) \\
\tilde{q}_{4}(t)  
\end{bmatrix}$ be the position matrix path of the minimizer $\mathcal{P}_{Q_2}=\mathcal{P}_{Q_2}([0,1])$. We assume the two boundary configurations in $\mathcal{P}_{Q_2}=\mathcal{P}_{Q_2}([0,1])$ to be 
\begin{equation}\label{boundarypq2}
\tilde{q}(0)=\begin{bmatrix}
-a_{12}-c_{12} & 0 \\
-a_{12} &   0\\
(2a_{12}+c_{12})/2 & b_{12} \\
(2a_{12}+c_{12})/2 & -b_{12} 
\end{bmatrix}, \,  \,\, \tilde{q}(1)= \begin{bmatrix}
-a_{22} & -b_{22} \\
-a_{22} &   b_{22}\\
a_{22} & c_{22} \\
a_{22} & -c_{22} 
\end{bmatrix}R(\theta), 
\end{equation} where $a_{12}, a_{22} \in \mathbb{R}$ and $b_{12}, b_{22}, c_{12}, c_{22} \geq 0$. In this subsection, we show that
\begin{lemma}\label{lowerbdd2}
Let $\theta \in (0, \pi/10]$. If the minimizing path $\mathcal{P}_{Q_2}([0,1])$ has a boundary collision, then its action $\mathcal{A}$ satisfies
\begin{equation}\label{estlowerbddpq2}
\mathcal{A}=\mathcal{A}(\tilde{q}) \geq  \frac{3}{8}16^{\frac{1}{3}} \left[ \pi^{\frac{2}{3}}+ \theta^{\frac{2}{3}} + 2\left(2 \theta \right)^{\frac{2}{3}} \right]. 
\end{equation}
\end{lemma}
\begin{proof}
According to the results of Marchal \cite{Mar} and Chenciner \cite{CA}, the only possible collisions in $\mathcal{P}_{Q_2}([0,1])$ are on the boundaries: $\tilde{q}(0)$ and $\tilde{q}(1)$. By the definitions of the two configuration sets $Q_S$ and $Q_{E_2}$ in \eqref{2QS} and \eqref{2QEi}, the possible boundary collisions could be the following:
\begin{enumerate}
\item At $t=0$: \\
\quad $q_1=q_2$; \, $q_3=q_4$; \, $q_2=q_3=q_4$; \, $q_1=q_3=q_4$, \, $q_1=q_2=q_3=q_4$;
\item At $t=1$: \\
 \quad $q_1=q_2$; \, $q_3=q_4$; \,  $q_1=q_2=q_3=q_4$.
\end{enumerate}
Note that when estimating the action $\mathcal{A}_{ij} \, (1 \leq i <j \leq 4)$, the triple collisions and total collisions can be treated as several binary collisions happening simultaneously. Hence, we only need to discuss the four possible binary collisions case by case. 
 \vspace{0.2in} 
 
\textbf{Case 1:} $\tilde{q}_1(1)=\tilde{q}_2(1)$ in $\mathcal{P}_{Q_2}([0,1])$.  \\

By Theorem \ref{keplerestimate}, we have
\begin{equation}\label{a12est1pq2}
\mathcal{A}_{12} \geq \frac{3}{2} \left( 16 \pi^2\right)^{\frac{1}{3}}.
\end{equation}
Let $\beta_1 \in [0, \pi]$ be the polar angle of the vector $\overrightarrow{q_1q_3}$ at $t=0$. Let $\beta_2 \in [0, \pi]$ be the polar angle of the vector $\overrightarrow{q_2q_3}$ at $t=0$. In the configuration $\tilde{q}(1) R(-\theta)$, let $\alpha \in [0, \pi]$ be the polar angle of the vector $\overrightarrow{q_1q_3}$ at $t=1$.  By Theorem \ref{keplerestimate}, if $\alpha + \theta - \beta_1 \leq \pi$, $\theta - \alpha+ \beta_1 \leq \pi$, $\theta+ \alpha - \beta_2 \leq \pi$ and $ \theta-\alpha + \beta_2 \leq \pi$, we have
\begin{align*}
 \mathcal{A}_{13} \geq \frac{3}{2}16^{\frac{1}{3}} \left(\alpha + \theta - \beta_1\right)^{\frac{2}{3}} ,  & \qquad  \,\,    \mathcal{A}_{14} \geq \frac{3}{2}16^{\frac{1}{3}} \left(  \theta - \alpha+ \beta_1\right)^{\frac{2}{3}}, \nonumber \\
 \mathcal{A}_{23} \geq \frac{3}{2}16^{\frac{1}{3}} \left( \theta+ \alpha - \beta_2 \right)^{\frac{2}{3}} ,  & \qquad  \,\,    \mathcal{A}_{24} \geq \frac{3}{2}16^{\frac{1}{3}} \left( \theta-\alpha + \beta_2 \right)^{\frac{2}{3}}, \\
 \mathcal{A}_{34} \geq \frac{3}{2}16^{\frac{1}{3}} \theta^{\frac{2}{3}}. &
\end{align*}
It follows that
\begin{eqnarray}\label{a13a14a23a24a34est1pq2}
 & & \mathcal{A}_{13} +\mathcal{A}_{14} + \mathcal{A}_{23} +\mathcal{A}_{24} +\mathcal{A}_{34}  \nonumber \\
&\geq& \frac{3}{2}16^{\frac{1}{3}} \left[ \left(\alpha + \theta - \beta_1\right)^{\frac{2}{3}} + \left(  \theta - \alpha+ \beta_1\right)^{\frac{2}{3}}+  \left( \theta+ \alpha - \beta_2 \right)^{\frac{2}{3}} + \left( \theta-\alpha + \beta_2 \right)^{\frac{2}{3}} +\theta^{\frac{2}{3}}\right]  \nonumber \\
&\geq&  \frac{3}{2}16^{\frac{1}{3}} \left[ \theta^{\frac{2}{3}} + 2\left(2 \theta \right)^{\frac{2}{3}} \right].
\end{eqnarray}
Note that $\theta \in (0, \pi/10]$. If $\alpha + \theta - \beta_1 > \pi$ or $\theta - \alpha+ \beta_1 > \pi$, 
\[  \mathcal{A}_{13}+  \mathcal{A}_{14} \geq  \min \left\{ \frac{3}{2}16^{\frac{1}{3}}\left(2 \theta \right)^{\frac{2}{3}}, \, \frac{3}{2}16^{\frac{1}{3}}\left(2 \pi - 2 \theta \right)^{\frac{2}{3}}, \,    \frac{3}{2}16^{\frac{1}{3}}\left(4 \pi - 2 \theta \right)^{\frac{2}{3}} \right\} =\frac{3}{2}16^{\frac{1}{3}} \left(2 \theta \right)^{\frac{2}{3}}. \]
Similarly, if $\alpha + \theta - \beta_2 > \pi$ or $\theta - \alpha+ \beta_2 > \pi$, 
\[  \mathcal{A}_{23}+  \mathcal{A}_{24} \geq \min \left\{ \frac{3}{2}16^{\frac{1}{3}}\left(2 \theta \right)^{\frac{2}{3}}, \, \frac{3}{2}16^{\frac{1}{3}}\left(2 \pi - 2 \theta \right)^{\frac{2}{3}}, \,    \frac{3}{2}16^{\frac{1}{3}}\left(4 \pi - 2 \theta \right)^{\frac{2}{3}} \right\} = \frac{3}{2}16^{\frac{1}{3}}\left(2 \theta \right)^{\frac{2}{3}}. \]
It follows that inequality \eqref{a13a14a23a24a34est1pq2} holds for all $\alpha, \beta_1, \beta_2 \in [0, \pi]$. Hence, by \eqref{formulaofaction}, \eqref{a12est1pq2} and \eqref{a13a14a23a24a34est1pq2}, the action $\mathcal{A}$ satisfies
\begin{equation}\label{est1pq2}
\mathcal{A} \geq  \frac{3}{8}16^{\frac{1}{3}} \left[ \pi^{\frac{2}{3}}+ \theta^{\frac{2}{3}} + 2\left(2 \theta \right)^{\frac{2}{3}} \right].
\end{equation}

 \vspace{0.2in} 
 
\textbf{Case 2:}  $\tilde{q}_3(1)=\tilde{q}_4(1)$ in $\mathcal{P}_{Q_2}([0,1])$.  \\

By Theorem \ref{keplerestimate}, we have
\begin{equation}\label{a34est2pq2}
\mathcal{A}_{34} \geq \frac{3}{2} \left( 16 \pi^2\right)^{\frac{1}{3}}.
\end{equation} 
Let $\beta_1 \in [0, \pi]$ be the polar angle of the vector $\overrightarrow{q_1q_3}$ at $t=0$. Let $\beta_2 \in [0, \pi]$ be the polar angle of the vector $\overrightarrow{q_2q_3}$ at $t=0$. In the configuration $\tilde{q}(1) R(-\theta)$, let $\alpha\in [0, \pi]$ be the polar angle of the vector $\overrightarrow{q_1q_4}$ at $t=1$. By Theorem \ref{keplerestimate}, if $\alpha + \theta - \beta_1 \leq \pi$, $\alpha+\theta + \beta_1 \leq \pi$, $\alpha - \theta + \beta_2 \leq \pi$ and $\alpha - \theta - \beta_2 \leq \pi$, we have 
\begin{align*}
 \mathcal{A}_{13} \geq \frac{3}{2}16^{\frac{1}{3}} \left(\alpha + \theta - \beta_1\right)^{\frac{2}{3}} ,  & \qquad  \,\,    \mathcal{A}_{14} \geq \frac{3}{2}16^{\frac{1}{3}} \left(  \alpha+\theta + \beta_1\right)^{\frac{2}{3}},    \nonumber \\
 \mathcal{A}_{23} \geq \frac{3}{2}16^{\frac{1}{3}} \left( \alpha - \theta+ \beta_2 \right)^{\frac{2}{3}},   & \qquad  \,\,    \mathcal{A}_{24} \geq \frac{3}{2}16^{\frac{1}{3}} \left( \alpha -\theta- \beta_2 \right)^{\frac{2}{3}} , \\
 \mathcal{A}_{12} \geq \frac{3}{2}16^{\frac{1}{3}} \left( \frac{\pi}{2}+ \theta \right)^{\frac{2}{3}}. &
\end{align*}
It follows that if $\alpha + \theta - \beta_1 \leq \pi$, $\alpha+\theta + \beta_1 \leq \pi$, $\alpha - \theta + \beta_2 \leq \pi$ and $\alpha - \theta - \beta_2 \leq \pi$, the following inequality holds:
\begin{eqnarray}\label{a13a14a23a24a34est2pq2}
 & & \mathcal{A}_{13} +\mathcal{A}_{14} + \mathcal{A}_{23} +\mathcal{A}_{24}  \nonumber \\
&\geq & \frac{3}{2}16^{\frac{1}{3}} \left[ \left(\alpha + \theta - \beta_1\right)^{\frac{2}{3}} + \left(  \theta + \alpha+ \beta_1\right)^{\frac{2}{3}}+  \left(  \alpha -\theta + \beta_2 \right)^{\frac{2}{3}} + \left( \alpha - \theta - \beta_2 \right)^{\frac{2}{3}} \right]  \nonumber \\
&\geq &  \frac{3}{2}16^{\frac{1}{3}} \left[\left( 2 \theta + 2 \alpha\right)^{\frac{2}{3}}+ \left( 2 \alpha- 2 \theta  \right)^{\frac{2}{3}}  \right] \nonumber \\
&\geq &  \frac{3}{2}16^{\frac{1}{3}} \left(4 \theta \right)^{\frac{2}{3}}.
\end{eqnarray}
If $\alpha + \theta - \beta_1 > \pi$ or $\alpha+\theta + \beta_1>\pi$,
\[   \mathcal{A}_{13} +\mathcal{A}_{14}  \geq \min  \left \{ \frac{3}{2}16^{\frac{1}{3}} \left( 2 \theta + 2 \alpha\right)^{\frac{2}{3}},   \, \frac{3}{2}16^{\frac{1}{3}} \left( 2 \pi  -2 \theta - 2 \alpha\right)^{\frac{2}{3}} \right\}.  \]
If $\alpha - \theta +\beta_2 > \pi$ or $\alpha-\theta - \beta_2>\pi$,
\[   \mathcal{A}_{23} +\mathcal{A}_{24}  \geq \min  \left\{ \frac{3}{2}16^{\frac{1}{3}} \left( 2 \alpha-2 \theta  \right)^{\frac{2}{3}},   \, \frac{3}{2}16^{\frac{1}{3}} \left( 2 \pi  +2 \theta - 2 \alpha\right)^{\frac{2}{3}} \right\}.  \]
Note that $\theta \in (0, \pi/10]$. It follows that 
\begin{eqnarray*}
& & \mathcal{A}_{13} +\mathcal{A}_{14}+\mathcal{A}_{23} +\mathcal{A}_{24}  \\
&\geq& \min \left\{  \frac{3}{2}16^{\frac{1}{3}} \left(4 \theta \right)^{\frac{2}{3}},  \, \frac{3}{2}16^{\frac{1}{3}} \left(2 \pi- 4 \theta \right)^{\frac{2}{3}}, \, \frac{3}{2}16^{\frac{1}{3}} \left(2 \pi +4 \theta \right)^{\frac{2}{3}} \right\}  \\
&=& \frac{3}{2}16^{\frac{1}{3}} \left(4 \theta \right)^{\frac{2}{3}}.
\end{eqnarray*}
Hence, the action $\mathcal{A}$ satisfies
\begin{equation}\label{est2pq2}
\begin{split}
\mathcal{A} &= \frac{1}{4} \left[ \mathcal{A}_{13} +\mathcal{A}_{14}+\mathcal{A}_{23} +\mathcal{A}_{24} +   \mathcal{A}_{12}+ \mathcal{A}_{34} \right] \\
& \geq \frac{3}{2}16^{\frac{1}{3}} \left[   \left(4 \theta \right)^{\frac{2}{3}}+ \left( \frac{\pi}{2}+ \theta \right)^{\frac{2}{3}} +   \pi^{\frac{2}{3}}\right].
\end{split} 
\end{equation}

 \vspace{0.2in} 
 
\textbf{Case 3:}  $\tilde{q}_1(0)=\tilde{q}_2(0)$ in $\mathcal{P}_{Q_2}([0,1])$.  \\

By Theorem \ref{keplerestimate}, we have
\begin{equation}\label{a34est2pq2}
\mathcal{A}_{12} \geq \frac{3}{2} \left( 16 \pi^2\right)^{\frac{1}{3}}.
\end{equation} 

Let $\beta \in [0, \pi]$ be the polar angle of the vector $\overrightarrow{q_1q_3}$ at $t=0$. In the configuration $\tilde{q}(1) R(-\theta)$, let $\alpha_1\in [0, \pi]$ be the polar angle of the vector $\overrightarrow{q_1q_3}$ at $t=1$ and let $\alpha_2 \in [0, 2\pi)$ be the polar angle of the vector $\overrightarrow{q_2q_3}$ at $t=1$. By Theorem \ref{keplerestimate}, if $\alpha_1 + \theta - \beta \leq \pi$, $\alpha_1-\theta -\beta \leq \pi$, $\alpha_2 + \theta- \beta \leq \pi$ and $\alpha_2 -\theta- \beta \leq \pi$, we have
\begin{align*}
 \mathcal{A}_{13} \geq \frac{3}{2}16^{\frac{1}{3}} \left(\alpha_1 + \theta - \beta\right)^{\frac{2}{3}} ,  & \qquad  \,\,    \mathcal{A}_{24} \geq \frac{3}{2}16^{\frac{1}{3}} \left(  \alpha_1-\theta -\beta \right)^{\frac{2}{3}},    \nonumber \\
 \mathcal{A}_{23} \geq \frac{3}{2}16^{\frac{1}{3}} \left( \alpha_2 + \theta- \beta \right)^{\frac{2}{3}},   & \qquad  \,\,    \mathcal{A}_{14} \geq \frac{3}{2}16^{\frac{1}{3}} \left( \alpha_2 -\theta- \beta \right)^{\frac{2}{3}} , \\
 \mathcal{A}_{34} \geq \frac{3}{2}16^{\frac{1}{3}}\theta^{\frac{2}{3}}. &
\end{align*}
It follows that, if $\alpha_1 + \theta - \beta \leq \pi$, $|\alpha_1-\theta -\beta| \leq \pi$, $\alpha_2 + \theta- \beta \leq \pi$ and $|\alpha_2 -\theta- \beta| \leq \pi$, the following inequality holds:
\begin{eqnarray}\label{a13a14a23a24est3pq2}
 & & \mathcal{A}_{13} +\mathcal{A}_{14} + \mathcal{A}_{23} +\mathcal{A}_{24}  \nonumber \\
&\geq & \frac{3}{2}16^{\frac{1}{3}} \left[ \left(\alpha _1+ \theta - \beta\right)^{\frac{2}{3}} + \left(  \alpha_1- \theta  - \beta\right)^{\frac{2}{3}}+  \left(  \alpha_2 +\theta - \beta \right)^{\frac{2}{3}} + \left( \alpha_2 -\theta - \beta \right)^{\frac{2}{3}} \right]  \nonumber \\
&\geq &  3 (16)^{\frac{1}{3}} \left(2 \theta \right)^{\frac{2}{3}}.
\end{eqnarray}

If $\alpha_1 + \theta - \beta > \pi$ or $ |\alpha_1-\theta -\beta| > \pi$,  we have
\begin{eqnarray*}
& & \mathcal{A}_{13} + \mathcal{A}_{24} \\
&\geq & \min \left\{ \frac{3}{2}16^{\frac{1}{3}}\left(2 \theta \right)^{\frac{2}{3}}, \,  \frac{3}{2}16^{\frac{1}{3}}\left(2 \pi- 2 \theta \right)^{\frac{2}{3}}, \, \frac{3}{2}16^{\frac{1}{3}}\left(2 \pi + 2 \theta \right)^{\frac{2}{3}}  \right\}\\
&= & \frac{3}{2}16^{\frac{1}{3}}\left(2 \theta \right)^{\frac{2}{3}}.  
\end{eqnarray*}
Similarly, if $\alpha_2 + \theta- \beta > \pi$ or $|\alpha_2 -\theta- \beta| > \pi$,
\begin{eqnarray*}
& & \mathcal{A}_{23} + \mathcal{A}_{14} \\
&\geq & \min \left\{ \frac{3}{2}16^{\frac{1}{3}}\left(2 \theta \right)^{\frac{2}{3}}, \,  \frac{3}{2}16^{\frac{1}{3}}\left(2 \pi- 2 \theta \right)^{\frac{2}{3}}, \, \frac{3}{2}16^{\frac{1}{3}}\left(2 \pi + 2 \theta \right)^{\frac{2}{3}}  \right\}\\
&= & \frac{3}{2}16^{\frac{1}{3}}\left(2 \theta \right)^{\frac{2}{3}}.  
\end{eqnarray*}
Hence, inequality \eqref{a13a14a23a24est3pq2} holds for all values of $\beta, \alpha_1 \in [0, \pi]$ and $\alpha_2 \in [0, 2 \pi)$. Therefore, the action $\mathcal{A}$ satisfies
\begin{equation}\label{est3pq2}
\mathcal{A} \geq  \frac{3}{8}16^{\frac{1}{3}} \left[ \pi^{\frac{2}{3}}+ \theta^{\frac{2}{3}} + 2\left(2 \theta \right)^{\frac{2}{3}} \right].
\end{equation}

 \vspace{0.2in} 
 
\textbf{Case 4:}  $\tilde{q}_3(0)=\tilde{q}_4(0)$ in $\mathcal{P}_{Q_2}([0,1])$.  \\

By Theorem \ref{keplerestimate}, we have
\begin{equation}\label{a34est3pq2}
\mathcal{A}_{34} \geq \frac{3}{2} \left( 16 \pi^2\right)^{\frac{1}{3}}.
\end{equation} 
Note that the vector $\overrightarrow{q_1q_2}$ rotates an angle $\pi/2 +\theta$. It implies that
\begin{equation}\label{est12incase4}
\mathcal{A}_{12} \geq \frac{3}{2}16^{\frac{1}{3}} \left( \frac{\pi}{2} + \theta \right)^{\frac{2}{3}}. 
\end{equation} 
In the configuration $\tilde{q}(1) R(-\theta)$, let $\alpha_1$ be the polar angle of the vector $\overrightarrow{q_1q_3}$ at $t=1$, and let $\alpha_2$ be the polar angle of the vector $\overrightarrow{q_2q_3}$ at $t=1$. By the definition of the configurations $Q_s$ and $Q_{e_2}$, it implies that $\alpha_1 \in [0, \pi]$ and $\alpha_2 \in [0, 2\pi)$. In the configuration $\tilde{q}(0)$, all the four bodies are on the $x$-axis. There are basically three subcases: 
\begin{enumerate}[(i):] 
\item the collision pair $\tilde{q}_3(0)=\tilde{q}_4(0)$ is on the right hand side of $\tilde{q}_2(0)$. By Theorem \ref{keplerestimate}, if $\alpha_1 + \theta \leq \pi$, $\alpha_2 + \theta \leq \pi$ and $\alpha_2- \theta \leq \pi$, we have
\begin{align*}
 \mathcal{A}_{13} \geq \frac{3}{2}16^{\frac{1}{3}} \left(\alpha_1 + \theta\right)^{\frac{2}{3}} ,  & \qquad  \,\,    \mathcal{A}_{24} \geq \frac{3}{2}16^{\frac{1}{3}} \left(  \alpha_1 - \theta\right)^{\frac{2}{3}},    \nonumber \\
 \mathcal{A}_{23} \geq \frac{3}{2}16^{\frac{1}{3}} \left( \alpha_2 + \theta \right)^{\frac{2}{3}},   & \qquad  \,\,    \mathcal{A}_{14} \geq \frac{3}{2}16^{\frac{1}{3}} \left( \alpha_2-  \theta\right)^{\frac{2}{3}} .&
\end{align*}
It follows that if $\alpha_1 + \theta \leq \pi$, $\alpha_2 + \theta \leq \pi$ and $\alpha_2- \theta \leq \pi$, the following inequality holds:
\begin{eqnarray}\label{a13a14a23a24est4pq2}
 & & \mathcal{A}_{13} +\mathcal{A}_{14} + \mathcal{A}_{23} +\mathcal{A}_{24}  \nonumber \\
&\geq & \frac{3}{2}16^{\frac{1}{3}} \left[ \left(\alpha _1+ \theta \right)^{\frac{2}{3}} + \left(  \alpha_1- \theta  \right)^{\frac{2}{3}}+  \left(  \alpha_2 +\theta  \right)^{\frac{2}{3}} + \left( \alpha_2 -\theta  \right)^{\frac{2}{3}} \right]  \nonumber \\
&\geq &  3 (16)^{\frac{1}{3}} \left(2 \theta \right)^{\frac{2}{3}}.
\end{eqnarray}
If $\alpha_1 + \theta > \pi$, 
\[\mathcal{A}_{13}+\mathcal{A}_{24} \geq  \frac{3}{2}16^{\frac{1}{3}} \left(2 \pi - 2\theta\right)^{\frac{2}{3}} >  \frac{3}{2}16^{\frac{1}{3}} \left( 2\theta\right)^{\frac{2}{3}}. \]
If $\alpha_2 + \theta > \pi$ or $\alpha_2- \theta > \pi$, 
\[ \mathcal{A}_{23}+  \mathcal{A}_{14}  \geq \min \left\{\frac{3}{2}16^{\frac{1}{3}} \left(2 \pi - 2\theta\right)^{\frac{2}{3}} ,    \frac{3}{2}16^{\frac{1}{3}} \left( 2\theta\right)^{\frac{2}{3}}\right\}= \frac{3}{2}16^{\frac{1}{3}} \left( 2\theta\right)^{\frac{2}{3}}. \]
It follows that inequality \eqref{a13a14a23a24est4pq2} holds for any $\alpha_1 \in [0, \pi]$ and $\alpha_2 \in [0, 2\pi)$. 

\item The collision pair is in the middle of body 1 and 2:
\[\tilde{q}_1(0) \leq \tilde{q}_3(0)=\tilde{q}_4(0) \leq \tilde{q}_2(0).\]
By Theorem \ref{keplerestimate}, if $\alpha_1 + \theta \leq \pi$, $\pi- \alpha_1 + \theta \leq \pi$, $\alpha_2 + \theta-\pi \leq \pi$ and $\alpha_2- \theta \leq \pi$, we have
\begin{align*}
 \mathcal{A}_{13} \geq \frac{3}{2}16^{\frac{1}{3}} \left(\alpha_1 + \theta\right)^{\frac{2}{3}} ,  & \qquad  \,\,    \mathcal{A}_{24} \geq \frac{3}{2}16^{\frac{1}{3}} \left(  \pi- \alpha_1 + \theta\right)^{\frac{2}{3}},    \nonumber \\
 \mathcal{A}_{23} \geq \frac{3}{2}16^{\frac{1}{3}} \left( \alpha_2 + \theta- \pi \right)^{\frac{2}{3}},   & \qquad  \,\,    \mathcal{A}_{14} \geq \frac{3}{2}16^{\frac{1}{3}} \left( \alpha_2-  \theta\right)^{\frac{2}{3}}. & 
\end{align*}
 It follows that 
 \begin{eqnarray}\label{a13a14a23a24case42}
 & & \mathcal{A}_{13} +\mathcal{A}_{14} + \mathcal{A}_{23} +\mathcal{A}_{24}  \nonumber \\
&\geq & \frac{3}{2}16^{\frac{1}{3}} \left[ \left(\alpha _1+ \theta \right)^{\frac{2}{3}} + \left(  \pi- \alpha_1 + \theta  \right)^{\frac{2}{3}}+  \left( \alpha_2 + \theta- \pi    \right)^{\frac{2}{3}} + \left( \alpha_2 -\theta  \right)^{\frac{2}{3}} \right]  \nonumber \\
&\geq &  \frac{3}{2} (16)^{\frac{1}{3}} \left[ \left(\pi+2 \theta \right)^{\frac{2}{3}}  +\left(\pi- 2 \theta \right)^{\frac{2}{3}} \right] \nonumber\\
& > & 3 (16)^{\frac{1}{3}} \left(2 \theta \right)^{\frac{2}{3}}.
\end{eqnarray}
If $\alpha_1 + \theta > \pi$ or $\pi- \alpha_1 + \theta > \pi$,
\[  \mathcal{A}_{13}+  \mathcal{A}_{24} \geq   \min\left\{  \frac{3}{2} (16)^{\frac{1}{3}}\left(\pi+2 \theta \right)^{\frac{2}{3}}, \frac{3}{2} (16)^{\frac{1}{3}}\left(\pi- 2 \theta \right)^{\frac{2}{3}}   \right\}>   \frac{3}{2} (16)^{\frac{1}{3}}\left(2 \theta \right)^{\frac{2}{3}}.  \]
If $\alpha_2 + \theta-\pi > \pi$ or  $\alpha_2- \theta > \pi$, 
\[  \mathcal{A}_{23}+  \mathcal{A}_{14} \geq   \min\left\{  \frac{3}{2} (16)^{\frac{1}{3}}\left(\pi+2 \theta \right)^{\frac{2}{3}}, \frac{3}{2} (16)^{\frac{1}{3}}\left(\pi- 2 \theta \right)^{\frac{2}{3}}   \right\}>   \frac{3}{2} (16)^{\frac{1}{3}}\left(2 \theta \right)^{\frac{2}{3}}.  \]
Hence, inequality \eqref{a13a14a23a24case42} holds for any $\alpha_1 \in [0, \pi]$ and $\alpha_2 \in [0, 2\pi)$.

\item the collision pair $\tilde{q}_3(0)=\tilde{q}_4(0)$ is on the left hand side of $\tilde{q}_1(0)$. By Theorem \ref{keplerestimate}, if $\pi- \alpha_1 + \theta \leq \pi$, $\alpha_2 + \theta- \pi  \leq \pi$ and $\pi -\alpha_2+  \theta \leq \pi$,
\begin{align*}
 \mathcal{A}_{13} \geq \frac{3}{2}16^{\frac{1}{3}} \left(\alpha_1 + \theta-\pi \right)^{\frac{2}{3}} ,  & \qquad  \,\,    \mathcal{A}_{24} \geq \frac{3}{2}16^{\frac{1}{3}} \left(  \pi- \alpha_1 + \theta\right)^{\frac{2}{3}},    \nonumber \\
 \mathcal{A}_{23} \geq \frac{3}{2}16^{\frac{1}{3}} \left( \alpha_2 + \theta- \pi \right)^{\frac{2}{3}},   & \qquad  \,\,    \mathcal{A}_{14} \geq \frac{3}{2}16^{\frac{1}{3}} \left( \pi -\alpha_2+  \theta\right)^{\frac{2}{3}}. & 
\end{align*}
 It follows that 
 \begin{eqnarray}\label{a13a14a23a24case43}
 & & \mathcal{A}_{13} +\mathcal{A}_{14} + \mathcal{A}_{23} +\mathcal{A}_{24}  \geq 3 (16)^{\frac{1}{3}} \left(2 \theta \right)^{\frac{2}{3}}.
\end{eqnarray}
If $\pi- \alpha_1 + \theta > \pi$,
\[\mathcal{A}_{13}+  \mathcal{A}_{24} \geq  \frac{3}{2} (16)^{\frac{1}{3}}\left(\pi- 2 \theta \right)^{\frac{2}{3}} >\frac{3}{2} (16)^{\frac{1}{3}}\left(2 \theta \right)^{\frac{2}{3}}.  \]
If $\alpha_2 + \theta- \pi  > \pi$ or $\pi -\alpha_2+  \theta > \pi$,
\[  \mathcal{A}_{23}+  \mathcal{A}_{14} \geq \frac{3}{2} (16)^{\frac{1}{3}}\left(\pi- 2 \theta \right)^{\frac{2}{3}} >   \frac{3}{2} (16)^{\frac{1}{3}}\left(2 \theta \right)^{\frac{2}{3}}.  \]
Hence, inequality \eqref{a13a14a23a24case43} holds for any $\alpha_1 \in [0, \pi]$ and $\alpha_2 \in [0, 2\pi)$.
\end{enumerate}
By the analysis in the three subcases ((i) to (iii)), it follows that 
 \begin{eqnarray}\label{a13a14a23a24}
 & & \mathcal{A}_{13} +\mathcal{A}_{14} + \mathcal{A}_{23} +\mathcal{A}_{24}  \geq 3 (16)^{\frac{1}{3}} \left(2 \theta \right)^{\frac{2}{3}}.
\end{eqnarray}
Therefore, by inequalities \eqref{est12incase4}, \eqref{a34est3pq2} and \eqref{a13a14a23a24}, the action $\mathcal{A}$ satisfies
\begin{equation}\label{est4pq2}
\mathcal{A} \geq  \frac{3}{8}16^{\frac{1}{3}} \left[ \pi^{\frac{2}{3}}+ \left(\frac{\pi}{2}+\theta \right)^{\frac{2}{3}}+  2\left(2 \theta \right)^{\frac{2}{3}}\right].
\end{equation}

By the discussions above, for $\theta \in (0, \frac{\pi}{10}]$, if the minimizing path $\mathcal{P}_{Q_2}([0,1])$ has boundary collisions, the estimates in the 4 cases (\eqref{est1pq2}, \eqref{est2pq2}, \eqref{est3pq2}, \eqref{est4pq2}) imply that 
\begin{equation}\label{actionboundpq2}
\mathcal{A} \geq  \frac{3}{8}16^{\frac{1}{3}} \left[ \pi^{\frac{2}{3}}+ \theta^{\frac{2}{3}} + 2\left(2 \theta \right)^{\frac{2}{3}} \right].
\end{equation}
The proof is complete. 
 \end{proof}

\section{Definition of test paths}\label{testpath}
In this section, we define new test paths connecting $Q_S$ and $Q_{E_i} \, (i=1,2)$. In order to apply the level estimate method, we need to find test paths which have action values strictly less than the lower bound of action in Lemma \ref{lowerbdd1} or Lemma \ref{lowerbdd2}. The main idea is to use piecewise smooth linear functions which are linear approximations of the action minimizers $\mathcal{P}_{Q_1}$ and $\mathcal{P}_{Q_2}$. The test paths are introduced in the following subsections. 

\subsection{test paths connecting $Q_S$ and $Q_{E_1}$}
Note that
\begin{equation}\label{boundarypq1test}
Q_s=\begin{bmatrix}
-a_1-c_1 & 0 \\
-a_1 &   0\\
(2a_1+c_1)/2 & b_1 \\
(2a_1+c_1)/2 & -b_1 
\end{bmatrix}, \, \,  Q_{e_1}=\begin{bmatrix}
-b_2 & -a_2\\
-c_2 &   a_2\\
c_2 & a_2 \\
b_2 & -a_2
\end{bmatrix}R(\theta),
\end{equation}
where $a_1, a_2 \in \mathbb{R}$ and $b_1, b_2, c_1, c_2 \geq 0$. Recall that for each given $\theta \in (0, \pi/10]$, $Q_S$ and $Q_{E_1}$ are defined to be the boundary configuration sets:
\begin{equation}\label{QS}
 Q_S= \left\{ Q_s \, \bigg| \, a_1 \in \mathbb{R}, \, b_1 \geq 0, \, c_1 \geq 0   \right\},
\end{equation}
\begin{equation}\label{QE}
 Q_{E_1}= \left\{ Q_{e_1} \, \bigg| \, a_2 \in \mathbb{R}, \, b_2 \geq 0, \, c_2 \geq 0   \right\},
\end{equation}
where $Q_s$ and $Q_{e_1}$ are defined in \eqref{boundarypq1test}. We set $P(Q_S, Q_{E_1})$ to be the set of paths in $H^1([0,1], \chi)$ which have boundaries in $Q_S$ and $Q_{E_1}$:
\[ P(Q_S, Q_{E_1}):=\left\{q(t) \in H^1([0,1], \chi) \, \bigg| \, q(0) \in  Q_S,  \, q(1) \in Q_{E_1} \right\}.  \]
For given $\theta \in (0, \pi/10]$, our goal is to define some test path $\mathcal{P}_{test} \in P(Q_S, Q_{E_1})$, such that its action $\mathcal{A}_{test}= \mathcal{A}(\mathcal{P}_{test})$ is strictly less than the lower bound \eqref{estlowerbddpq1} in Lemma \ref{lowerbdd1}: 
\[\mathcal{A}_{col} \geq \frac{3}{8} 16^{\frac{1}{3}} \left[  \left( 2 \pi^2\right)^{\frac{1}{3}} + 2 (2 \theta)^{\frac{2}{3}} \right]> \mathcal{A}(\mathcal{P}_{test}),  \] 
where $\mathcal{A}_{col}$ denotes the action of the minimizing path $\mathcal{P}_{Q_1}$ with boundary collisions. For convenience, we set
\begin{equation}\label{galowbd}
g_1(\theta)= \frac{3}{8} 16^{\frac{1}{3}} \left[  \left( 2 \pi^2\right)^{\frac{1}{3}} + 2 (2 \theta)^{\frac{2}{3}} \right]. 
\end{equation}
The following lemma holds for small $\theta$:
\begin{lemma}\label{testpathdef1}
For each $\theta \in (0, 0.0539 \pi]$, there exists a test path $\mathcal{P}_{test} \in P(Q_S, Q_{E_1})$, such that its action $\mathcal{A}_{test}= \mathcal{A}(\mathcal{P}_{test})$ satisfies
\[ \mathcal{A}_{test}= \mathcal{A}(\mathcal{P}_{test})< g_1(\theta).  \]
\end{lemma}
\begin{proof}
Let $\bar{q}(t) \, (t \in [0,1])$ be the position matrix path of $\mathcal{P}_{test}$. The test path $\mathcal{P}_{test}$ is defined in two steps. 

First, we define a test path for a fixed angle $\theta=\theta_0$. At time $t=t_j=\frac{j}{10}\, (j=0,1,2, \dots, 10)$, we can find the position matrices $q(\frac{j}{10})$ in the minimizer $\mathcal{P}_{Q_1} \equiv \mathcal{P}_{Q_1, \theta_0}$. In $\mathcal{P}_{test}$, we set
\[ \bar{q}(t_j)=\bar{q}(\frac{j}{10}) = q(\frac{j}{10}). \]
 For $t \in [t_j, t_{j+1}]$, the path will be the linear connection between $\bar{q}(t_j)$ and $\bar{q}(t_{j+1})\, (j=0,1,2, \dots, 9)$. Furthermore, we assume the bodies move at constant speeds in the time interval $[ \frac{j}{10}, \, \frac{j+1}{10} ]\, (j=0,1,2, \dots, 9)$. That is, for each $j\, (j=0,1,2, \dots, 9)$, 
 \begin{equation}\label{qbardef}
  \bar{q}(t) = \bar{q}(t_j) + 10\left(t- \frac{j}{10} \right)\left[\bar{q}(t_{j+1})- \bar{q}(t_j) \right], \quad  t \in \left[ \frac{j}{10}, \, \frac{j+1}{10} \right]. 
 \end{equation}
Once the 11 matrices $q(t_j)=q(\frac{j}{10}) \, (i=0,1,2, \dots, 10)$ are given, the action value $\mathcal{A}_{test}= \mathcal{A}(\mathcal{P}_{test})$ can be calculated. By its definition, $\mathcal{A}_{test}$ should be close to the minimum action $\mathcal{A}\left(\mathcal{P}_{Q_1} \right)$. Note that in \eqref{galowbd}, $g_1(\theta)$ is the lower bound of action of the minimizer $\mathcal{P}_{Q_1}$ with boundary collisions. If 
\[ \mathcal{A}_{test} < g_1(\theta_0),\] 
it follows that the minimizer $\mathcal{P}_{Q_1}=\mathcal{P}_{Q_1, \theta_0}$ has no collision singularity. 

The next step is to define a test path for an interval of $\theta_0$. Assume at $\theta=\theta_0$, the test path defined in the first step satisfies $\mathcal{A}_{test} < g_1(\theta_0)$. Then it is reasonable to expect in a small interval of $\theta_0$, one can perturb the test path $\mathcal{P}_{test}=\mathcal{P}_{test, \, \theta_0}$ defined for $\theta=\theta_0$, such that the inequality $\mathcal{A}_{test} < g_1(\theta)$ still holds. In this paper, the perturbed path is defined as follows.

We fix $\mathcal{P}_{test} (t \in [0, \frac{9}{10}])$. For $t \in [ \frac{9}{10}, 1]$, we perturb the last point $\bar{q}(t_{10})= \bar{q}(1)$ in $\mathcal{P}_{test}= \mathcal{P}_{test, \theta_0}$ such that it satisfies the boundary condition 
\[ \bar{q}(1) \in Q_{E_1}= \left\{Q_{e_1} \, \bigg|  \, a_2 \in \mathbb{R}, \, b_2 \geq 0, \,  c_2 \geq 0  \right\}.\] 
Then we connect $\bar{q}(1)$ with $\bar{q}(t_9)=\bar{q}(\frac{9}{10})=q(\frac{9}{10})$ by straight lines. We set $\bar{q}(1) = \begin{bmatrix}
\bar{q}_{1,10} \\
\bar{q}_{2,10} \\
\bar{q}_{3,10} \\
\bar{q}_{4,10} 
\end{bmatrix}$ and 
\begin{align*}
\bar{q}_{1,10} &= \left[ -a_{21}, \,  -b_{21} \right]R(\theta)= \left[  -a_{21} \cos \theta+ b_{21} \sin \theta,   \, -a_{21} \sin \theta - b_{21} \cos \theta\right], \\
\bar{q}_{2,10} &= \left[ -c_{21}, \, b_{21} \right]R(\theta)= \left[  -c_{21} \cos \theta - b_{21} \sin \theta,    \,  -c_{21} \sin \theta+ b_{21} \cos \theta\right], \\
\bar{q}_{3,10} &= \left[ c_{21}, \, b_{21} \right]R(\theta)= \left[  c_{21} \cos \theta- b_{21} \sin \theta,    \,  c_{21} \sin \theta+ b_{21} \cos \theta\right], \\
\bar{q}_{4,10} &= \left[ a_{21}, \,  -b_{21} \right]R(\theta)= \left[  a_{21} \cos \theta+ b_{21} \sin \theta,   \, a_{21} \sin \theta - b_{21} \cos \theta\right], 
\end{align*}
where $a_{21}, b_{21}, c_{21}$ are values in $q(1)= \begin{bmatrix}
-b_{21} & -a_{21} \\
-c_{21} &   a_{21}\\
c_{21} & a_{21} \\
b_{21} & -a_{21} 
\end{bmatrix}R(\theta_0)$ of the minimizer $\mathcal{P}_{Q_1}$ at given $\theta=\theta_0$, and $\theta$ is in a small interval of $\theta_0$. It is clear that  for each $\theta$, $\bar{q}(1) \in  Q_{E_1}$. The perturbed path will have action smaller than $g_1(\theta)$ in a small interval of $\theta_0$. Actually, this process can be repeated for different values of $\theta_0$ and eventually we can define test paths for all possible $\theta$. In what follows, we explain how to calculate the action of a test path $\mathcal{A}_{test}=\mathcal{A}(\mathcal{P}_{test})$.

Let $\mathcal{A}_{j+1}$ be the action of the linear path connecting $\bar{q}(\frac{j}{10})$ and $\bar{q}(\frac{j+1}{10}), \, (j=0,1,2, \dots, 9).$ Then the action $\mathcal{A}_{test}=\mathcal{A}(\mathcal{P}_{test})$ satisfies 
\[\mathcal{A}_{test}=\mathcal{A}(\mathcal{P}_{test})=\sum_{j=0}^9 \mathcal{A}_{j+1}. \]
Indeed, one can directly calculate each action $\mathcal{A}_{j+1} \, (j=0,1,2, \dots, 9)$, which are determined by the given 11 position matrices $\bar{q}(t_k) \, (k=0, 1,2, \dots, 10)$. For $t \in [ \frac{j}{10}, \frac{j+1}{10}]$, let $K_{j+1}$ be the corresponding kinetic energy and $U_{j+1}$ be the potential energy. It follows that 
\[ \mathcal{A}_{j+1}= \int_{\frac{j}{10}}^{\frac{j+1}{10}} K_{j+1} \, dt +   \int_{\frac{j}{10}}^{\frac{j+1}{10}} U_{j+1} \, dt,  \quad j=0,1,2, \dots, 9. \]
Note that the linear path has a constant velocity in the time $t \in [\frac{j}{10}, \frac{j+1}{10}]$. By \eqref{qbardef}, it follows that the kinetic energy is
\[K_{j+1}=   \frac{1}{2} \sum_{i=1}^4 |\dot{\bar{q}}_{i}(t)|^2= 50 \sum_{i=1}^4 \big|\bar{q}_{i,(j+1)} -  \bar{q}_{i, j} \big|^2. \]
It implies that 
\begin{equation}\label{formulaofk}
\int_{\frac{j}{10}}^{\frac{j+1}{10}}  K_{j+1} \, dt = 5 \sum_{i=1}^4 \big|\bar{q}_{i,(j+1)} -  \bar{q}_{i, j} \big|^2, \quad j=0,1,2, \dots, 9.
\end{equation}
The potential energy is 
\[U_{j+1} = \sum_{1 \leq i<k \leq 4}\frac{1}{|\bar{q}_{i}(t) -\bar{q}_{k}(t) | } , \qquad t \in [\frac{j}{10}, \frac{j+1}{10}],  \quad j=0,1,2, \dots, 9.   \]
where $\bar{q}_{i}(t) \, (i=1,2,3,4)$ is defined by \eqref{qbardef}. Let $u=10(t-\frac{j}{10})$. Then
\begin{eqnarray}\label{intofpotential}
& & \int_{\frac{j}{10}}^{\frac{j+1}{10}}   \frac{1}{|\bar{q}_{i}(t) -\bar{q}_{k}(t) | }  \, dt  \nonumber\\
&=& \frac{1}{10}\int_{0}^{1}   \frac{1}{ |\bar{q}_{i, j}- \bar{q}_{k, j} +  u\left( \bar{q}_{i, (j+1)} - \bar{q}_{i, j} + \bar{q}_{k, j} -\bar{q}_{k, {j+1}}    \right)| }   \, du. 
\end{eqnarray}
Note that the integral in \eqref{intofpotential} always has the form $\displaystyle \int_{0}^1   \frac{du}{\sqrt{(a+b u)^2 + (c+d u)^2}}$ and it can be integrated as follows:
\begin{eqnarray}\label{generalformulaint}
& & \int_{0}^1   \frac{1}{\sqrt{(a+b u)^2 + (c+d u)^2}} \, du  \nonumber\\
&=& \frac{1}{\sqrt{b^2+d^2}} \ln \left[  \frac{ab+cd}{b^2+d^2} +u + \sqrt{\frac{(a+bu)^2+(c+du)^2 }{b^2+d^2}} \right] \Bigg|_0^1
\end{eqnarray}
Hence, by \eqref{formulaofk}, \eqref{intofpotential} and \eqref{generalformulaint}, once the coordinates of $\bar{q}_{i, j} =  \bar{q}_{i}(\frac{j}{10})\, (i=1,2,3,4; j=0,1,2, \dots, 10)$ are given,  we can find the exact value of $\mathcal{A}_{j+1} \, (j=0,1, 2, \dots, 9)$. Therefore, the action $\displaystyle \mathcal{A}(\mathcal{P}_{test}) = \sum_{j=0}^9 \mathcal{A}_{j+1} $  can be calculated accurately. It is clear that $\mathcal{A}(\mathcal{P}_{test})$ will be a function of $\theta$ and to exclude possible collisions in $\mathcal{P}_{Q_1}$, we need to prove that for $\theta$ in a certain interval of $\theta_0$, the following inequality holds:
\begin{equation}\label{testpathdataineq}
   \mathcal{A}(\mathcal{P}_{test}) < g_1(\theta)= \frac{3}{8} 16^{\frac{1}{3}} \left[  \left( 2 \pi^2\right)^{\frac{1}{3}} + 2 (2 \theta)^{\frac{2}{3}} \right]. 
\end{equation}

In order to make inequality \eqref{testpathdataineq} true, we will choose 7 different values of $\theta_0$ and define the test paths in each interval of $\theta_0$. For given $\theta_0$, test paths $\mathcal{P}_{test}= \mathcal{P}_{test, \, \theta, \, \theta_0}$ can be determined by the values of $q(t_j)= q(\frac{j}{10}) \, (j=0,1,2, \dots, 10)$ in the minimizing path $\mathcal{P}_{Q_1}$ corresponding to $\theta_0$. It is clear that
\[\bar{q}(0)  =\begin{bmatrix}
\bar{q}_{1} (0)\\
\bar{q}_{2} (0)\\
\bar{q}_{3} (0) \\
\bar{q}_{4}(0) 
\end{bmatrix} \in Q_S=\left\{Q_s \, \bigg|\,  a_1 \in \mathbb{R}, \, b_1 \geq 0, \, c_1 \geq 0  \right\},  \]
\[\bar{q}(1)  =\begin{bmatrix}
\bar{q}_{1} (1)\\
\bar{q}_{2} (1)\\
\bar{q}_{3} (1) \\
\bar{q}_{4}(1) 
\end{bmatrix}\in Q_{E_1}=\left\{Q_{e_1} \, \bigg| \, a_2 \in \mathbb{R}, \, b_2 \geq 0, \, c_2 \geq 0  \right\}, \] where $Q_s$ and $Q_{e_1}$ are defined by \eqref{boundarypq1test}. 

In the following tables (from Table \ref{table1} to Table \ref{table7}), the position coordinates of $\bar{q}_{i, j} =\bar{q}_{i}(\frac{j}{10}) \,  (i=1,2,3; j=0,1,2, \dots, 10)$ of $\mathcal{P}_{test}= \mathcal{P}_{test, \, \theta, \, \theta_0}$ are given, where the position $\bar{q}_{4, j}\, (j=0,1,2, \dots, 10)$ satisfies
\[\bar{q}_{4, j}=-\bar{q}_{1, j}-\bar{q}_{2, j}-\bar{q}_{3, j}.\] 
To cover the interval $\theta \in (0, \, 0.0539 \pi]$, we take 7 different values of $\theta_0$. For each interval of $\theta_0$, we define test paths $\mathcal{P}_{test, \, \theta, \, \theta_0}$ by linear connections of the adjacent points $\bar{q}_{i, j}$ and $\bar{q}_{i, (j+1)}$ \, $(i= 1,2,3,4; j=0,1, 2, \dots, 9)$ in each table (from Table \ref{table1} to Table \ref{table7}) with constant speeds.
\begin{enumerate}[(i)]
\item $\theta_0= 0.0539 \pi$ : $\text{it works for} \, \, \theta \in [0.052 \pi, \, 0.0539 \pi]$, in which $\bar{q}_{i, j} \, (i=1,2,3; j=0,1,2, \dots ,10)$ of the test paths are given by Table \ref{table1};
\begin{table}[!htbp]
\centering
\tiny
\begin{tabular}{ |c|c|c|c|c|} 
 \hline
\multicolumn{4}{|c|}{$\theta_0= 0.0539 \pi, \quad \quad \quad \theta \in [0.052 \pi, \, 0.0539 \pi]$} \\
\hline
$t$ &  $\bar{q}_1$ &$\bar{q}_2$    &   $\bar{q}_3$ \\
  \hline 
  $0$ &  $(-3.2373, \,   0)$ &  $(-1.9803, \,    0)$ & $(2.6088, \, 0.6679)$ \\
  \hline
  $0.1$ & $(-3.2338518, \,  -0.10869889)$  &  $ (-1.9830012, \, 0.020343213)$ & $ (2.5493669, \,  0.70922736)$  \\
 \hline
  $0.2$ & $( -3.2235450, \,  -0.21673333)$  &  $( -1.9910668, \, 0.040047898)$  & $(2.4896691, \, 0.74485905)$  \\
   \hline
  $0.3$ & $( -3.2064926, \,  -0.32344983)$  & $(-2.0043844, \, 0.058486424)$  & $( 2.4301885, \,  0.077483714)$ \\
   \hline
  $0.4$ &  $(-3.1828790, \, -0.42821628)$  & $(-2.0227700, \, 0.075052441 ) $  &    $( 2.3714050, \, 0.079924115 ) $\\
     \hline
   $0.5$ & $( -3.1529557, \,  -0.53043141)$ &  $ ( -2.0459733, \, 0.089170282) $  &  $ (2.3137956, \, 0.81818913) $  \\
   \hline
  $0.6$ & $(-3.1170343, \,  -0.62953290)$ &  $  (-2.0736839, \,  0.10030302)$ & $(2.2578326, \, 0.83183907)$   \\
     \hline
   $0.7$ & $( -3.0754793, \,  -0.72500378)$ & $ ( -2.1055390, \, 0.10795890) $ &  $ ( 2.2039811, \, 0.84039042)$ \\
      \hline 
   $0.8$ &   $(-3.0286998, \, -0.81637729)$ &  $(-2.1411316, \, 0.11169614)  $ &   $ ( 2.1526963, \, 0.84408580  )$\\
      \hline
   $0.9$ & $( -2.9771413, \, -0.90323980)$ &  $(  -2.1800191, \, 0.11112592)  $  &  $(2.1044192, \, 0.84321242)$ \\
      \hline
  $1$ &  $(-3.0455313 ,\,  -0.47883639) R(\theta)$  & $ ( -2.1721057, \, 0.47883639)R(\theta)$ &   $( 2.1721057,\, 0.47883639)R(\theta)$  \\
 \hline
\end{tabular}
\caption{\label{table1} The positions of $\bar{q}_{i, j}=\bar{q}_{i}(\frac{j}{10}) \, (i=1,2,3, \, j=0,1,2, \dots, 10)$ in the path $\mathcal{P}_{test}= \mathcal{P}_{test, \, \theta, \, \theta_0}$ corresponding to $\theta  \in [0.052 \pi, \, 0.0539 \pi]$.}
\end{table}

\item $\theta_0= 0.05 \pi$: $\text{it works for} \, \, \theta \in [0.0425 \pi, \, 0.052 \pi ]$, in which $\bar{q}_{i, j} \, (i=1,2,3; j=0,1,2, \dots, 10)$ of the test paths are given by Table \ref{table2};
\begin{table}[!htbp]
\centering
\tiny
\begin{tabular}{ |c|c|c|c|c|} 
 \hline
\multicolumn{4}{|c|}{$\theta_0= 0.05 \pi, \quad  \quad \quad \theta \in [0.0425 \pi, \, 0.052 \pi]$} \\
\hline
$t$ &  $\bar{q}_1$ &$\bar{q}_2$    &   $\bar{q}_3$ \\
  \hline 
  $0$ &  $(  -3.3798, \, 0)$ &  $(  -2.1048, \,    0)$ & $( 2.7423, \,   0.6707)$ \\
  \hline
  $0.1$ & $(-3.3764651, \,  -0.10694486)$  &  $ (  -2.1074592, \,   0.020788486)$ & $ ( 2.6827420, \,  0.71095798)$  \\
 \hline
  $0.2$ & $(  -3.3664949, \,  -0.21326080)$  &  $( -2.1154021, \,    0.040969780  )$  & $( 2.6229874, \,  0.74558327 )$  \\
   \hline
  $0.3$ & $(   -3.3499926, \,  -0.31832838)$  & $( -2.1285260, \,  0.059946130  )$  & $( 2.5635141, \,   0.77462096 )$ \\
   \hline
  $0.4$ &  $(  -3.3271269, \,   -0.42154673)$  & $(  -2.1466626, \,   0.077138259 ) $  &    $( 2.5047976, \,  0.79815402 ) $\\
     \hline
   $0.5$ & $( -3.2981286, \,  -0.52234186)$ &  $ ( -2.1695819, \,    0.091993643  ) $  &  $ ( 2.4473089, \,   0.81630412) $  \\
   \hline
  $0.6$ & $(  -3.2632845, \,   -0.62017390)$ &  $  (  -2.1969978, \,   0.10399373  )$ & $( 2.3915131, \,  0.82923263 )$   \\
     \hline
   $0.7$ & $( -3.2229319, \,  -0.71454308 )$ & $ (  -2.2285745, \,   0.11265991 ) $ &  $ ( 2.3378663, \,  0.83714174 )$ \\
      \hline 
   $0.8$ &   $(  -3.1774514, \, -0.80499430)$ &  $( -2.2639331, \,   0.11755805   )  $ &   $ ( 2.2868129, \, 0.84027555 )$\\
      \hline
   $0.9$ & $(  -3.1272594, \,  -0.89112033 )$ &  $( -2.3026591, \,   0.11830176 )  $  &  $(  2.2387817, \, 0.83892102 )$ \\
      \hline
  $1$ &  $( -3.1871131, \,   -0.47989766 ) R(\theta)$  & $ (  -2.2975267, \,  0.47989766 )R(\theta)$ &   $(2.2975267, \,  0.47989766 )R(\theta)$  \\
 \hline
\end{tabular}
\caption{\label{table2} The positions of $\bar{q}_{i, j}=\bar{q}_{i}(\frac{j}{10}) \, (i=1,2,3, \, j=0,1,2, \dots, 10)$ in the path $\mathcal{P}_{test}= \mathcal{P}_{test, \, \theta, \, \theta_0}$ corresponding to $\theta  \in [0.0425 \pi, \, 0.052 \pi]$.}
\end{table}

\item $\theta_0= 0.04 \pi$: $\text{it works for} \,  \,  \theta \in [0.032 \pi, \, 0.0425 \pi]$, in which $\bar{q}_{i, j} \, (i=1,2,3; j=0,1,2, \dots, 10)$ of the test paths are given by Table \ref{table3};
\begin{table}[!htbp]
\centering
\tiny
\begin{tabular}{ |c|c|c|c|c|} 
 \hline
\multicolumn{4}{|c|}{$\theta_0= 0.04 \pi, \quad \quad  \quad \theta \in [0.032 \pi, \, 0.0425 \pi]$} \\
\hline
$t$ &  $\bar{q}_1$ &$\bar{q}_2$    &   $\bar{q}_3$ \\
  \hline 
  $0$ &  $(  -3.8384, \, 0)$ &  $(-2.5204, \, 0 )$ & $( 3.1794, \, 0.6797)$ \\
  \hline
  $0.1$ & $(-3.8353238, \,   -0.10228769)$  &  $ ( -2.5229746, \, 0.022379837 )$ & $ (3.1196968, \,  0.71691715)$  \\
 \hline
  $0.2$ & $(-3.8261229, \,   -0.20402287 )$  &  $( -2.5306708, \,   0.044220020  )$  & $(3.0599576, \, 0.74868311 )$  \\
   \hline
  $0.3$ & $( -3.8108803, \,   -0.30465966)$  & $( -2.5434057, \, 0.064987499 )$  & $( 3.0006455, \,  0.77504804 )$ \\
   \hline
  $0.4$ &  $(  -3.7897324, \,  -0.40366521)$  & $( -2.5610431, \,   0.084162235 ) $  &    $( 2.9422200, \, 0.79610002) $\\
     \hline
   $0.5$ & $( -3.7628669, \,  -0.50052579)$ &  $ (-2.5833956, \, 0.10124321 ) $  &  $ (2.8851338, \, 0.81196531) $  \\
   \hline
  $0.6$ & $(-3.7305193, \,   -0.59475220)$ &  $  (  -2.6102281, \, 0.11575387 )$ & $( 2.8298306, \, 0.82280854)$   \\
     \hline
   $0.7$ & $( -3.6929690, \,   -0.68588461)$ & $ ( -2.6412617, \, 0.12724697) $ &  $ (  2.7767426, \,  0.82883291)$ \\
      \hline 
   $0.8$ &   $( -3.6505354, \,  -0.77349655 )$ &  $( -2.6761780, \,   0.13530853 )  $ &   $ ( 2.7262867, \, 0.83028019 )$\\
      \hline
   $0.9$ & $(-3.6035724, \,   -0.85719819)$ &  $(-2.7146239, \, 0.13956111)  $  &  $( 2.6788621, \, 0.82743053 )$ \\
      \hline
  $1$ &  $(  -3.6418439, \, -0.4.8401120 ) R(\theta)$  & $ ( -2.7169781, \, 0.48401120)R(\theta)$ &   $( 2.7169781, \, 0.48401120)R(\theta)$  \\
 \hline
\end{tabular}
\caption{\label{table3} The positions of $\bar{q}_{i, j}=\bar{q}_{i}(\frac{j}{10}) \, (i=1,2,3, \, j=0,1,2, \dots, 10)$ in the path $\mathcal{P}_{test}= \mathcal{P}_{test, \, \theta, \, \theta_0}$ corresponding to $\theta  \in [0.032 \pi, \, 0.0425 \pi]$.}
\end{table}

\item $\theta_0= 0.03 \pi$: $\text{it works for} \, \, \theta \in [0.021 \pi, \, 0.032 \pi ]$, in which $\bar{q}_{i, j} \, (i=1,2,3; j=0,1,2, \dots, 10)$ of the test paths are given by Table \ref{table4}; 
\begin{table}[!htbp]
\centering
\tiny
\begin{tabular}{ |c|c|c|c|c|} 
 \hline
\multicolumn{4}{|c|}{$\theta_0= 0.03 \pi, \quad \quad  \quad \theta \in [0.021\pi, \, 0.032 \pi]$} \\
\hline
$t$ &  $\bar{q}_1$ &$\bar{q}_2$    &   $\bar{q}_3$ \\
  \hline 
  $0$ &  $(  -4.5268, \, 0)$ &  $(  -3.1670, \, 0 )$ & $(  3.8469, \, 0.6913 )$ \\
  \hline
  $0.1$ & $( -4.5239567, \,  -0.097278221)$  &  $ ( -3.1695018, \, 0.024759311)$ & $ ( 3.7872935, \, 0.72492965 )$  \\
 \hline
  $0.2$ & $(  -4.5154494, \,  -0.19406672 )$  &  $( -3.1769846, \, 0.049035431)$  & $(3.7277905, \, 0.75331483 )$  \\
   \hline
  $0.3$ & $(-4.5013463, \,  -0.28988053)$  & $( -3.1893802, \,  0.072349921)$  & $(3.6688335, \, 0.77650740)$ \\
   \hline
  $0.4$ &  $( -4.4817597, \, -0.38424409 )$  & $(-3.2065763, \, 0.094233735 ) $  &    $( 3.6108600, \, 0.79459590) $\\
     \hline
   $0.5$ & $( -4.4568450, \, -0.4.7669572 )$ &  $ ( -3.2284179, \, 0.11423167 ) $  &  $ ( 3.5542996, \, 0.80770530) $  \\
   \hline
  $0.6$ & $( -4.4267983, \,  -0.56679180 )$ &  $  ( -3.2547087, \, 0.13190656 )$ & $(3.4995719, \, 0.81599674)$   \\
     \hline
   $0.7$ & $( -4.3918545, \,   -0.6.5411055 )$ & $ (-3.2852143, \, 0.14684305) $ &  $ (3.4470829, \, 0.81966704)$ \\
      \hline 
   $0.8$ &   $(-4.3522838, \,  -0.73825552 )$ &  $( -3.3196646, \, 0.15865107 )  $ &   $ (3.3972229, \, 0.81894806)$\\
      \hline
   $0.9$ & $( -4.3083893, \,   -0.81885848)$ &  $(-3.3577570, \, 0.16696879)  $  &  $( 3.3503628, \, 0.81410582)$ \\
      \hline
  $1$ &  $(-4.3258766, \, -0.49065865 ) R(\theta)$  & $ (-3.3679379, \, 0.49065865)R(\theta)$ &   $( 3.3679379, \, 0.49065865)R(\theta)$  \\
 \hline
\end{tabular}
\caption{\label{table4} The positions of $\bar{q}_{i, j}=\bar{q}_{i}(\frac{j}{10}) \, (i=1,2,3, \, j=0,1,2, \dots, 10)$ in the path $\mathcal{P}_{test}= \mathcal{P}_{test, \, \theta, \, \theta_0}$ corresponding to $\theta  \in [0.021 \pi, \, 0.032 \pi]$.}
\end{table}

\item $\theta_0= 0.015 \pi$: $\text{it works for} \, \, \theta \in [0.008 \pi, \, 0.021 \pi ]$, in which $\bar{q}_{i, j} \, (i=1,2,3; j=0,1,2, \dots, 10)$ of the test paths are given by Table \ref{table5};
\begin{table}[!htbp]
\centering
\tiny
\begin{tabular}{ |c|c|c|c|c|} 
 \hline
\multicolumn{4}{|c|}{$\theta_0= 0.015 \pi, \quad \quad  \quad \theta \in [0.008 \pi, \, 0.021 \pi]$} \\
\hline
$t$ &  $\bar{q}_1$ &$\bar{q}_2$    &   $\bar{q}_3$ \\
  \hline 
  $0$ &  $(  -6.8044, \, 0 )$ &  $(-5.3824, \, 0)$ & $( 6.0934, \, 0.713)$ \\
  \hline
  $0.1$ & $( -6.8018689, \, -0.088070938 )$  &  $ ( -5.3847957, \, 0.030651152)$ & $ ( 6.03434, \, 0.73924844 )$  \\
 \hline
  $0.2$ & $( -6.7942932, \, -0.17572811)$  &  $( -5.3919655, \, 0.060889821)$  & $( 5.9755501, \, 0.76058997 )$  \\
   \hline
  $0.3$ & $(-6.7817250, \, -0.26256082)$  & $( -5.4038570, \, 0.090306594)$  & $(  5.9174332, \, 0.77707380)$ \\
   \hline
  $0.4$ &  $( -6.7642511, \, -0.34816448 )$  & $(-5.4203836, \,  0.11849815) $  &    $( 5.8603869, \, 0.78878204 ) $\\
     \hline
   $0.5$ & $( -6.7419916, \,  -0.43214354 )$ &  $ ( -5.4414251, \, 0.14507026) $  &  $ (5.8048010, \, 0.79582929) $  \\
   \hline
  $0.6$ & $( -6.7150996, \, -0.51411445)$ &  $  (-5.4668285, \, 0.16964059 )$ & $(5.7510549, \, 0.79836191)$   \\
     \hline
   $0.7$ & $( -6.6837594, \,  -0.59370828)$ & $ (-5.4964096, \, 0.19184153 ) $ &  $ (5.6995150, \, 0.79655724)$ \\
      \hline 
   $0.8$ &   $(-6.6481851, \,  -0.67057345)$ &  $(-5.5299541, \, 0.21132275)  $ &   $ ( 5.6505325, \, 0.79062253)$\\
      \hline
   $0.9$ & $( -6.6086194, \, -0.74437807)$ &  $( -5.5672195, \, 0.22775364)  $  &  $(5.6044406, \, 0.78079365)$ \\
      \hline
  $1$ &  $(-6.5964255, \, -0.50463822) R(\theta)$  & $ ( -5.5903673, \, 0.50463822)R(\theta)$ &   $( 5.5903673, \, 0.50463822)R(\theta)$  \\
 \hline
\end{tabular}
\caption{\label{table5} The positions of $\bar{q}_{i, j}=\bar{q}_{i}(\frac{j}{10}) \, (i=1,2,3, \, j=0,1,2, \dots, 10)$ in the path $\mathcal{P}_{test}= \mathcal{P}_{test, \, \theta, \, \theta_0}$ corresponding to $\theta  \in [0.008 \pi, \, 0.021 \pi]$.}
\end{table}

\item $\theta_0= 0.006 \pi$: $\text{it works for}  \, \, \theta \in [0.002 \pi, \, 0.008 \pi ]$, in which $\bar{q}_{i, j} \, (i=1,2,3; j=0,1,2, \dots, 10)$ of the test paths are given by Table \ref{table6};
\begin{table}[!htbp]
\centering
\tiny
\begin{tabular}{ |c|c|c|c|c|} 
 \hline
\multicolumn{4}{|c|}{$\theta_0= 0.006 \pi, \quad \quad  \quad \theta \in [0.002 \pi, \, 0.008 \pi]$} \\
\hline
$t$ &  $\bar{q}_1$ &$\bar{q}_2$    &   $\bar{q}_3$ \\
  \hline 
  $0$ &  $( -11.9409, \, 0)$ &  $(-10.4845, \, 0)$ & $(11.2127, \, 0.7286)$ \\
  \hline
  $0.1$ & $( -11.938525, \, -0.079676751)$  &  $ (-10.486835, \, 0.037416776)$ & $ (11.154183, \, 0.74737592)$  \\
 \hline
  $0.2$ & $(-11.931417, \, -0.15897450)$  &  $( -10.493824, \, 0.074454696 )$  & $( 11.096004, \, 0.76145882)$  \\
   \hline
  $0.3$ & $(-11.919620, \, -0.23751671)$  & $( -10.505421, \,0.11073738)$  & $(11.038538, \, 0.77089404)$ \\
   \hline
  $0.4$ &  $(-11.903211, \, -0.31493179)$  & $( -10.521552, \, 0.14589340) $  &    $(10.982156, \, 0.77575689) $\\
     \hline
   $0.5$ & $(-11.882295, \, -0.39085553)$ &  $ (-10.542109, \, 0.17955866) $  &  $ ( 10.927221, \, 0.77615218 ) $  \\
   \hline
  $0.6$ & $( -11.857007, \,  -0.46493341)$ &  $  (-10.566959, \, 0.21137882)$ & $(10.874086, \, 0.77221354)$   \\
     \hline
   $0.7$ & $( -11.827509, \, -0.53682297)$ & $ (-10.595939, \, 0.24101157) $ &  $ (10.823093, \,    0.76410258)$ \\
      \hline 
   $0.8$ &   $(  -11.793991, \, -0.60619603)$ &  $( -10.628860, \, 0.26812885)  $ &   $ ( 10.774571, \, 0.75200789 )$\\
      \hline
   $0.9$ & $(  -11.756667, \, -0.67274074)$ &  $(-10.665506, \, 0.29241900)  $  &  $( 10.728831, \,  0.73614379 )$ \\
      \hline
  $1$ &  $(  -11.727573, \,  -0.51520877 ) R(\theta)$  & $ (-10.697826, \, 0.51520877)R(\theta)$ &   $(10.697826, \, 0.51520877 )R(\theta)$  \\
 \hline
\end{tabular}
\caption{\label{table6} The positions of $\bar{q}_{i, j}=\bar{q}_{i}(\frac{j}{10}) \, (i=1,2,3, \, j=0,1,2, \dots, 10)$ in the path $\mathcal{P}_{test}= \mathcal{P}_{test, \, \theta, \, \theta_0}$ corresponding to $\theta  \in [0.002 \pi, \, 0.008 \pi]$.}
\end{table}

\item  $\theta_0= 0.002 \pi$: $\text{it works for} \, \, \theta \in (0, \, 0.002 \pi ]$, in which $\bar{q}_{i, j}$ \, $(i=1,2,3; j=0,1,2, \dots, 10)$ of the test paths are given by Table \ref{table7}.

\begin{table}[!htbp]
\centering
\tiny
\begin{tabular}{ |c|c|c|c|c|} 
 \hline
\multicolumn{4}{|c|}{$\theta_0= 0.002 \pi, \quad \quad  \quad \theta \in (0, \, 0.002 \pi]$} \\
\hline
$t$ &  $\bar{q}_1$ &$\bar{q}_2$    &   $\bar{q}_3$ \\
  \hline 
  $0$ &  $(-22.1946, \, 0)$ &  $(  -20.6762, \, 0 )$ & $(21.4354, \, 0.7313 )$ \\
  \hline
  $0.1$ & $(-22.191384, \, -0.070385616)$  &  $ (-20.680105, \, 0.043445618)$ & $ (21.375409, \, 0.74215348)$  \\
 \hline
  $0.2$ & $(-22.183819, \,  -0.14044405 )$  &  $(-20.688337, \,  0.086564071)$  & $( 21.315793, \, 0.74834948)$  \\
   \hline
  $0.3$ & $( -22.171924, \, -0.20984650)$  & $( -20.700878, \, 0.12902656)$  & $( 21.256933, \, 0.74993810)$ \\
   \hline
  $0.4$ &  $( -22.155744, \, -0.27826393)$  & $(-20.717682, \, 0.17050407) $  &    $( 21.199207, \, 0.74700452) $ \\
     \hline
   $0.5$ & $(-22.135350, \, -0.34536891 )$ &  $ (-20.738678, \, 0.21066919) $  &  $ ( 21.142977, \, 0.73966795) $  \\
   \hline
  $0.6$ & $(-22.110836, \, -0.41083750) $ &  $  (-20.763772, \, 0.24919799)$ & $(21.088596, \, 0.72807994 )$   \\
     \hline
   $0.7$ & $(-22.082322, \, -0.47435123)$ & $ (-20.792844, \,  0.28577201) $ &  $ (21.036396, \, 0.71242237)$ \\
      \hline 
   $0.8$ &   $(-22.049956, \, -0.53559912)$ &  $(-20.825747, \,  0.32008029)  $ &   $ ( 20.986688, \, 0.69290503)$\\
      \hline
   $0.9$ & $(-22.013907, \, -0.59427981)$ &  $(-20.862312, \, 0.35182148) $  &  $(20.939764, \, 0.66976288)$ \\
      \hline
  $1$ &  $(-21.978021, \, -0.51202271) R(\theta)$  & $ (-20.899536, \,  0.51202271)R(\theta)$ &   $(20.899536, \, 0.51202271)R(\theta)$  \\
 \hline
\end{tabular}
\caption{\label{table7} The positions of $\bar{q}_{i, j}=\bar{q}_{i}(\frac{j}{10}) \, (i=1,2,3, \, j=0,1,2, \dots, 10)$ in the path $\mathcal{P}_{test}= \mathcal{P}_{test, \, \theta, \, \theta_0}$ corresponding to $\theta  \in (0, \, 0.002 \pi]$.}
\end{table}
\end{enumerate}
Seven different figures (Fig.~\ref{fig1}, Fig.~\ref{fig2}) are given to show that the following inequality 
\[ \mathcal{A}_{test}= \mathcal{A}(\mathcal{P}_{test})< g_1(\theta) \]
holds for every $\theta \in (0, \, 0.0539 \pi]$. It follows that when $\theta \in (0, \, 0.0539 \pi]$, the action minimizer $\mathcal{P}_{Q_1}$ connecting $Q_{S}$ and $Q_{E_1}$ is free of collision. The proof is complete.
 \begin{figure}[!htbp]
 \begin{center}
 \includegraphics[width=4.6in]{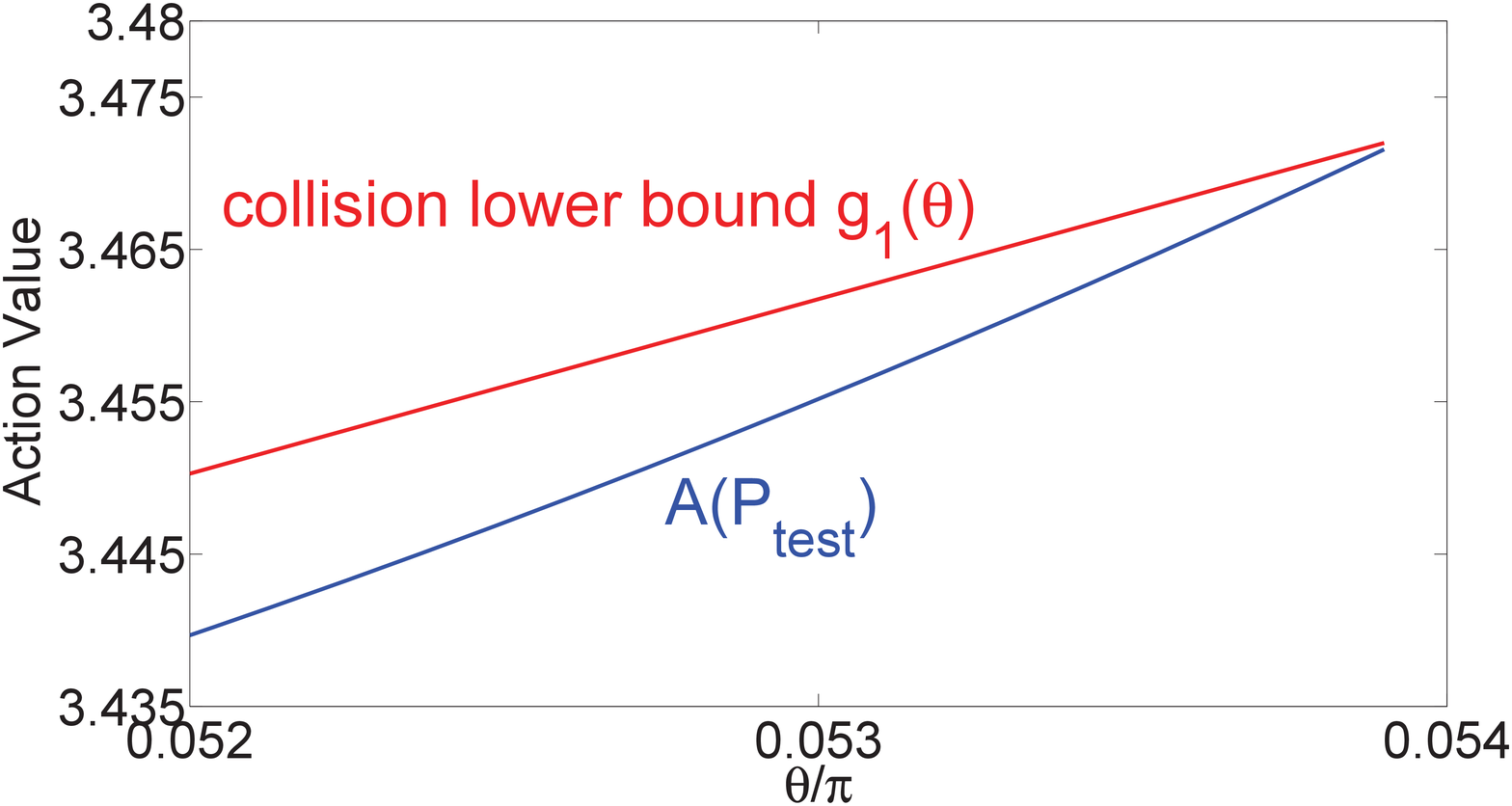}
 \end{center}
 \caption{\label{fig1} In the figure, the horizontal axis is $\theta/\pi$, and the vertical axis is the action value $\mathcal{A}$. When $\theta \in [0.052\pi, 0.0539 \pi]$,  the red curve is the graph of $g_1(\theta)$, the lower bound of action of paths with boundary collisions; while the purple curve is the graph of $\mathcal{A}(\mathcal{P}_{test})$, action of the test path $\mathcal{P}_{test}$.  }
    \end{figure}
    
 \begin{figure}[!htbp]
 \begin{center}
\subfigure[ $ 0.0425 \pi\leq \theta \leq 0.052 \pi $ ]{\includegraphics[width=2.47in]{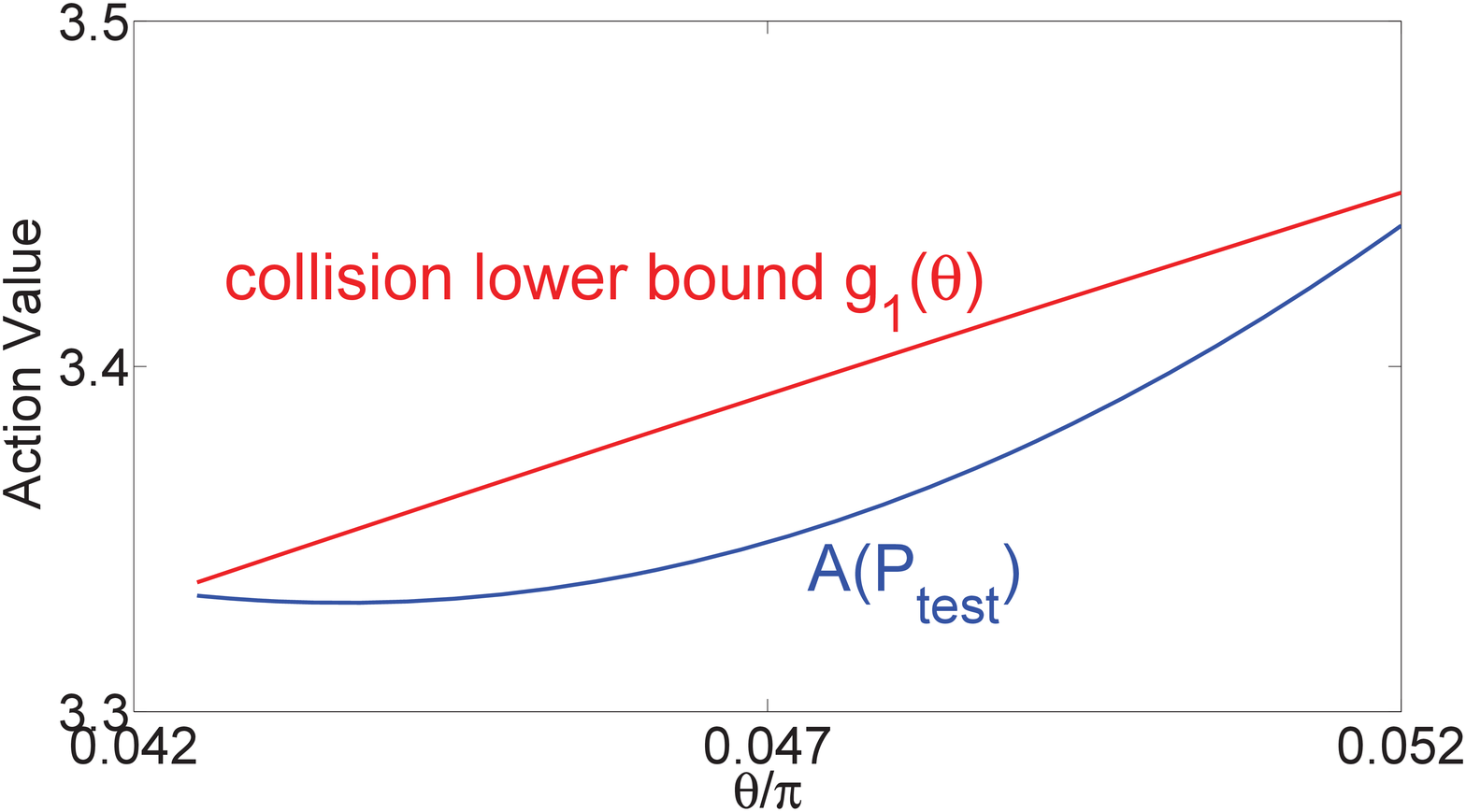}} 
\subfigure[ $0.032 \pi \leq \theta \leq 0.0425\pi$ ]{\includegraphics[width=2.47in]{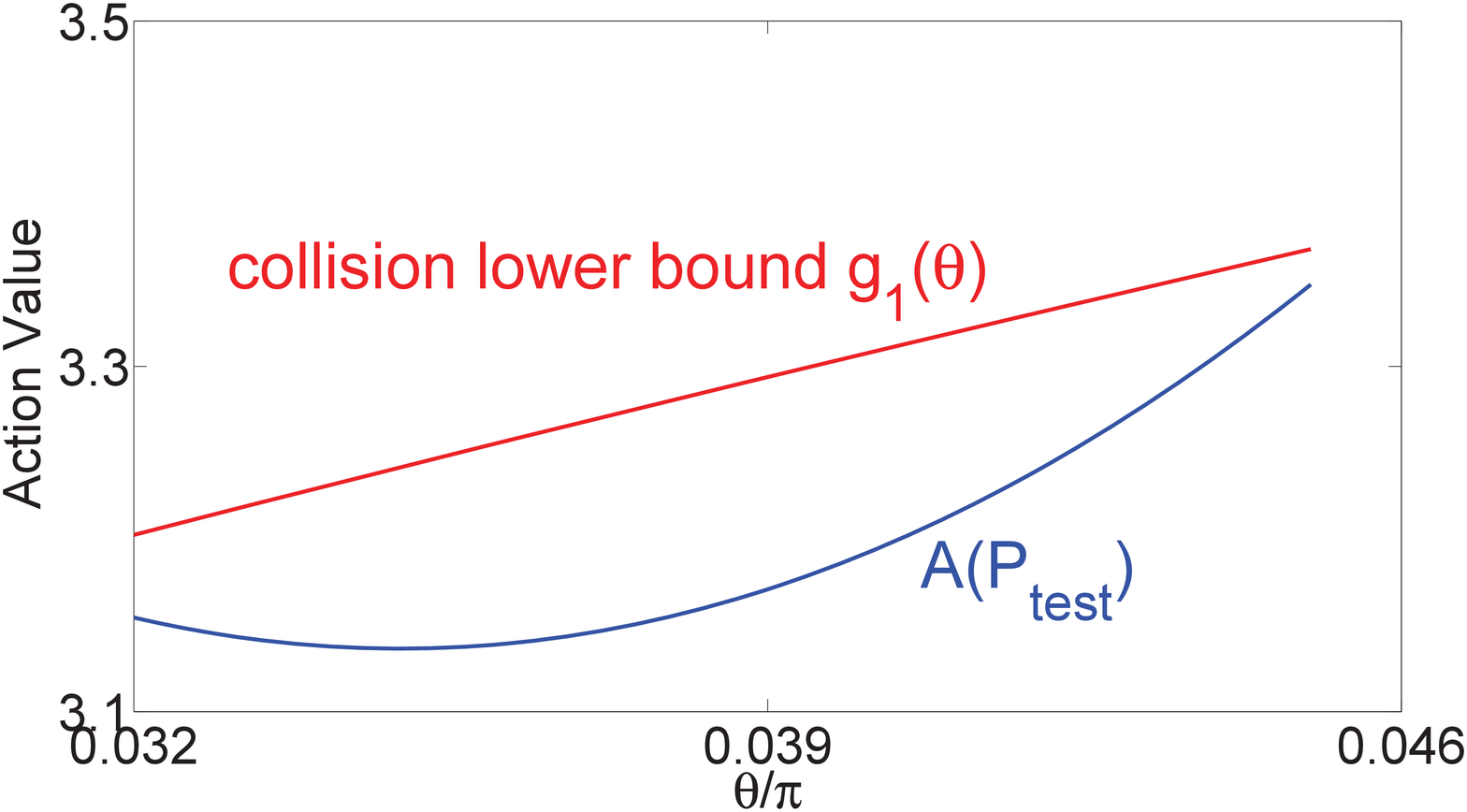}}
\subfigure[ $0.021\pi  \leq \theta \leq  0.032\pi$ ]{\includegraphics[width=2.47in]{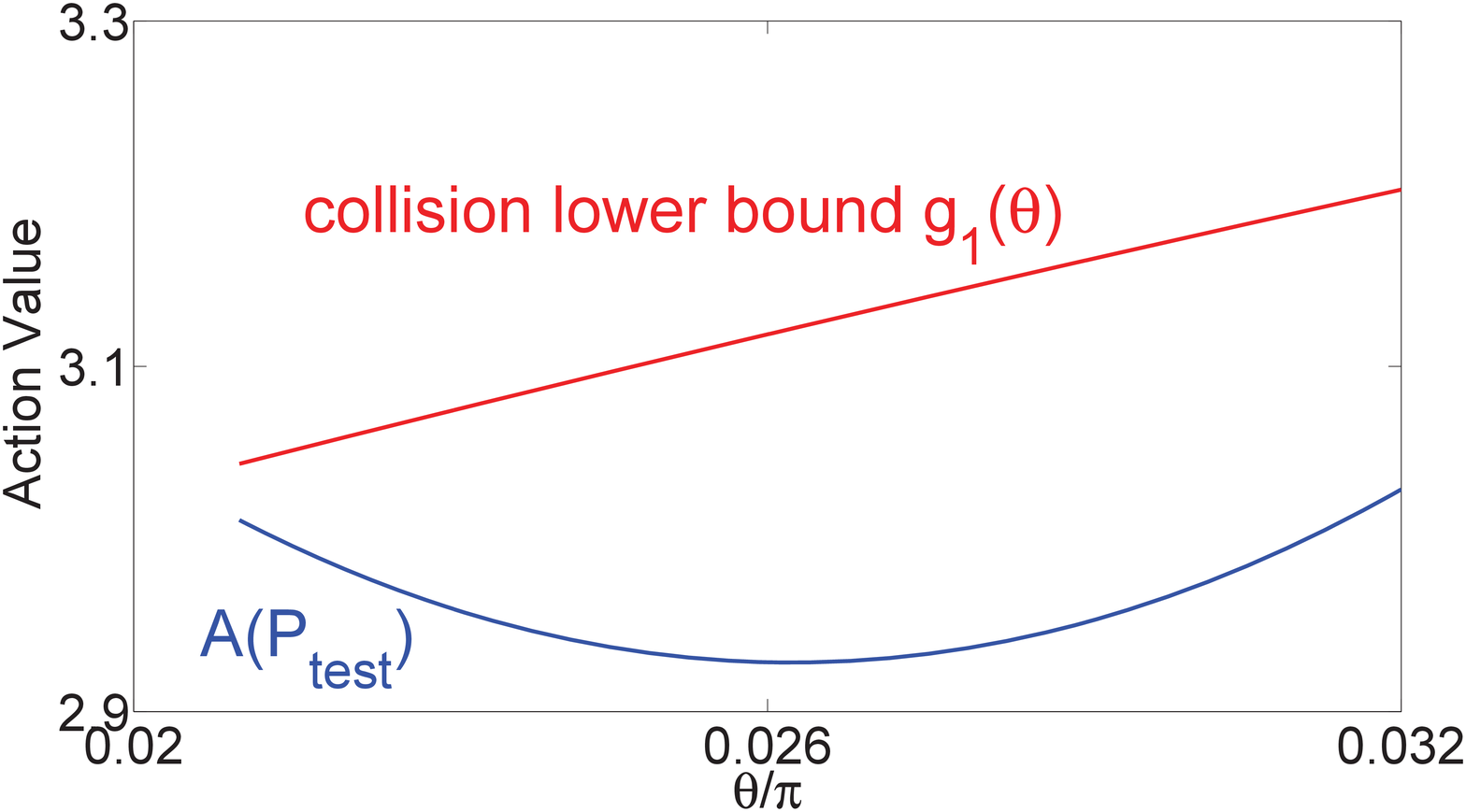}}
\subfigure[ $0.008 \pi  \leq \theta \leq  0.021\pi$ ]{\includegraphics[width=2.47in]{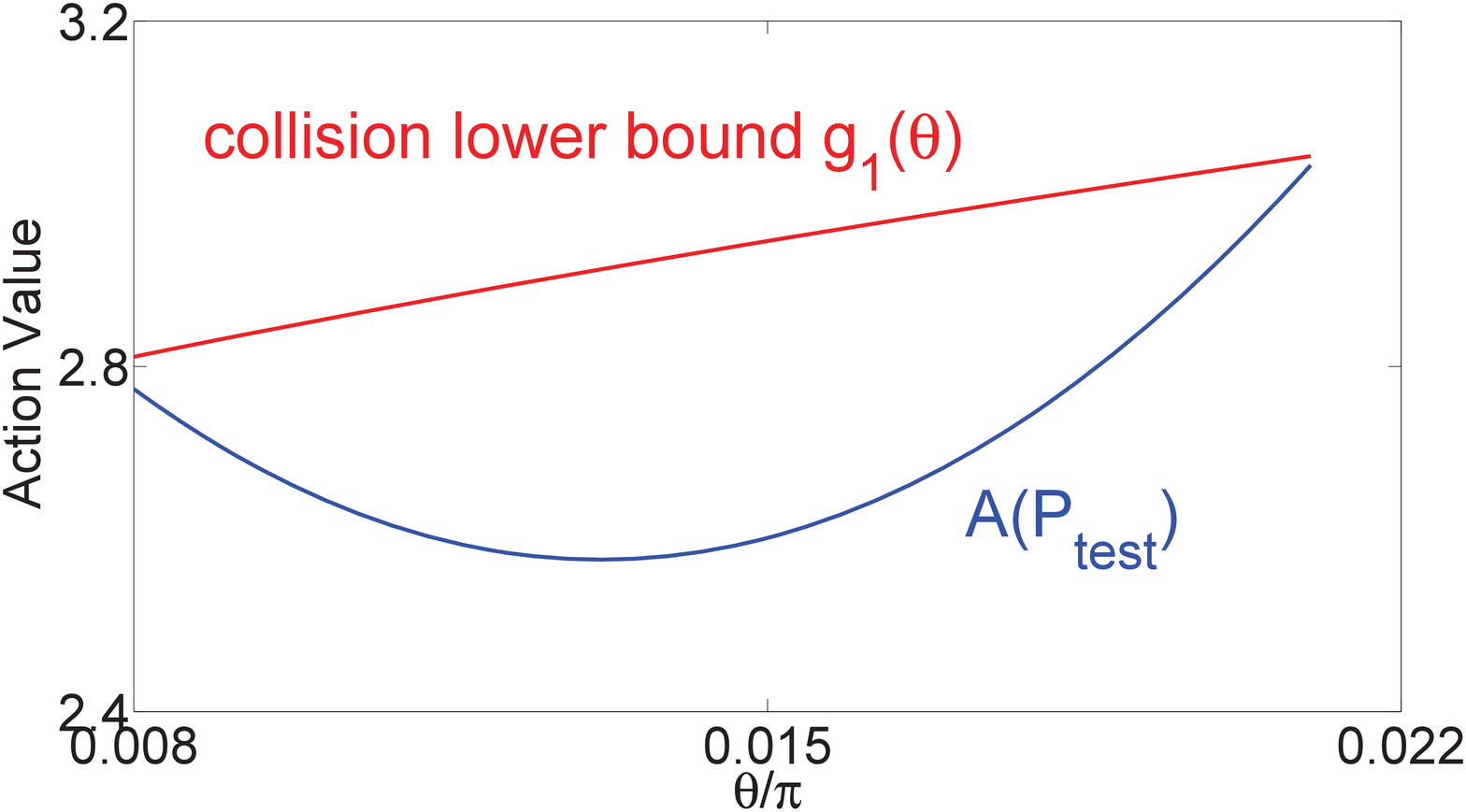}} 
\subfigure[ $0.002 \pi  \leq \theta \leq  0.008\pi$ ]{\includegraphics[width=2.47in]{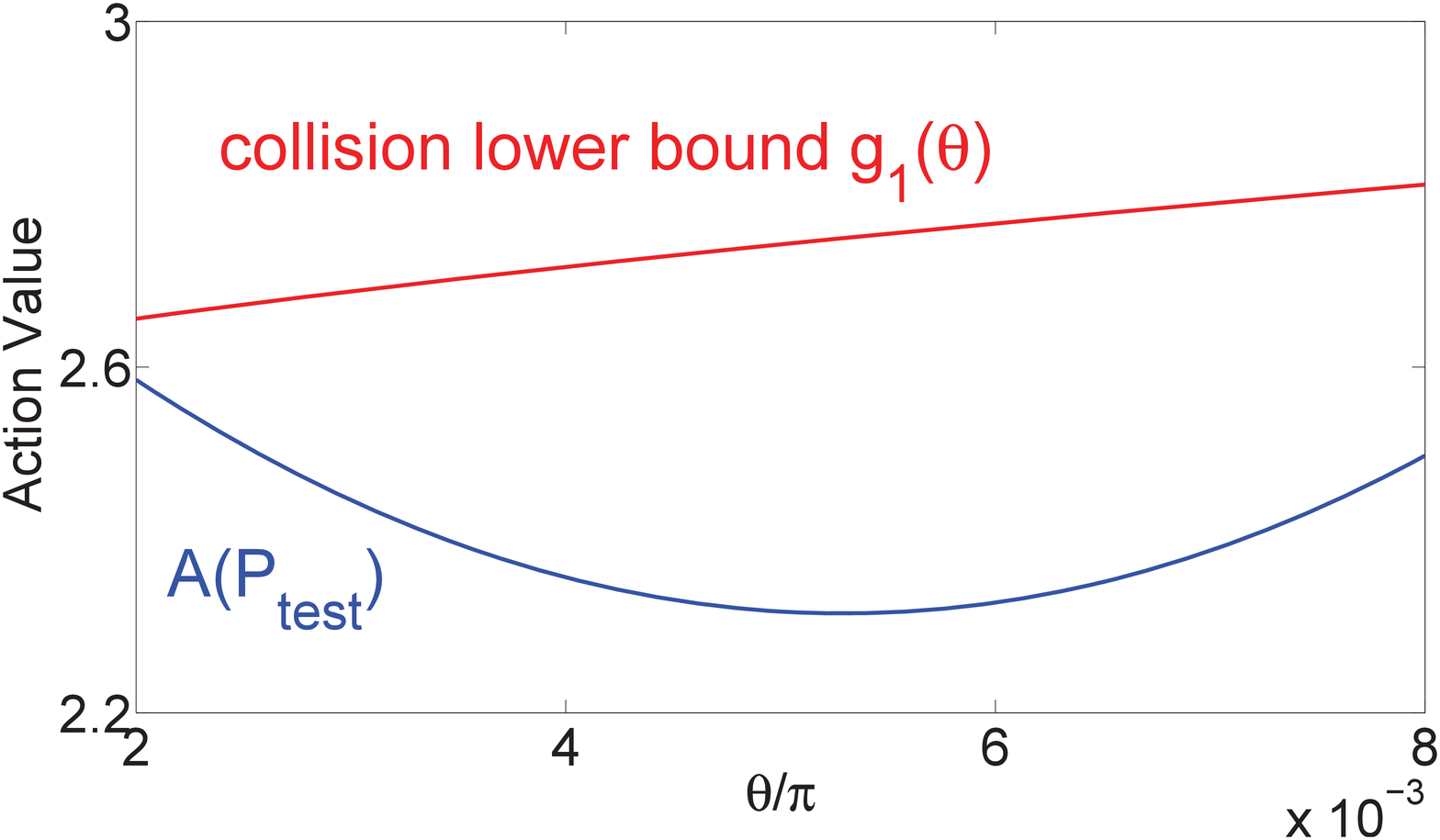}}
\subfigure[ $0 < \theta \leq  0.002\pi$ ]{\includegraphics[width=2.47in]{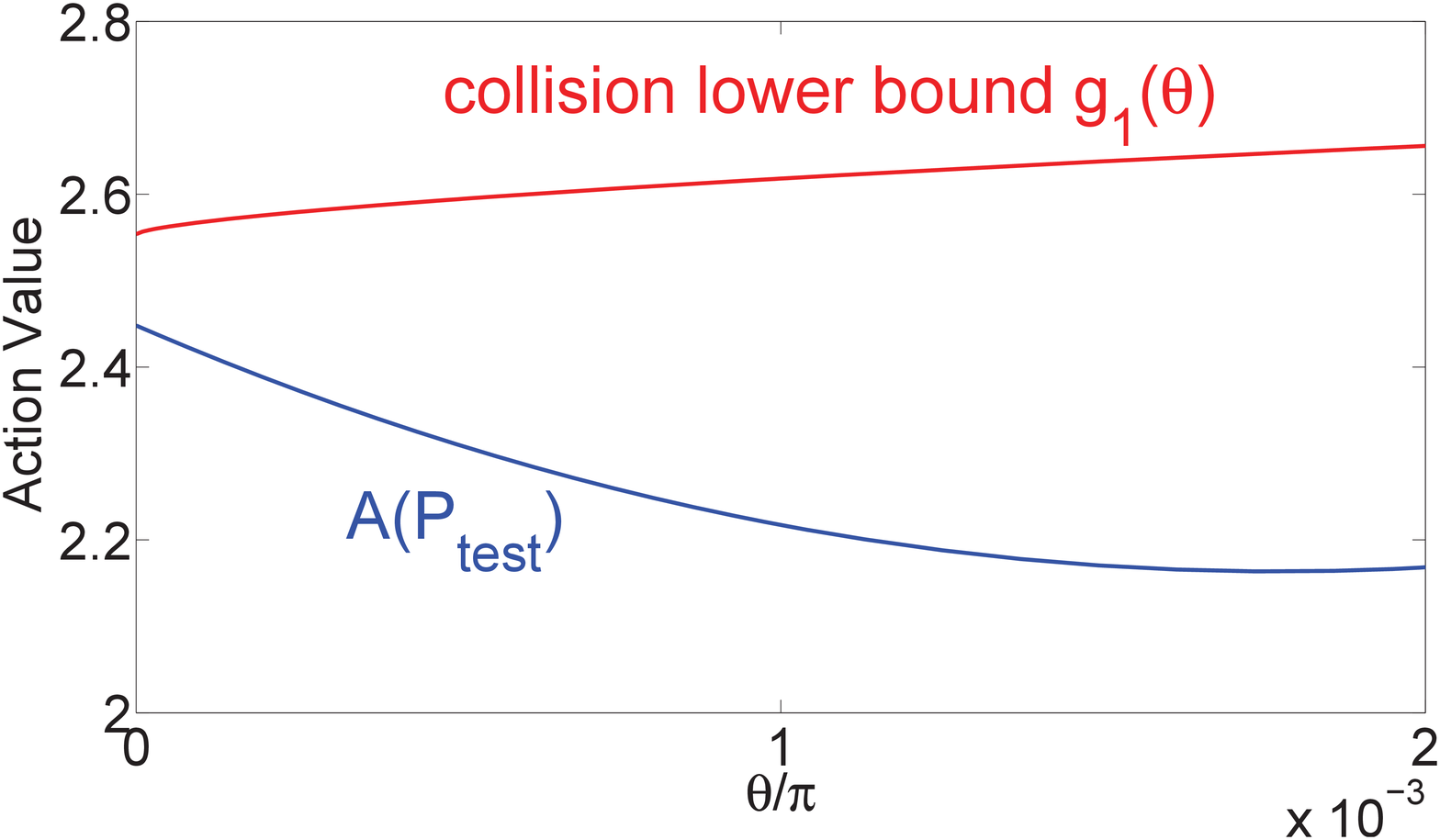}} 
 \end{center}
 \caption{\label{fig2} In each subfigure, the horizontal axis is $\theta/\pi$, and the vertical axis is the action value $\mathcal{A}$. The graphs of $g_1(\theta)$ (lower bound of action of paths with boundary collisions) and the graphs of $\mathcal{A}(\mathcal{P}_{test})$ (action of the test path $\mathcal{P}_{test}$) are shown for different intervals of $\theta$. For each $\theta \in (0, 0.052 \pi]$, the test path $\mathcal{P}_{test}$ is defined as a piecewise smooth linear function, whose nodal points are given by the tables (from Table \ref{table2} to Table \ref{table7}). }
  \end{figure} 

\end{proof}

\subsection{test paths connecting $Q_S$ and $Q_{E_2}$}
Recall that 
\begin{equation}\label{boundarypq2}
Q_s=\begin{bmatrix}
-a_{1}-c_{1} & 0 \\
-a_{1} &   0\\
(2a_{1}+c_{1})/2 & b_{1} \\
(2a_{1}+c_{1})/2 & -b_{1} 
\end{bmatrix}, \, \,  Q_{e_2}= \begin{bmatrix}
-a_{2} & -b_{2} \\
-a_{2} &   b_{2}\\
a_{2} & c_{2} \\
a_{2} & -c_{2} 
\end{bmatrix}R(\theta), 
\end{equation}
where $a_1, a_2 \in \mathbb{R}$, $b_1, b_2, c_1, c_2 \geq 0$. For each given $\theta$, $Q_S$ and $Q_{E_2}$ are set to be the boundary configuration sets:
\begin{equation}\label{QS2}
 Q_S= \left\{ Q_s \, \bigg| \, a_1 \in \mathbb{R}, \,  b_1 \geq 0,  \, c_1 \geq 0   \right\},
\end{equation}
\begin{equation}\label{QE2}
 Q_{E_2}= \left\{ Q_{e_2} \, \bigg| \, a_2 \in \mathbb{R}, \, b_2 \geq 0,  \, c_2 \geq 0   \right\},
\end{equation}
where $Q_s$ and $Q_{e_2}$ are defined in \eqref{boundarypq2}. We set $P(Q_S, Q_{E_2})$ to be the set of paths in $H^1([0,1], \chi)$ which have boundaries in $Q_S$ and $Q_{E_2}$:
\[ P(Q_S, Q_{E_2}):=\left\{q(t) \in H^1([0,1], \chi) \, \bigg| \,  q(0) \in  Q_S,  \, q(1) \in Q_{E_2} \right\}.  \]
Let $\tilde{q}(t)= \begin{bmatrix}
\tilde{q}_{1}(t) \\
\tilde{q}_{2}(t) \\
\tilde{q}_{3}(t) \\
\tilde{q}_{4}(t)  
\end{bmatrix}$ be the position matrix path of the action minimizer $\mathcal{P}_{Q_2}([0,1])$. By Lemma \ref{lowerbdd2}, if $\mathcal{P}_{Q_2}([0,1])$ has boundary collisions, its action $\mathcal{A}_{col}$ has a lower bound 
\[ \mathcal{A}_{col} \geq  \frac{3}{8}16^{\frac{1}{3}} \left[ \pi^{\frac{2}{3}}+ \theta^{\frac{2}{3}} + 2\left(2 \theta \right)^{\frac{2}{3}} \right].  \]
Let 
\[ g_2(\theta)= \frac{3}{8}16^{\frac{1}{3}} \left[ \pi^{\frac{2}{3}}+ \theta^{\frac{2}{3}} + 2\left(2 \theta \right)^{\frac{2}{3}} \right]. \]
For small $\theta$, the following lemma holds:
\begin{lemma}\label{testpathdef2}
For any given $\theta \in (0, 0.0664 \pi]$, there exists a test path $\mathcal{P}_{test2}\in P(Q_S, Q_{E_2})$,  such that its action $\mathcal{A}_{test2}= \mathcal{A}(\mathcal{P}_{test2})$ satisfies
\[ \mathcal{A}_{test2}= \mathcal{A}(\mathcal{P}_{test2}) < g_2(\theta).  \]
\end{lemma}
\begin{proof}
The proof of Lemma \ref{testpathdef2} follows by a similar argument of Lemma \ref{testpathdef1}. The only difference is to choose a new set of $\theta_0$ and to define the corresponding test paths $\mathcal{P}_{test2}= \mathcal{P}_{test2, \, \theta, \, \theta_0}$ by piecewise smooth linear functions, which satisfy 
\[ \mathcal{P}_{test2} \in P(Q_S, Q_{E_2}) \]
and its action $\mathcal{A}(\mathcal{P}_{test2})< g_2(\theta)$ for each $\theta$.  

Let $\hat{q}(t)$ be the position matrix path of $\mathcal{P}_{test2}$. Once the 11 nodal points $\hat{q}_{i,j}= \hat{q}_{i}(\frac{j}{10})\, (i=1,2,3,4, \, j=0,1,2, \dots, 10)$ are given, the test path $\mathcal{P}_{test2}=\mathcal{P}_{test2, \, \theta, \, \theta_0}$ can then be defined by 
\begin{equation}\label{defofptest2}
  \hat{q}(t) = \hat{q}(t_j) + 10\left(t- \frac{j}{10} \right)\left[\hat{q}(t_{j+1})- \hat{q}(t_j) \right], \quad  t \in \left[ \frac{j}{10}, \, \frac{j+1}{10} \right]. 
\end{equation}
Nine different $\theta_0$ are chosen and the coordinates of $\hat{q}_{i, j}=\hat{q}_i(\frac{j}{10}) \, (i=1,2,3, j=0,1,2, \dots, 10)$ of $\mathcal{P}_{test2}=\mathcal{P}_{test2, \, \theta_0}$ are given in the following talbes (from Table \ref{table1p4} to Table \ref{table9p4}), while $\hat{q}_{4, j}$ satisfies
\[ \hat{q}_{4, j}= -\hat{q}_{1, j}-\hat{q}_{2, j}-\hat{q}_{3, j}, \quad j=0,1,2, \dots, 10.\]

\begin{enumerate}[(i)]
\item $\theta_0= 0.066 \pi$: $\text{it works for} \,  \theta \in [0.0647 \pi, \, 0.0664 \pi ]$, in which $\hat{q}_{i, j} \, (i=1,2,3; j=0,1, 2, \dots, 10)$ of the test paths are given by Table \ref{table1p4};

\begin{table}[!htbp]
\centering
\tiny
\begin{tabular}{ |c|c|c|c|c|} 
 \hline
\multicolumn{4}{|c|}{$\theta_0= 0.066 \pi, \quad \quad \quad \theta \in [0.0647 \pi, \, 0.0664 \pi]$} \\
\hline
$t$ &  $\hat{q}_1$ &$\hat{q}_2$    &   $\hat{q}_3$ \\
  \hline 
  $0$ &  $(-2.3940, \, 0)$ &  $( -1.5412, \, 0)$ & $( 1.9676, \, 2.2651)$ \\
  \hline 
  $0.1$ & $( -2.3867766, \, -0.11723453)$  &  $ (-1.5475744, \, 0.035630714)$ & $ (1.9202780, \, 2.3054125)$  \\
 \hline
  $0.2$ & $(-2.3653339, \, -0.23199426)$  &  $(-1.5664706, \, 0.068823700)$  & $(  1.8721221, \, 2.3447184)$  \\
   \hline
  $0.3$ & $(-2.3303440, \,  -0.34189349)$  & $(   -1.5972184, \, 0.097230010)$  & $( 1.8231487, \, 2.3830017 )$ \\
   \hline
  $0.4$ &  $(-2.2828944, \, -0.44471890)$  & $(  -1.6387332, \, 0.11867262 ) $  &    $(1.7733756, \, 2.4202482 ) $\\
     \hline
   $0.5$ & $(-2.2244424, \,  -0.53850264)$ &  $ (-1.6895619, \, 0.13121924 ) $  &  $ (  1.7228224, \, 2.4564451) $  \\
   \hline
  $0.6$ & $(-2.1567572, \, -0.62158163)$ &  $  ( -1.7479403, \, 0.13324168 )$ & $(1.6715106, \, 2.4915809 )$   \\
     \hline
   $0.7$ & $(-2.0818550, \, -0.69264221)$ & $ ( -1.8118580, \, 0.12346047 ) $ &  $ ( 1.6194636, \, 2.5256449 )$ \\
      \hline 
   $0.8$ &   $(   -2.0019308, \, -0.75075035)$ &  $(-1.8791264, \, 0.10097519 )  $ &   $ (  1.5667066, \, 2.5586270 )$\\
      \hline
   $0.9$ & $(-1.9192898, \, -0.79536906 )$ &  $(  -1.9474468, \, 0.065282074)  $  &  $( 1.5132661, \, 2.5905172 )$ \\
      \hline
  $1$ &  $( -1.9670678, \,  -0.43064357) R(\theta)$  & $ ( -1.9670678, \, 0.43064357)R(\theta)$ &   $(1.9670678, \,  2.2647707)R(\theta)$  \\
 \hline
\end{tabular}
\caption{\label{table1p4} The positions of $\hat{q}_{i, j}=\hat{q}_{i}(\frac{j}{10}) \, (i=1,2,3, \, j=0,1,2, \dots, 10)$ in the path $\mathcal{P}_{test2}= \mathcal{P}_{test2, \, \theta, \, \theta_0}$ corresponding to each $\theta  \in [0.0647 \pi, \, 0.0664 \pi]$.}
\end{table}

\item $\theta_0= 0.0625 \pi$: $\text{it works for} \,  \theta \in [0.057 \pi, \, 0.0647 \pi ]$, in which $\hat{q}_{i, j} \, (i=1,2,3; j=0,1, 2, \dots, 10)$ of the test paths are given by Table \ref{table2p4};
\begin{table}[!htbp]
\centering
\tiny
\begin{tabular}{ |c|c|c|c|c|} 
 \hline
\multicolumn{4}{|c|}{$\theta_0= 0.0625 \pi, \quad \quad \quad \theta \in [0.057 \pi, \, 0.0647 \pi]$} \\
\hline
$t$ &  $\hat{q}_1$ &$\hat{q}_2$    &   $\hat{q}_3$ \\
  \hline 
  $0$ &  $(-2.4684, \, 0 )$ &  $(-1.6120, \, 0  )$ & $( 2.0402, \, 2.3491 )$ \\
  \hline
  $0.1$ & $(-2.4612581,\, -0.11629121 )$  &  $ (-1.6183528, \, 0.036146652)$ & $ ( 1.9937276, \, 2.3887080)$  \\
 \hline
  $0.2$ & $( -2.4400538, \,  -0.23014692 )$  &  $(-1.6371903, \,  0.069890411)$  & $(   1.9464796, \, 2.4273818 )$  \\
   \hline
  $0.3$ & $(-2.4054424, \,   -0.33921680)$  & $( -1.6678588, \,    0.098913300)$  & $( 1.8984702, \, 2.4651074)$ \\
   \hline
  $0.4$ &  $( -2.3584853, \,  -0.44131575)$  & $(-1.7092995, \,    0.12106216) $  &    $( 1.8497153, \, 2.5018721 ) $\\
     \hline
   $0.5$ & $( -2.3006078, \,  -0.53449472)$ &  $ (-1.7600904, \,  0.13441927) $  &  $ ( 1.8002319, \, 2.5376646) $  \\
   \hline
  $0.6$ & $( -2.2335446, \,   -0.61709914)$ &  $  (-1.8185012, \, 0.13736081)$ & $( 1.7500390, \,  2.5724747)$   \\
     \hline
   $0.7$ & $(-2.1592783, \, -0.68781396)$ & $ (-1.8825542, \, 0.12860187 ) $ &  $ (  1.6991572, \, 2.6062928)$ \\
      \hline 
   $0.8$ &   $(  -2.0799735, \, -0.74569513)$ &  $(-1.9500902, \,  0.10722806 )  $ &   $ ( 1.6476085, \, 2.6391099 )$\\
      \hline
   $0.9$ & $(-1.9979103, \,  -0.79018904)$ &  $(  -2.0188347, \, 0.072715031)  $  &  $(  1.5954161, \, 2.6709169)$ \\
      \hline
  $1$ &  $( -2.0388113, \,  -0.43168370) R(\theta)$  & $ ( -2.0388113, \, 0.43168370)R(\theta)$ &   $(  2.0388113, \, 2.3488448)R(\theta)$  \\
 \hline
\end{tabular}
\caption{\label{table2p4} The positions of $\hat{q}_{i, j}=\hat{q}_{i}(\frac{j}{10}) \, (i=1,2,3, \, j=0,1,2, \dots, 10)$ in the path $\mathcal{P}_{test2}= \mathcal{P}_{test2, \, \theta, \, \theta_0}$ corresponding to each $\theta  \in [0.057 \pi, \, 0.0647 \pi]$.}
\end{table}

\item $\theta_0= 0.053 \pi$: $\text{it works for} \,  \theta \in [0.046 \pi, \, 0.057 \pi ]$, in which $\hat{q}_{i, j} \, (i=1,2,3; j=0,1, 2, \dots, 10)$ of the test paths are given by Table \ref{table3p4};
\begin{table}[!htbp]
\centering
\tiny
\begin{tabular}{ |c|c|c|c|c|} 
 \hline
\multicolumn{4}{|c|}{$\theta_0= 0.053 \pi, \quad \quad \quad \theta \in [0.046 \pi, \, 0.057 \pi]$} \\
\hline
$t$ &  $\hat{q}_1$ &$\hat{q}_2$    &   $\hat{q}_3$ \\
  \hline 
  $0$ &  $(-2.7099, \, 0)$ &  $( -1.8419, \, 0 )$ & $( 2.2759, \, 2.6219)$ \\
  \hline
  $0.1$ & $( -2.7030019\,  -0.11356099 )$  &  $ ( -1.8481652, \, 0.037774686)$ & $ ( 2.2319651, \, 2.6594303)$  \\
 \hline
  $0.2$ & $(-2.6825135, \,  -0.22480332)$  &  $(-1.8667554, \, 0.073252835)$  & $(   2.1874065, \, 2.6962144 )$  \\
   \hline
  $0.3$ & $( -2.6490442, \,   -0.33148442)$  & $(-1.8970621, \, 0.10421386 )$  & $(  2.1422340, \, 2.7322426)$ \\
   \hline
  $0.4$ &  $( -2.6035838, \, -0.43151004)$  & $( -1.9380968, \, 0.12858522) $  &    $(  2.0964584, \, 2.7675063) $\\
     \hline
   $0.5$ & $(-2.5474670, \,  -0.52299919)$ &  $ (-1.9885273, \, 0.14450728) $  &  $ (  2.0500912, \, 2.8019975) $  \\
   \hline
  $0.6$ & $(-2.4823257, \, -0.60433943 )$ &  $  (-2.0467242, \, 0.15038855)$ & $( 2.0031453, \, 2.8357091)$   \\
     \hline
   $0.7$ & $(-2.4100348, \, -0.67423100)$ & $ ( -2.1108155, \, 0.14494988) $ &  $ (1.9556343, \, 2.8686344)$ \\
      \hline 
   $0.8$ &   $( -2.3326531, \, -0.73171960)$ &  $(-2.1787457, \, 0.12725724)  $ &   $ (  1.9075733, \, 2.9007667)$\\
      \hline
   $0.9$ & $( -2.2523613, \,  -0.77621822)$ &  $( -2.2483372, \, 0.096743648)  $  &  $( 1.8589777, \, 2.9320996)$ \\
      \hline
  $1$ &  $( -2.2752051, \,  -0.43647159 ) R(\theta)$  & $ ( -2.2752051, \, 0.43647159)R(\theta)$ &   $( 2.2752051, \, 2.6216937)R(\theta)$  \\
 \hline
\end{tabular}
\caption{\label{table3p4} The positions of $\hat{q}_{i, j}=\hat{q}_{i}(\frac{j}{10}) \, (i=1,2,3, \, j=0,1,2, \dots, 10)$ in the path $\mathcal{P}_{test2}= \mathcal{P}_{test2, \, \theta, \, \theta_0}$ corresponding to each $\theta  \in [0.046 \pi, \, 0.057 \pi]$.}
\end{table}

\item $\theta_0= 0.04 \pi$: $\text{it works for} \,  \theta \in [0.032 \pi, \, 0.046 \pi ]$, in which $\hat{q}_{i, j} \, (i=1,2,3; j=0,1, 2, \dots, 10)$ of the test paths are given by Table \ref{table4p4};
\begin{table}[!htbp]
\centering
\tiny
\begin{tabular}{ |c|c|c|c|c|} 
 \hline
\multicolumn{4}{|c|}{$\theta_0= 0.04\pi, \quad \quad \quad \theta \in [0.032 \pi, \, 0.046 \pi]$} \\
\hline
$t$ &  $\hat{q}_1$ &$\hat{q}_2$    &   $\hat{q}_3$ \\
  \hline 
  $0$ &  $( -3.1854, \, 0)$ &  $(-2.3020, \, 0  )$ & $(  2.7437, \, 3.1629)$ \\
  \hline  
  $0.1$ & $( -3.1788132, \, -0.10942535)$  &  $ (  -2.3081525, \, 0.040467254)$ & $ ( 2.7037536, \, 3.1971271 )$  \\
 \hline
  $0.2$ & $(  -3.1592398, \, -0.21667601)$  &  $( -2.3264232, \, 0.078771243 )$  & $(  2.6633778, \, 3.2308449)$  \\
   \hline
  $0.3$ & $(  -3.1272338, \, -0.31964291)$  & $(  -2.3562584, \, 0.11281425 )$  & $( 2.6225777, \, 3.2640483 )$ \\
   \hline
  $0.4$ &  $( -3.0836987, \, -0.41634553)$  & $( -2.3967554, \, 0.14062699) $  &    $( 2.5813588, \, 3.2967327 ) $\\
     \hline
   $0.5$ & $(  -3.0298580, \, -0.50498975)$ &  $ (  -2.4466917, \, 0.16042641) $  &  $ (  2.5397271, \, 3.3288940) $  \\
   \hline
  $0.6$ & $( -2.9672164, \, -0.58401877)$ &  $  (-2.5045634, \, 0.17066663 )$ & $( 2.4976891, \,  3.3605282)$   \\
     \hline
   $0.7$ & $( -2.8975146, \, -0.65215563)$ & $ (-2.5686313, \, 0.17008141) $ &  $ ( 2.4552520, \, 3.3916315)$ \\
      \hline 
   $0.8$ &   $(  -2.8226773, \, -0.70843673)$ &  $( -2.6369721, \, 0.15771770)  $ &   $ ( 2.4124232, \, 3.4222005 )$\\
      \hline
   $0.9$ & $( -2.7447581, \, -0.75223596 )$ &  $(   -2.7075335, \, 0.13295989)  $  &  $(    2.3692106, \, 3.4522313)$ \\
      \hline
  $1$ &  $( -2.7430323, \, -0.44297971) R(\theta)$  & $ (-2.7430323, \, 0.44297971)R(\theta)$ &   $( 2.7430323, \, 3.1627879)R(\theta)$  \\
 \hline
\end{tabular}
\caption{\label{table4p4} The positions of $\hat{q}_{i, j}=\hat{q}_{i}(\frac{j}{10}) \, (i=1,2,3, \, j=0,1,2, \dots, 10)$ in the path $\mathcal{P}_{test2}= \mathcal{P}_{test2, \, \theta, \, \theta_0}$ corresponding to each $\theta  \in [0.032 \pi, \, 0.046 \pi]$.}
\end{table}

\item $\theta_0= 0.03 \pi$: $\text{it works for} \,  \theta \in [0.025 \pi, \, 0.032 \pi ]$, in which $\hat{q}_{i, j} \, (i=1,2,3; j=0,1, 2, \dots, 10)$ of the test paths are given by Table \ref{table5p4};
\begin{table}[!htbp]
\centering
\tiny
\begin{tabular}{ |c|c|c|c|c|} 
 \hline
\multicolumn{4}{|c|}{$\theta_0= 0.03\pi, \quad \quad \quad \theta \in [0.025 \pi, \, 0.032 \pi]$} \\
\hline
$t$ &  $\hat{q}_1$ &$\hat{q}_2$    &   $\hat{q}_3$ \\
  \hline 
  $0$ &  $(-3.7698, \, 0  )$ &  $(-2.8742, \, 0  )$ & $(   3.322, \, 3.8316)$ \\
  \hline 
  $0.1$ & $( -3.7634455, \, -0.10572303)$  &  $ (  -2.8802589, \, 0.043103997)$ & $ (   3.2857526, \, 3.8627383 )$  \\
 \hline
  $0.2$ & $( -3.7445556, \, -0.20937607 )$  &  $( -2.8982618, \, 0.084143820)$  & $(  3.2492120, \, 3.8935312)$  \\
   \hline
  $0.3$ & $(-3.7136462, \, -0.30894805 )$  & $(  -2.9276933, \, 0.12111418)$  & $( 3.2123810, \, 3.9239762  )$ \\
   \hline
  $0.4$ &  $( -3.6715597, \,  -0.40254366)$  & $( -2.9677111, \, 0.15212548 ) $  &    $( 3.1752623, \, 3.9540709  ) $\\
     \hline
   $0.5$ & $(  -3.6194403, \, -0.48843630)$ &  $ (  -3.0171714, \, 0.17545680) $  &  $ ( 3.1378590, \, 3.9838131  ) $  \\
   \hline
  $0.6$ & $( -3.5587005, \, -0.56511558)$ &  $  (-3.0746622, \, 0.18960332)$ & $(3.1001745, \, 4.0132008  )$   \\
     \hline
   $0.7$ & $(-3.4909805, \,  -0.63132791)$ & $ (-3.1385438, \, 0.19331694) $ &  $ ( 3.0622122, \, 4.0422318 )$ \\
      \hline 
   $0.8$ &   $(  -3.4181023, \, -0.68610960)$ &  $(   -3.2069949, \, 0.18563941)  $ &   $ (  3.0239760, \, 4.0709043 )$\\
      \hline
   $0.9$ & $(  -3.3420190, \, -0.72881164 )$ &  $(   -3.2780628, \, 0.16592709)  $  &  $(  2.9854698, \, 4.0992163  )$ \\
      \hline
  $1$ &  $( -3.3217121, \, -0.44850587) R(\theta)$  & $ (-3.3217121, \,  0.44850587 )R(\theta)$ &   $( 3.3217121, \, 3.8315403 )R(\theta)$  \\
 \hline
\end{tabular}
\caption{\label{table5p4} The positions of $\hat{q}_{i, j}=\hat{q}_{i}(\frac{j}{10}) \, (i=1,2,3, \, j=0,1,2, \dots, 10)$ in the path $\mathcal{P}_{test2}= \mathcal{P}_{test2, \, \theta, \, \theta_0}$ corresponding to each $\theta  \in [0.025 \pi, \, 0.032 \pi]$.}
\end{table}

\item $\theta_0= 0.02 \pi$: $\text{it works for} \,  \theta \in [0.014 \pi, \, 0.025 \pi ]$, in which $\hat{q}_{i, j} \, (i=1,2,3; j=0,1, 2, \dots, 10)$ of the test paths are given by Table \ref{table6p4};
\begin{table}[!htbp]
\centering
\tiny
\begin{tabular}{ |c|c|c|c|c|} 
 \hline
\multicolumn{4}{|c|}{$\theta_0= 0.02\pi, \quad \quad \quad \theta \in [0.014 \pi, \, 0.025 \pi]$} \\
\hline
$t$ &  $\hat{q}_1$ &$\hat{q}_2$    &   $\hat{q}_3$ \\
  \hline 
  $0$ &  $(-4.8050, \, 0  )$ &  $(  -3.8974, \, 0 )$ & $(  4.3512, \, 5.0208)$ \\
  \hline 
  $0.1$ & $(-4.7988632, \,  -0.10123926)$  &  $ ( -3.9033647, \,  0.046559631 )$ & $ (   4.3195742, \, 5.0480403 )$  \\
 \hline
  $0.2$ & $(  -4.7806148, \, -0.20050474 )$  &  $( -3.9210971, \, 0.091147728)$  & $( 4.2877773, \, 5.0750805)$  \\
   \hline
  $0.3$ & $(-4.7507361, \,   -0.29587581)$  & $(   -3.9501159, \, 0.13184590 )$  & $( 4.2558105, \, 5.1019196)$ \\
   \hline
  $0.4$ &  $(  -4.7100142, \, -0.38553656 )$  & $(-3.9896340, \, 0.16684044 ) $  &    $(4.2236748, \, 5.1285566) $\\
     \hline
   $0.5$ & $(  -4.6595211, \, -0.46782431 )$ &  $ ( -4.0385794, \, 0.19447087 ) $  &  $ (  4.1913715, \, 5.1549907) $  \\
   \hline
  $0.6$ & $( -4.6005841, \,  -0.54127381)$ &  $  ( -4.0956253, \, 0.21327410 )$ & $( 4.1589017, \, 5.1812209)$   \\
     \hline
   $0.7$ & $( -4.5347500, \,  -0.60465576 )$ & $ ( -4.1592248, \, 0.22202298) $ &  $ (  4.1262670, \, 5.2072466)$ \\
      \hline 
   $0.8$ &   $(  -4.4637438, \,  -0.65700904 )$ &  $(  -4.2276531, \, 0.21975851)  $ &   $ ( 4.0934686, \, 5.2330668 )$\\
      \hline
   $0.9$ & $(   -4.3894228, \,  -0.69766562 )$ &  $(-4.2990531, \,  0.20581475)  $  &  $( 4.0605081, \, 5.2586808 )$ \\
      \hline
  $1$ &  $( -4.3508183, \,   -0.45397356 ) R(\theta)$  & $ (-4.3508183, \, 0.45397356 )R(\theta)$ &   $( 4.3508183, \, 5.0207789 )R(\theta)$  \\
 \hline
\end{tabular}
\caption{\label{table6p4} The positions of $\hat{q}_{i, j}=\hat{q}_{i}(\frac{j}{10}) \, (i=1,2,3, \, j=0,1,2, \dots, 10)$ in the path $\mathcal{P}_{test2}= \mathcal{P}_{test2, \, \theta, \, \theta_0}$ corresponding to each $\theta  \in [0.014 \pi, \, 0.025 \pi]$.}
\end{table}

\item $\theta_0= 0.01 \pi$: $\text{it works for} \,  \theta \in [0.0055\pi, \, 0.014 \pi ]$, in which $\hat{q}_{i, j} \, (i=1,2,3; j=0,1, 2, \dots, 10)$ of the test paths are given by Table \ref{table7p4};
\begin{table}[!htbp]
\centering
\tiny
\begin{tabular}{ |c|c|c|c|c|} 
 \hline
\multicolumn{4}{|c|}{$\theta_0= 0.01\pi, \quad \quad \quad \theta \in [0.0055\pi, \, 0.014\pi]$} \\
\hline
$t$ &  $\hat{q}_1$ &$\hat{q}_2$    &   $\hat{q}_3$ \\
  \hline 
  $0$ &  $(-7.3979, \,  0)$ &  $( -6.4779, \, 0)$ & $(   6.9379, \, 7.8589)$ \\
  \hline   
  $0.1$ & $(-7.3919725, \, -0.095204783)$  &  $ ( -6.4837590, \, 0.051604857 )$ & $ (   6.9131258, \, 7.8806603 )$  \\
 \hline
  $0.2$ & $(   -7.3743409, \, -0.18852796 )$  &  $(-6.5011851, \, 0.10132855)$  & $( 6.8882833, \, 7.9023410 )$  \\
   \hline
  $0.3$ & $( -7.3454548, \,  -0.27813625)$  & $(  -6.5297287, \, 0.14733824 )$  & $( 6.8633727, \, 7.9239419 )$ \\
   \hline
  $0.4$ &  $(  -7.3060502, \, -0.36229170)$  & $(-6.5686539, \, 0.18789641 ) $  &    $(6.8383943, \, 7.9454628 ) $\\
     \hline
   $0.5$ & $(   -7.2571310, \, -0.43939623)$ &  $ (-6.6169567, \, 0.22140542) $  &  $ ( 6.8133482, \, 7.9669035 ) $  \\
   \hline
  $0.6$ & $(  -7.1999429, \,  -0.50803255)$ &  $  (  -6.6733914, \, 0.24644841 )$ & $(  6.7882348, \, 7.9882639 )$   \\
     \hline
   $0.7$ & $(  -7.1359416, \,  -0.56700035 )$ & $ (-6.7365024, \, 0.26182551 ) $ &  $ (6.7630544, \, 8.0095438 )$ \\
      \hline 
   $0.8$ &   $(  -7.0667555, \, -0.61534693 )$ &  $(   -6.8046614, \, 0.26658442)  $ &   $ ( 6.7378071, \, 8.0307430 )$\\
      \hline
   $0.9$ & $( -6.9941438, \, -0.65239146)$ &  $(  -6.8761091, \, 0.26004476)  $  &  $(6.7124932, \, 8.0518613 )$ \\
      \hline
  $1$ &  $(  -6.9378255, \,  -0.46004700 ) R(\theta)$  & $ (  -6.9378255, \, 0.46004700)R(\theta)$ &   $( 6.9378255, \,  7.8588677)R(\theta)$  \\
 \hline
\end{tabular}
\caption{\label{table7p4} The positions of $\hat{q}_{i, j}=\hat{q}_{i}(\frac{j}{10}) \, (i=1,2,3, \, j=0,1,2, \dots, 10)$ in the path $\mathcal{P}_{test2}= \mathcal{P}_{test2, \, \theta, \, \theta_0}$ corresponding to each $\theta  \in [0.0055 \pi, \, 0.014 \pi]$.}
\end{table}

\item $\theta_0= 0.005 \pi$: $\text{it works for} \,  \theta \in [0.0024 \pi, \, 0.0055 \pi ]$, in which $\hat{q}_{i, j} \, (i=1,2,3; j=0,1, 2, \dots, 10)$ of the test paths are given by Table \ref{table8p4};
\begin{table}[!htbp]
\centering
\tiny
\begin{tabular}{ |c|c|c|c|c|} 
 \hline
\multicolumn{4}{|c|}{$\theta_0= 0.005\pi, \quad \quad \quad \theta \in [0.0024 \pi, \, 0.0055 \pi]$} \\
\hline
$t$ &  $\hat{q}_1$ &$\hat{q}_2$    &   $\hat{q}_3$ \\
  \hline 
  $0$ &  $(-9.7211, \, 0)$ &  $(-8.7365, \, 0)$ & $(   9.2288, \, 15.8015)$ \\
  \hline  
  $0.1$ & $( -9.7109258, \, -0.083096348)$  &  $ (  -8.7456085, \, 0.054126363)$ & $ (    9.2034671, \, 15.814374 )$  \\
 \hline
  $0.2$ & $( -9.6903282, \, -0.16470663)$  &  $( -8.7650889, \, 0.10676675)$  & $(    9.1781086, \, 15.827216 )$  \\
   \hline
  $0.3$ & $(   -9.6594033, \, -0.24328462)$  & $(  -8.7948450, \, 0.15637503 )$  & $( 9.1527244, \, 15.840025)$ \\
   \hline
  $0.4$ &  $(  -9.6184684, \, -0.31723528)$  & $( -8.8345598, \, 0.20135624 ) $  &    $(  9.1273146, \, 15.852803) $\\
     \hline
   $0.5$ & $(   -9.5680909, \, -0.38494184 )$ &  $ ( -8.8836655, \, 0.24009370 ) $  &  $ ( 9.1018792, \, 15.865548) $  \\
   \hline
  $0.6$ & $(  -9.5091184, \, -0.44480457 )$ &  $  (  -8.9413149, \, 0.27098777 )$ & $( 9.0764182, \, 15.878261)$   \\
     \hline
   $0.7$ & $(   -9.4427023, \,  -0.49529314 )$ & $ ( -9.0063562, \, 0.29250820 ) $ &  $ ( 9.0509318, \, 15.890942)$ \\
      \hline 
   $0.8$ &   $( -9.3703131, \, -0.53501344 )$ &  $(  -9.0773190, \, 0.30326095 )  $ &   $ (   9.0254199, \, 15.903590)$\\
      \hline
   $0.9$ & $(   -9.2937388, \, -0.56278767)$ &  $(-9.1524155, \,  0.30206829)  $  &  $( 8.9998827, \, 15.916206)$ \\
      \hline
  $1$ &  $(  -9.2229969, \,  -0.43292838) R(\theta)$  & $ (  -9.2229969, \, 0.43292838)R(\theta)$ &   $( 9.2229969, \, 15.785861)R(\theta)$  \\
 \hline
\end{tabular}
\caption{\label{table8p4} The positions of $\hat{q}_{i, j}=\hat{q}_{i}(\frac{j}{10}) \, (i=1,2,3, \, j=0,1,2, \dots, 10)$ in the path $\mathcal{P}_{test2}=\mathcal{P}_{test2, \, \theta, \, \theta_0}$ corresponding to each $\theta  \in [0.0024 \pi, \, 0.0055 \pi]$.}
\end{table}

\item $\theta_0= 0.0012 \pi$: $\text{it works for} \,  \theta \in (0, \, 0.0024 \pi ]$, in which $\hat{q}_{i, j} \, (i=1,2,3; j=0,1, 2, \dots, 10)$ of the test paths are given by Table \ref{table9p4};
\begin{table}[!htbp]
\centering
\tiny
\begin{tabular}{ |c|c|c|c|c|} 
 \hline
\multicolumn{4}{|c|}{$\theta_0= 0.0012 \pi, \quad \quad \quad \theta \in (0, \, 0.0024 \pi]$} \\
\hline
$t$ &  $\hat{q}_1$ &$\hat{q}_2$    &   $\hat{q}_3$ \\
  \hline 
  $0$ &  $(-28.2392, \,  0)$ &  $(-27.3156, \, 0)$ & $(   27.7774, \, 27.3067)$ \\
  \hline 
  $0.1$ & $(-28.234034, \, -0.083759708)$  &  $ (    -27.320771, \, 0.062809708 )$ & $ (  27.767108, \, 27.317087 )$  \\
 \hline
  $0.2$ & $( -28.217347, \,  -0.16567242)$  &  $( -27.337454, \, 0.12377242)$  & $(  27.756811, \, 27.327469)$  \\
   \hline
  $0.3$ & $(   -28.189606, \, -0.24395190 )$  & $(   -27.365182, \, 0.18110192 )$  & $( 27.746509, \, 27.337845)$ \\
   \hline
  $0.4$ &  $( -28.151551, \,  -0.31691704)$  & $(   -27.403214, \, 0.23311708 ) $  &    $( 27.736203, \, 27.348215) $\\
     \hline
   $0.5$ & $( -28.104167, \,  -0.38303090)$ &  $ (  -27.450566, \, 0.27828097 ) $  &  $ (27.725891, \, 27.358579) $  \\
   \hline
  $0.6$ & $( -28.048664, \, -0.44093523 )$ &  $  (    -27.506027, \,  0.31523536)$ & $( 27.715576, \, 27.368938)$   \\
     \hline
   $0.7$ & $(  -27.986439, \, -0.48947998)$ & $ (   -27.568202, \, 0.34283018 ) $ &  $ ( 27.705255, \, 27.379291)$ \\
      \hline 
   $0.8$ &   $(  -27.919043, \,  -0.52774754)$ &  $(    -27.635537, \, 0.36014785 )  $ &   $ (     27.694930, \, 27.389639)$\\
      \hline
   $0.9$ & $(  -27.848149, \,  -0.55507162 )$ &  $(-27.706362, \, 0.36652205)  $  &  $(  27.684601, \, 27.399981)$ \\
      \hline
  $1$ &  $(   -27.777459, \, -0.46633552) R(\theta)$  & $ (  -27.777459, \,0.46633552)R(\theta)$ &   $(   27.777459, \, 27.305793)R(\theta)$  \\
 \hline
\end{tabular}
\caption{\label{table9p4} The positions of $\hat{q}_{i, j}=\hat{q}_{i}(\frac{j}{10}) \, (i=1,2,3, \, j=0,1,2, \dots, 10)$ in the path $\mathcal{P}_{test2}= \mathcal{P}_{test2, \, \theta, \, \theta_0}$ corresponding to each $\theta  \in (0, \, 0.0024 \pi]$.}
\end{table}
\end{enumerate}

9 different figures (Fig.~\ref{fig3}, Fig.~\ref{fig4}) are given to show that the following inequality 
\[ \mathcal{A}_{test2}= \mathcal{A}(\mathcal{P}_{test2})< g_2(\theta) \]
holds for every $\theta \in (0, \, 0.0664 \pi]$. It follows that when $\theta \in (0, \, 0.0664 \pi]$, the action minimizer $\mathcal{P}_{Q_2}$ connecting $Q_{S}$ and $Q_{E_2}$ is free of collision. The proof is complete.

 \begin{figure}[!htbp]
 \begin{center}
 \includegraphics[width=4.6in]{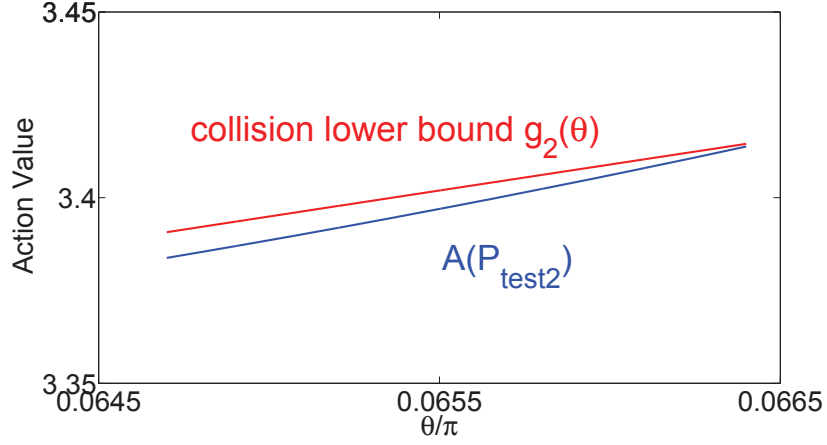}
 \end{center}
 \caption{\label{fig3} In the figure, the horizontal axis is $\theta/\pi$, and the vertical axis is the action value $\mathcal{A}$. When $\theta \in [0.0647 \pi, 0.0664 \pi]$,  the red curve is the graph of $g_2(\theta)$, the lower bound of action of boundary collision paths; while the purple curve is the graph of $\mathcal{A}(\mathcal{P}_{test2})$, action of the test path $\mathcal{P}_{test2}$.  }
    \end{figure}

 \begin{figure}[!htbp]
 \begin{center}
\subfigure[ $0.057 \pi \leq \theta \leq 0.0647\pi$ ]{\includegraphics[width=2.38in]{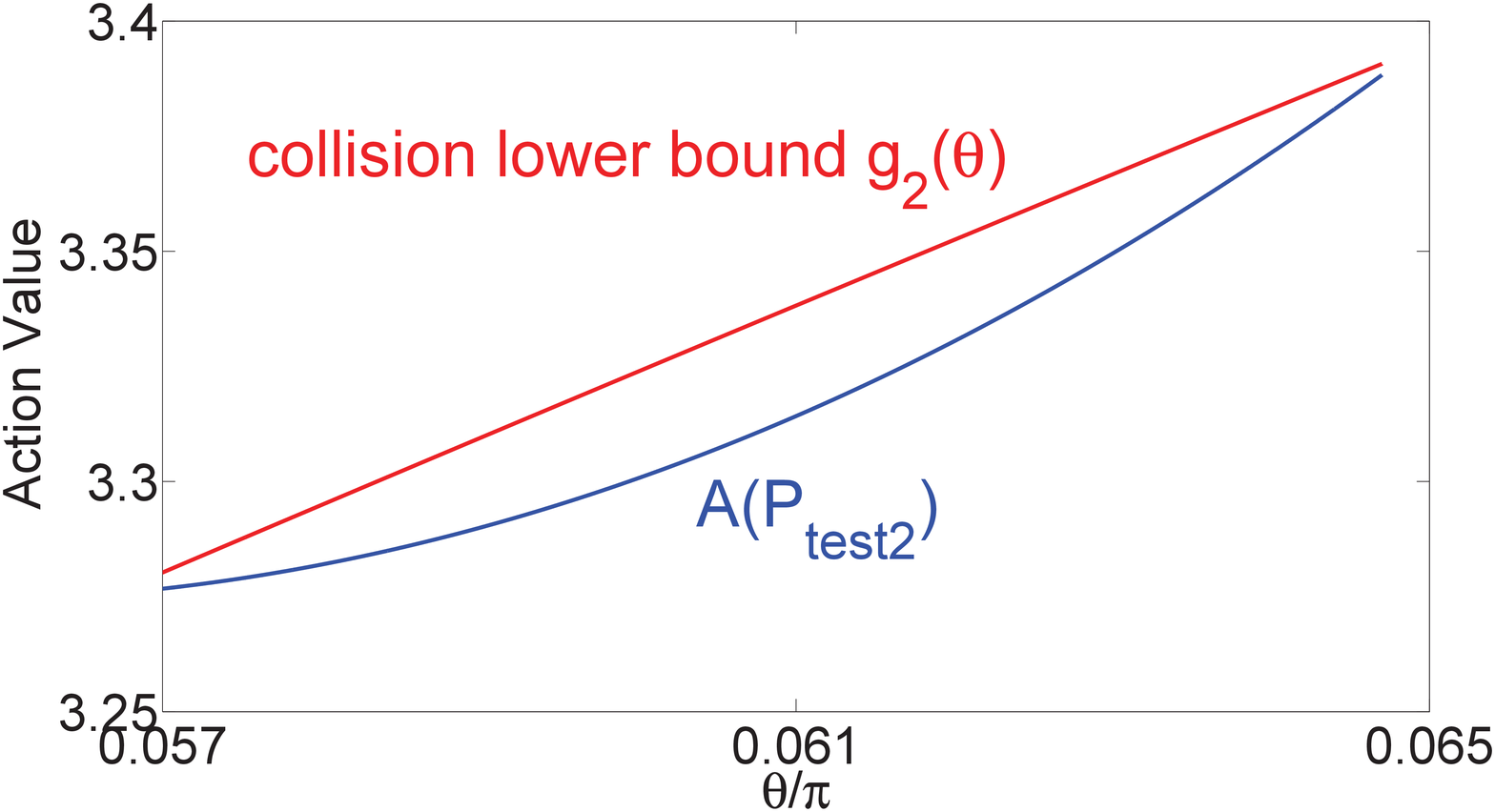}}
\subfigure[ $0.046\pi  \leq \theta \leq  0.057 \pi$ ]{\includegraphics[width=2.38in]{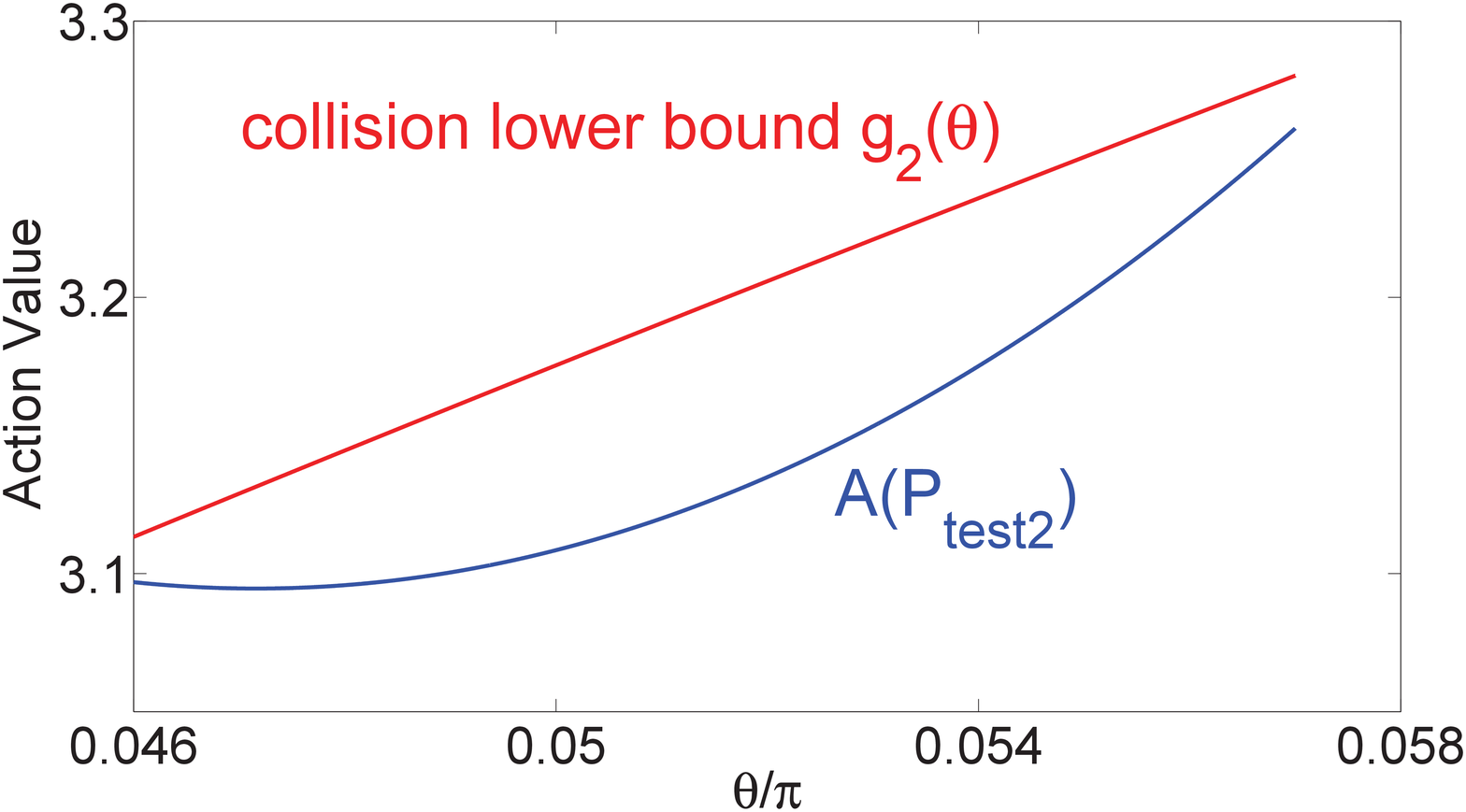}}
\subfigure[ $0.032 \pi  \leq \theta \leq  0.046\pi$ ]{\includegraphics[width=2.38in]{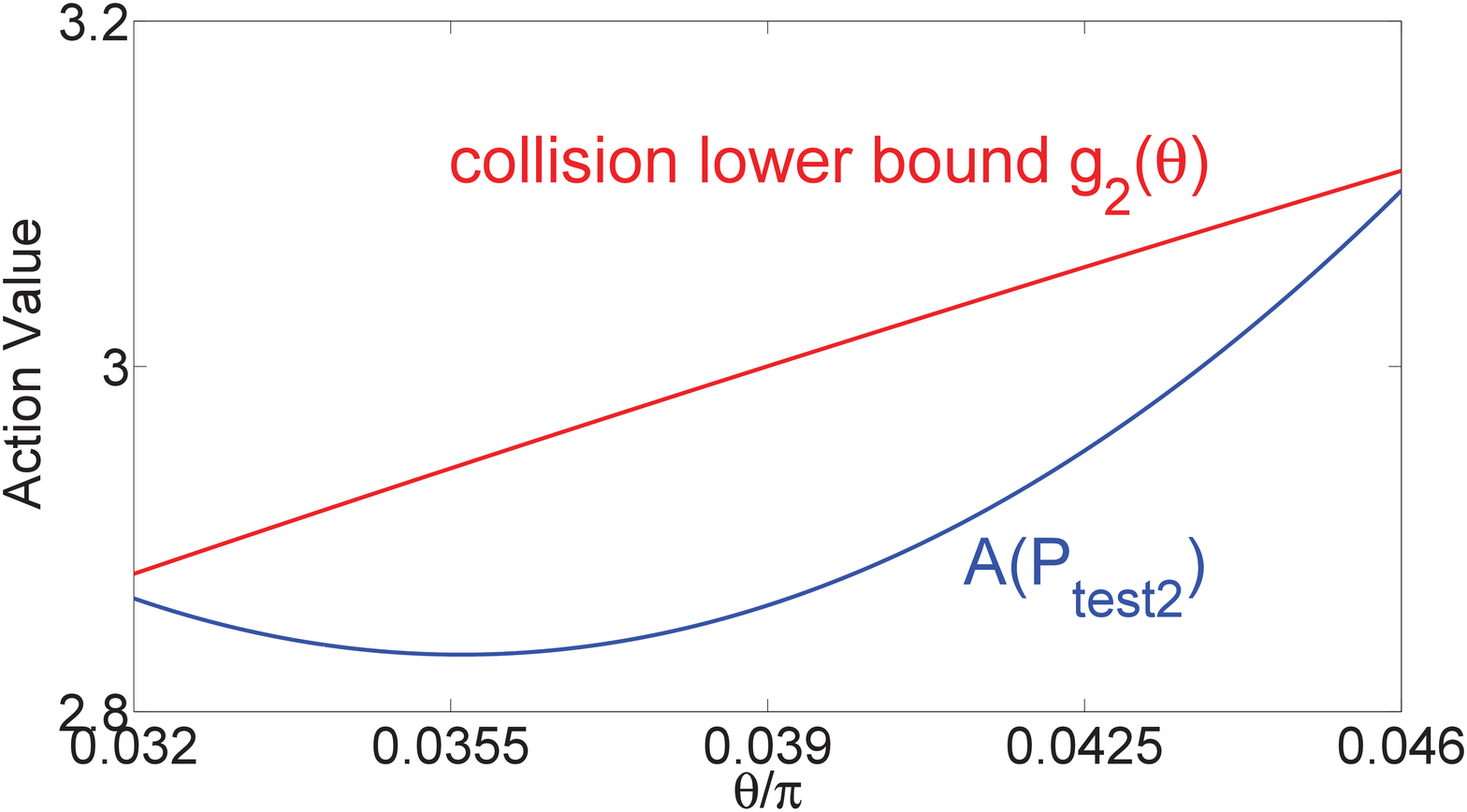}} 
\subfigure[ $0.025 \pi  \leq \theta \leq  0.032\pi$ ]{\includegraphics[width=2.38in]{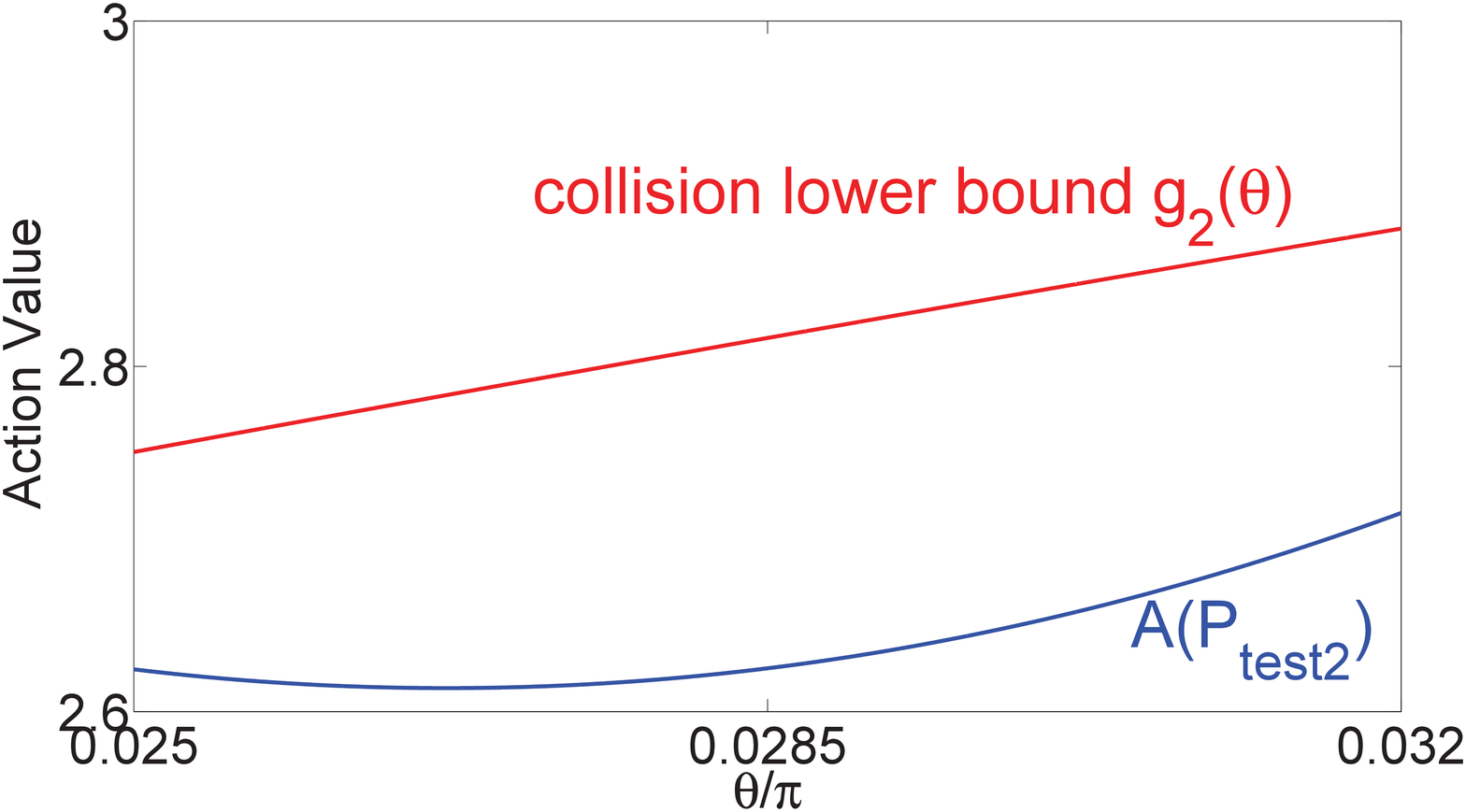}}
\subfigure[ $0.014 \pi  \leq \theta \leq  0.025\pi$ ]{\includegraphics[width=2.38in]{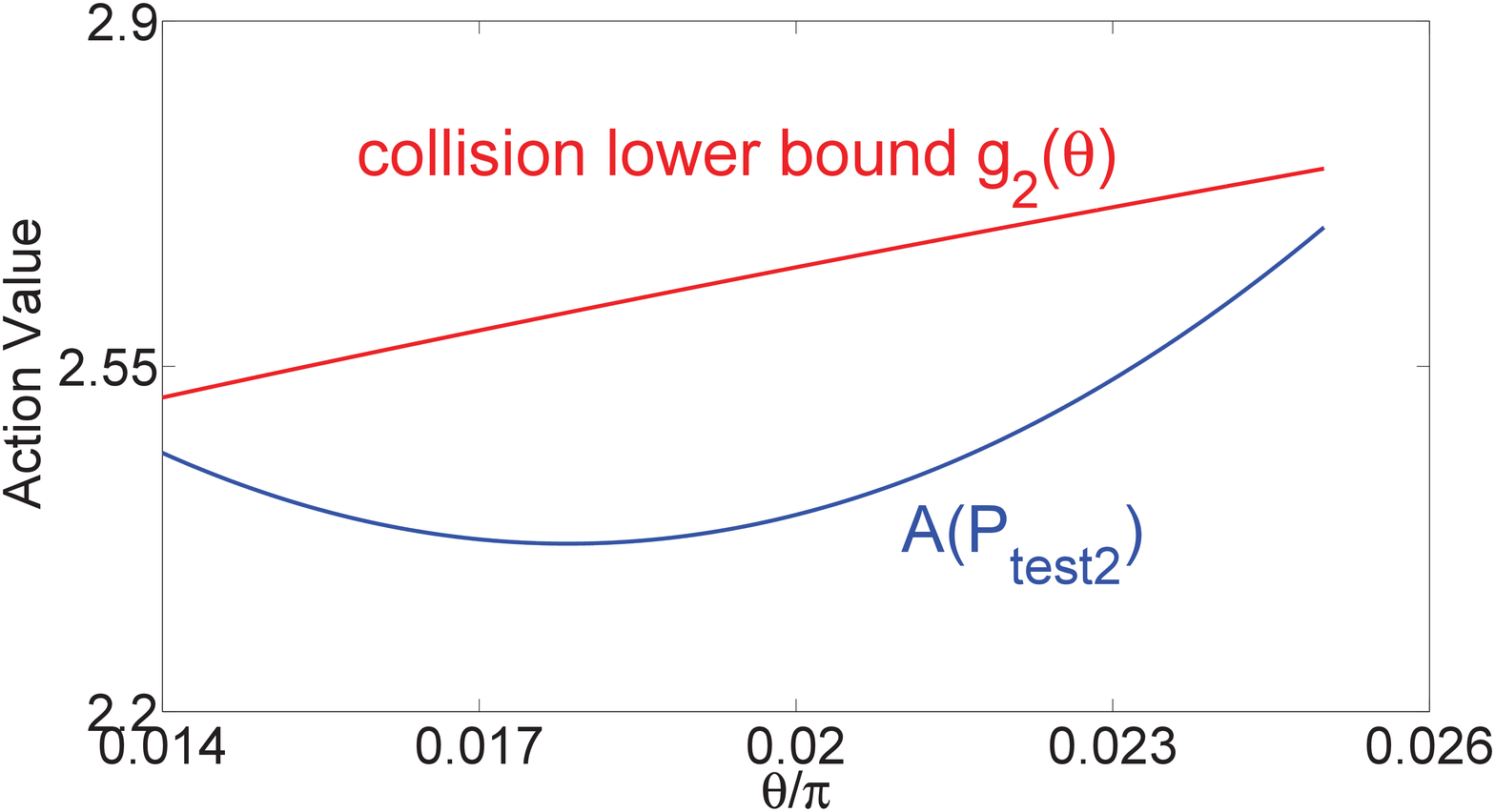}}
\subfigure[ $0.0055 \pi  \leq \theta \leq  0.014\pi$ ]{\includegraphics[width=2.38in]{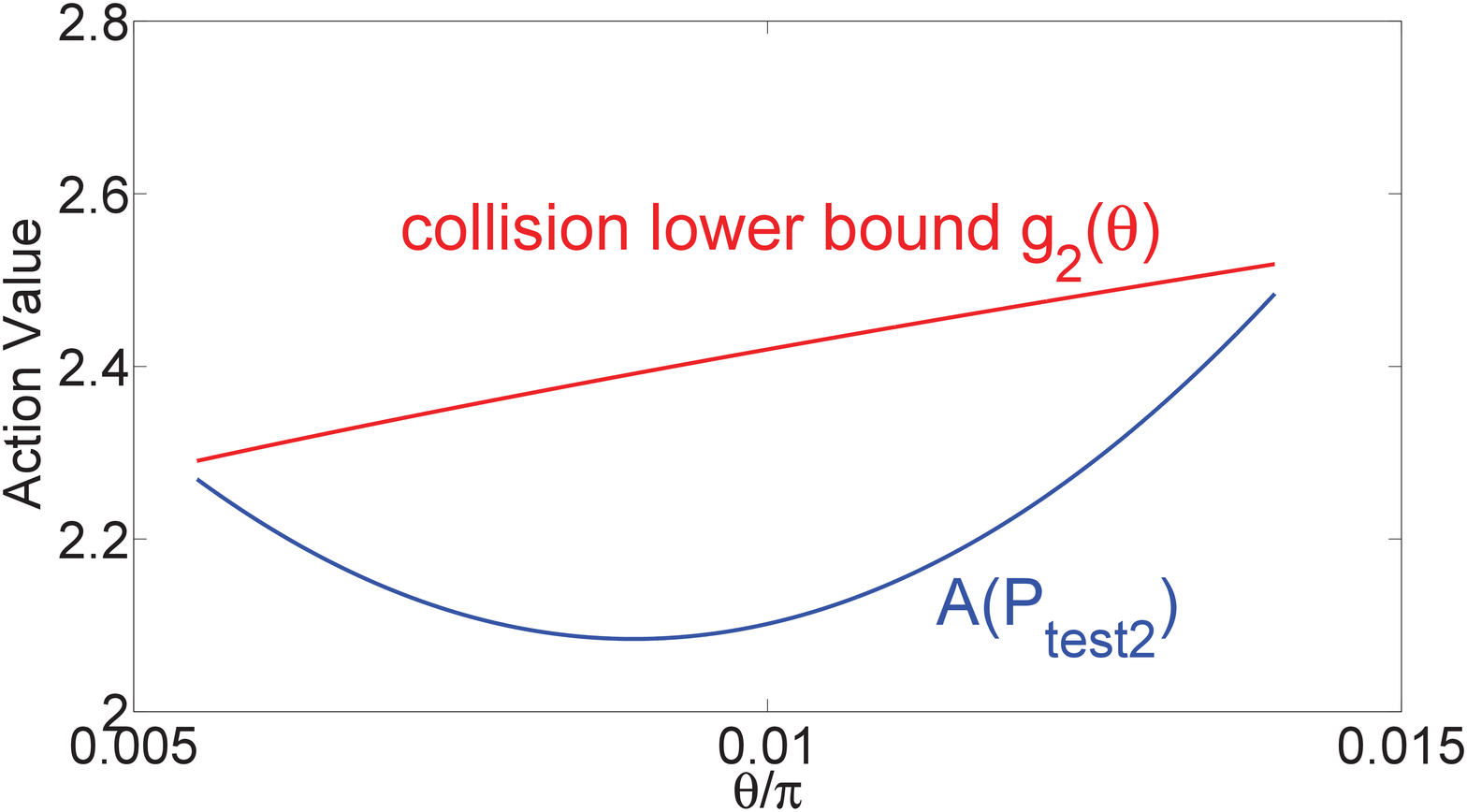}}
\subfigure[ $0.0024 \pi  \leq \theta \leq  0.0055\pi$ ]{\includegraphics[width=2.38in]{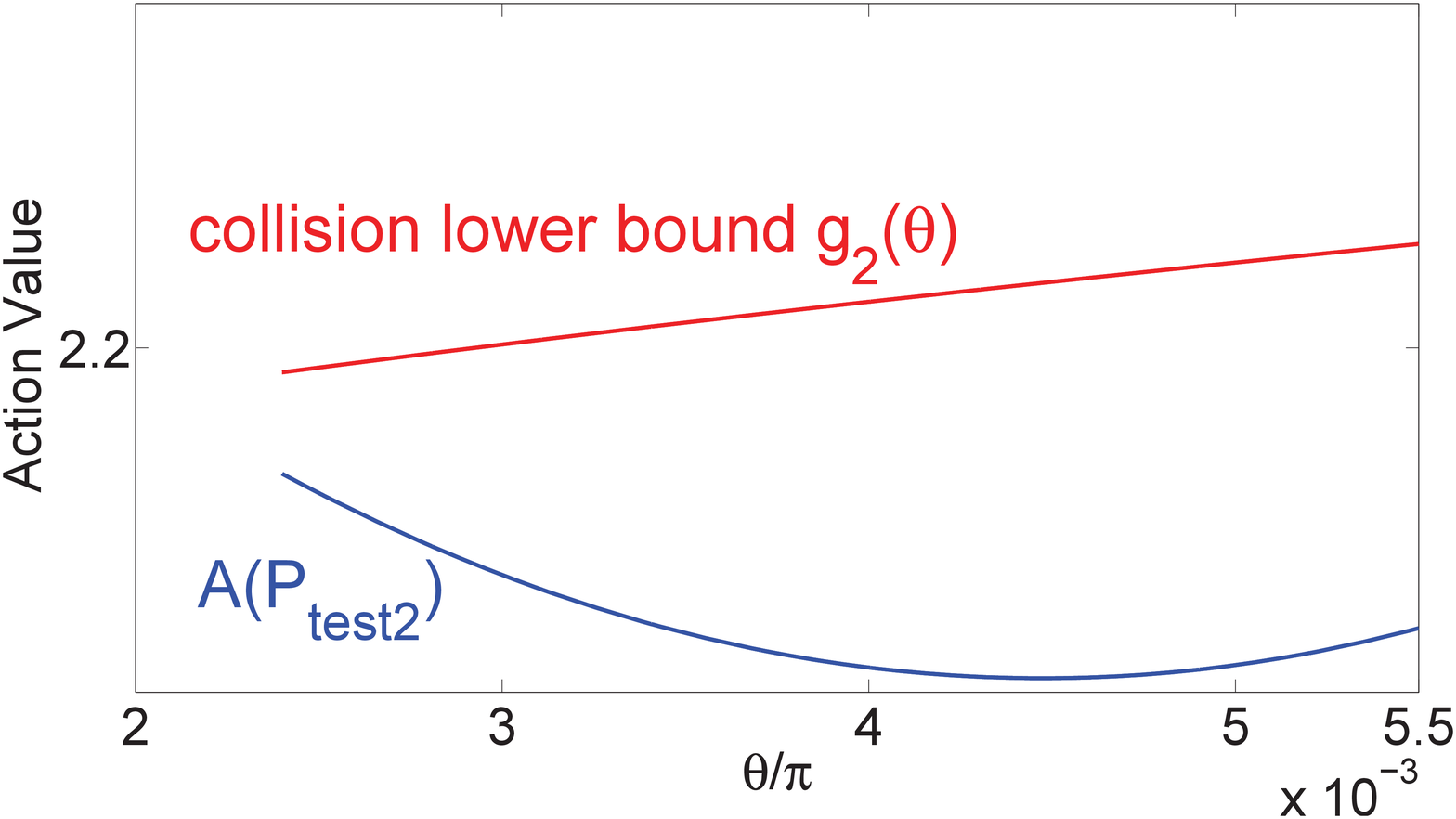}}
\subfigure[ $0 < \theta \leq  0.0024\pi$ ]{\includegraphics[width=2.38in]{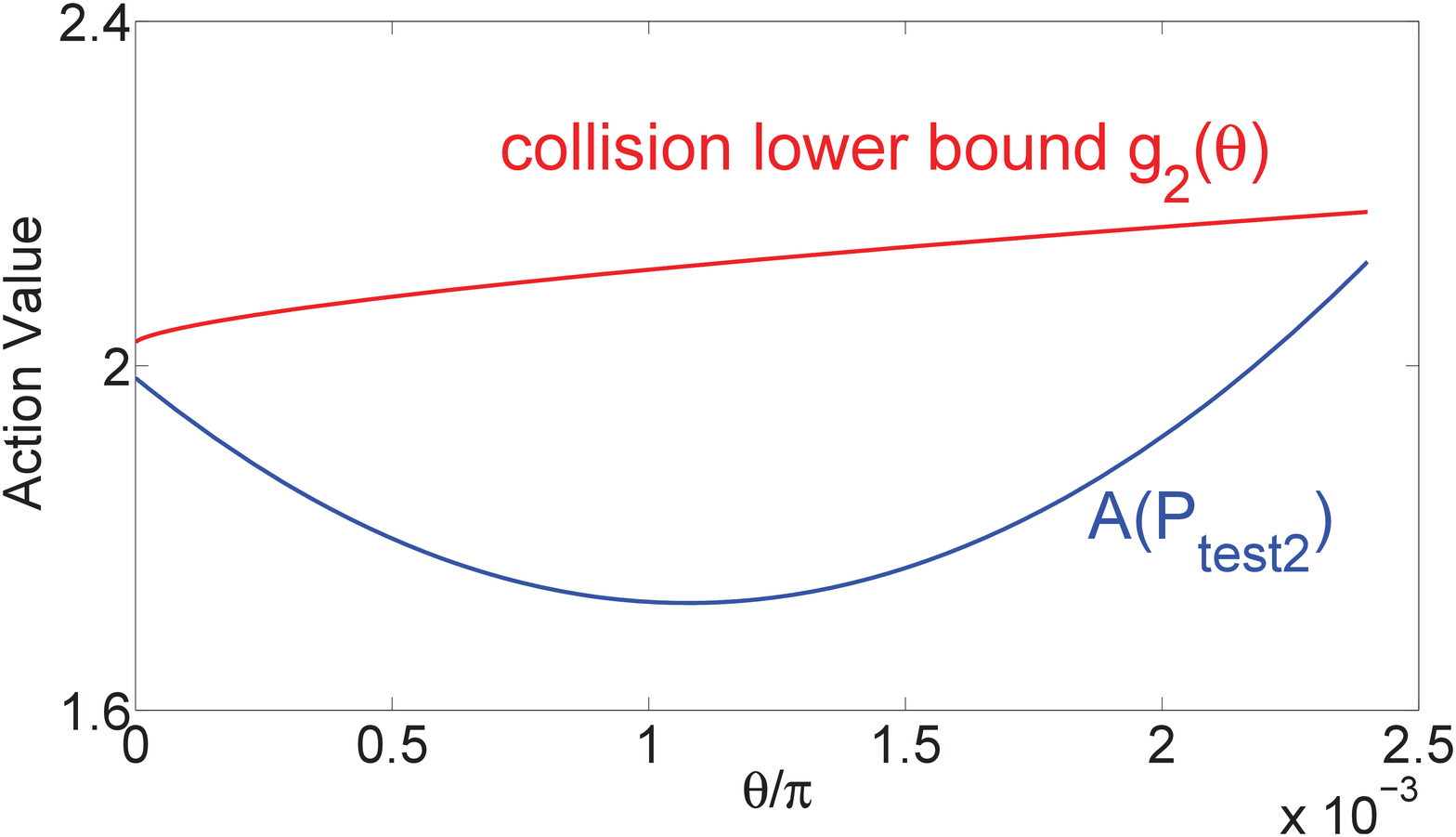}} 
 \end{center}
 \caption{\label{fig4} In each subfigure, the horizontal axis is $\theta/\pi$, and the vertical axis is the action value $\mathcal{A}$. The graphs of $g_2(\theta)$ (the lower bound of action of paths with boundary collisions) and the graphs of $\mathcal{A}(\mathcal{P}_{test2})$  (action of the test paths) are shown for different intervals of $\theta$. }
  \end{figure}

\end{proof}

\section{Existence of two sets of periodic or quasi-periodic orbits}\label{extperiodic}
Recall that the action minimizers $\mathcal{P}_{Q_i}= \mathcal{P}_{Q_i}([0,1]) \, (i=1,2)$ satisfy
\begin{equation}\label{pq1min}
\mathcal{A}(\mathcal{P}_{Q_i}) = \inf_{ \{ a_j\in \mathbb{R}, \, b_j \geq 0, \, c_j \geq 0 \, (j=1,2)\} } \inf_{\{ q(0)= Q_s, \, q(1)= Q_{e_i}, \, q(t) \in H^1([0, 1], \chi) \} }  \mathcal{A}, \, (i=1,2),
\end{equation}
where \begin{equation}
Q_s=\begin{bmatrix}
-a_1-c_1 & 0 \\
-a_1 &   0\\
(2a_1+c_1)/2 & b_1 \\
(2a_1+c_1)/2 & -b_1 
\end{bmatrix}, \quad a_1 \in \mathbb{R}, \, b_1\geq 0, \, c_1 \geq 0,  
\end{equation}
and the two ending configurations $Q_{e_1}$ and $Q_{e_2}$ are defined by
\begin{equation}
Q_{e_1}=\begin{bmatrix}
-b_2 & -a_2 \\
-c_2 &   a_2\\
c_2 & a_2 \\
b_2 & -a_2 
\end{bmatrix}R(\theta),  \qquad Q_{e_2}= \begin{bmatrix}
-a_2 & -b_2 \\
-a_2 &   b_2\\
a_2 & c_2 \\
a_2 & -c_2 
\end{bmatrix}R(\theta), 
\end{equation}
where $a_2 \in \mathbb{R}$, $b_2\geq 0$, $c_2 \geq 0$ and $R(\theta)=\begin{bmatrix}
 \cos(\theta)& \sin(\theta)\\
 -\sin(\theta)& \cos(\theta)
 \end{bmatrix}$. 
Let
\begin{equation}
 Q_S= \left\{ Q_s \, \bigg| \, a_1 \in \mathbb{R}, \, b_1 \geq 0, \,  c_1 \geq 0   \right\},
\end{equation}
\begin{equation}\label{QE2}
 Q_{E_i}= \left\{ Q_{e_i} \, \bigg| \, a_2 \in \mathbb{R}, \, b_2 \geq 0, \, c_2 \geq 0   \right\}, \quad (i=1,2).
\end{equation}
By Lemma \ref{lowerbdd1} in Section \ref{lowerbddcollision} and Lemma \ref{testpathdef1} in Section \ref{testpath}, the action minimizer $\mathcal{P}_{Q_1}= \mathcal{P}_{Q_1}([0,1])$ connecting $Q_{S}$ and $Q_{E_1}$ is free of collision when $\theta \in (0, 0.0539 \pi]$. Meanwhile, by Lemma \ref{lowerbdd2} in Section \ref{lowerbddcollision} and Lemma \ref{testpathdef2} in Section \ref{testpath},
the action minimizer $\mathcal{P}_{Q_2}=\mathcal{P}_{Q_2}([0,1])$ connecting $Q_{S}$ and $Q_{E_2}$ is free of collision when $\theta \in (0, 0.0664\pi]$. In this section, we show that if the two minimizers are free of collision, they can be extended to periodic or quasi-periodic orbits. 

\begin{theorem}\label{Q_sQ_e1orbitext}
When $\theta \in (0, 0.0539 \pi]$, the action minimizer $\mathcal{P}_{Q_1}= \mathcal{P}_{Q_1}([0,1])$, which connects the two boundary configuration sets $Q_{S}$ and $Q_{E_1}$, is collision-free and it can be extended to a periodic or quasi-periodic orbit.
\end{theorem} 
\begin{proof}
We set $q(t)=\begin{bmatrix}
q_{1}(t) \\
q_{2}(t) \\
q_{3}(t) \\
q_{4}(t)  
\end{bmatrix}$ to be the position matrix path of $\mathcal{P}_{Q_1}= \mathcal{P}_{Q_1}([0,1])$. Since $q(0) \in Q_{S}$ and $q(1) \in Q_{E_1}$, we can assume  
\begin{equation}
q(0)=\begin{bmatrix}
-a_{11}-c_{11} & 0 \\
-a_{11} &   0\\
(2a_{11}+c_{11})/2 & b_{11} \\
(2a_{11}+c_{11})/2 & -b_{11} 
\end{bmatrix}, \, \,  q(1)=\begin{bmatrix}
-b_{21} & -a_{21} \\
-c_{21} &   a_{21}\\
c_{21} & a_{21} \\
b_{21} & -a_{21} 
\end{bmatrix}R(\theta),
\end{equation}
where $a_{11}, a_{21} \in \mathbb{R}$ and the other four constants $b_{11}, b_{21}, c_{11}, c_{21}$ are nonnegative. Note that when $\theta \in (0, 0.0539 \pi]$, the action minimizer $\mathcal{P}_{Q_1}= \mathcal{P}_{Q_1}([0,1])$ is collision-free. By Lemma \ref{extensionformula1} and Corollary \ref{extensionformula2}, the path $\mathcal{P}_{Q_1}= \mathcal{P}_{Q_1}([0,1])$ can be extended to $\mathcal{P}_{Q_1}([-1,2])$. In general, the extension formula can be defined as follows:
 \begin{equation}\label{defofpq1ext}
 q(t)=\begin{cases}
 (q_1^T(t),\, q_2^T(t),\, q_3^T(t),\, q_4^T(t))^T, & t \in [0, \, 1],\\
 \\
 (-q_4^T(2-t),\, -q_3^T(2-t),\, -q_2^T(2-t), \, -q_1^T(2-t))^TBR(2\theta), &  t\in [1, \, 2],\\
 \\
(-q_3^T(t-2),\, -q_4^T(t-2),\, -q_2^T(t-2), \, -q_1^T(t-2))^T R(2\theta),&t\in [2, \, 4],  \\
 \\
 (q_2^T(t-4),\, q_1^T(t-4),\, q_4^T(t-4), \, q_3^T(t-4))^T  R(4 \theta), & t \in [4, \, 8], \\
 \\
  q(t-8k) R(8k \theta), \qquad \quad  t \in [8k, \, 8k+8], &
 \end{cases}
 \end{equation}
where $B=\begin{bmatrix}
1&0\\
0&-1
\end{bmatrix}$ and $k \in \mathbb{Z}$. Indeed, at $t=8$, 
\[ q_i(8) = q_i(0) R(8 \theta), \qquad   \dot{q}_i(8) = \dot{q}_i(0) R(8 \theta),  \quad i=1,2,3,4. \]
It implies that
\[ q(t+8) = q(t) R(8 \theta), \quad \quad t \in \mathbb{R}. \]
To show that $q(t)$ in \eqref{defofpq1ext} is a classical solution of the Newtonian equation, we only need to show that $q(t)$ is $C^1$. It is easy to check that $q(t)$ is continuous. Note that 
\begin{align}\label{velocityboundaryeqn}
& \dot{q}_{1x}(0)=\dot{q}_{2x}(0)=0,   &  \dot{q}_{3x}(0)=-\dot{q}_{4x}(0), \quad &  \quad \quad \dot{q}_{3y}(0)=\dot{q}_{4y}(0), \nonumber\\
& \dot{q}_1(1)= \dot{q}_4(1)B R(2 \theta),   &  \dot{q}_2(1)= \dot{q}_3(1)B R(2 \theta).   &
\end{align}
A direct calculation implies that $q(t)$ is $C^1$. If $\theta/\pi \in (0, 0.0539]$ is rational, we set $\displaystyle \frac{\theta}{\pi}= \frac{k_1}{l_1}$, where integers $k_1$, $l_1$ are relatively prime. It follows that $q(t+ 8l_1)= q(t)$. Hence, $q(t)$ is periodic. If $\theta \in (0, 0.0539 \pi]$ is irrational, then $q(t)$ is a quasi-periodic orbit. The proof is complete.
\end{proof}

\begin{theorem}\label{Q_sQ_e2orbitext}
When $\theta \in (0, 0.0664 \pi]$, the action minimizer $\mathcal{P}_{Q_2}= \mathcal{P}_{Q_2}([0,1])$, which connects the two boundary configuration sets $Q_{S}$ and $Q_{E_2}$, is collision-free and it can be extended to a periodic or quasi-periodic orbit.
\end{theorem} 
\begin{proof}
The proof follows by Lemma \ref{extensionformula1} and Corollary \ref{extensionformula2}. Let $\tilde{q}(t)= \begin{bmatrix}
\tilde{q}_{1}(t) \\
\tilde{q}_{2}(t) \\
\tilde{q}_{3}(t) \\
\tilde{q}_{4}(t)  
\end{bmatrix}$ be the position matrix path of the action minimizer $\mathcal{P}_{Q_2}=\mathcal{P}_{Q_2}([0,1])$ connecting $Q_S$ and $Q_{E_2}$. Since $\tilde{q}(0)$ and $\tilde{q}(1)$ in $\mathcal{P}_{Q_2}$ satisfy the boundary configurations, we can assume
\begin{equation}
\tilde{q}(0)=\begin{bmatrix}
-a_{12}-c_{12} & 0 \\
-a_{12} &   0\\
(2a_{12}+c_{12})/2 & b_{12} \\
(2a_{12}+c_{12})/2 & -b_{12} 
\end{bmatrix}, \, \, \tilde{q}(1)= \begin{bmatrix}
-a_{22} & -b_{22} \\
-a_{22} &   b_{22}\\
a_{22} & c_{22} \\
a_{22} & -c_{22} 
\end{bmatrix}R(\theta), 
\end{equation}
where $a_{12}, a_{22} \in \mathbb{R}$ and the other four constants $b_{12}, b_{22}, c_{12}, c_{22}$ are nonnegative. When $\theta \in (0, 0.0664 \pi]$, Lemma \ref{lowerbdd2} and Lemma \ref{testpathdef2} imply that $\mathcal{P}_{Q_2}=\mathcal{P}_{Q_2}([0,1])$ is collision-free. A general extension formula of $\tilde{q}(t)$ can be defined as follows.
 \begin{equation}\label{defofpq2ext}
 \tilde{q}(t)=\begin{cases}
 (\tilde{q}_1^T(t),\, \tilde{q}_2^T(t),\, \tilde{q}_3^T(t),\, \tilde{q}_4^T(t))^T, & t \in [0, \, 1],\\
 \\
 (\tilde{q}_2^T(2-t),\, \tilde{q}_1^T(2-t),\, \tilde{q}_4^T(2-t), \, \tilde{q}_3^T(2-t))^TBR(2\theta), &  t\in [1, \, 2],\\
 \\
(\tilde{q}_2^T(t-2),\, \tilde{q}_1^T(t-2),\, \tilde{q}_3^T(t-2), \, \tilde{q}_4^T(t-2))^T R(2\theta),&t\in [2, \, 4],  \\
 \\
  q(t-4k) R(4k \theta), \qquad \quad  t \in [4k, \, 4k+4], &
 \end{cases}
 \end{equation}
where $B=\begin{bmatrix}
1&0\\
0&-1
\end{bmatrix}$ and $k \in \mathbb{Z}$. Note that at $t=4$, 
\[ \tilde{q}_i(4) = \tilde{q}_i(0) R(4 \theta), \qquad   \dot{\tilde{q}}_i(4) = \dot{\tilde{q}}_i(0) R(4 \theta),  \quad i=1,2,3,4.  \]
It follows that
\[ q(t+4) = q(t) R(4 \theta), \quad \quad t \in \mathbb{R}. \]
By Lemma \ref{extensionformula1} and Corollary \ref{extensionformula2}, the velocities $\dot{\tilde{q}}$ at $t=0$ and $t=1$ satisfy
\begin{align}\label{velocityboundaryeqn2}
& \dot{\tilde{q}}_{1x}(0)=\dot{\tilde{q}}_{2x}(0)=0,   &  \dot{\tilde{q}}_{3x}(0)=-\dot{\tilde{q}}_{4x}(0), \quad &  \quad \quad \dot{\tilde{q}}_{3y}(0)=\dot{\tilde{q}}_{4y}(0), \nonumber\\
& \dot{\tilde{q}}_1(1)= -\dot{\tilde{q}}_2(1)B R(2 \theta),   &  \dot{\tilde{q}}_3(1)= -\dot{\tilde{q}}_4(1)B R(2 \theta).   &
\end{align}
Similar to Theorem \ref{Q_sQ_e1orbitext}, a direct computation implies that $\tilde{q}(t)$ is $C^1$ for all $t \in \mathbb{R}$. It follows that  $\tilde{q}(t)$ in \eqref{defofpq2ext} is a classical solution of the Newtonian equation. If $\theta/\pi \in (0, 0.0664]$ is rational, we set $\displaystyle \frac{\theta}{\pi}= \frac{k_2}{l_2}$, where integers $k_2$, $l_2$ are relatively prime. It follows that $\tilde{q}(t+ 4l_2)= \tilde{q}(t)$. Hence, $\tilde{q}(t)$ is periodic. If $\theta \in (0, 0.0664 \pi]$ is irrational, then $\tilde{q}(t)$ is a quasi-periodic orbit. The proof is complete.
\end{proof}


\begin{thebibliography}{00}

\bibitem{BR} Broucke, R.: Classification of periodic orbits in the four- and five-body problems, {\em Ann. NY Acad. Sci.} {\bf 107} (2004), 408--421.

\bibitem{CM} Chenciner, A., Montgomery, R.: A remarkable periodic solution of the three-body problem in the case of equal masses, {\em Ann. of Math.} {\bf 152} (2000), 881--901.

\bibitem{CA} Chenciner, A.: Action minimizing solutions in the Newtonian n-body problem: from homology to symmetry, {\em Proceedings of the International Congress of Mathematicians (Beijing, 2002)}, Higher Ed. Press, Beijing, 279--294, 2002.

\bibitem{CV} Chenciner, A., Venturelli, A.: Minima de l'int\'{e}grale d'action du probl\'{e}me Newtonien de 4 corps de masses  \'{e}gales dans $\mathbb{R}^3$: orbites ``Hip-Hop", {\em Celest. Mech. Dyn. Astro.} {\bf 77} (2000), 139--151.

\bibitem{CH} Chen, K.: Existence and minimizing properties of retrograde orbits to the three-body problem with various choices of masses, {\em Annals of Math.} {\bf 167} (2008), 325--348.

\bibitem{CH1} Chen, K.: Binary decompositions for planar N-body problems and symmetric periodic solutions, {\em Arch. Ration. Mech. Anal.} {\bf 170} (2003), 247--276.

\bibitem{CH2} Chen, K., Lin, Y.: On action-minimizing retrograde and prograde orbits of the three-body problem, {\em Comm. Math. Phys.} {\bf 291} (2009), 403--441.

\bibitem{CH3} Chen, K.: Removing collision singularities from action minimizers for the N-body problem with free boundaries, {\em Arch. Ration. Mech. Anal.} {\bf 181} (2006), 311--331.

\bibitem{FT} Ferrario, D., Terracini, S.: On the existence of collisionless equivariant minimizers for the classical n-body problem, {\em Invent. Math.} {\bf 155} (2004), 305--362. 

\bibitem{FU} Fusco, G., Gronchi, G., Negrini, P.: Platonic polyhedra, topological constraints and periodic solutions of the classical N-body problem, {\em Invent. Math.} {\bf 185} (2011), 283--332. 

\bibitem{Gor} Gordon, W. B.: A minimizing property of Keplerian orbits, {\em Amer. J. Math.} {\bf 99} (1977), 961--971.

\bibitem{LO}Long, Y.: Index theory for symplectic paths with applications, {\em Progress in Mathematics},  {\bf 207}, Birkh\"{a}user, Boston (2002).

\bibitem{Mar} Marchal, C.: How the method of minimization of action avoids singularities, {\em Celest. Mech. Dyn. Astro.} {\bf 83} (2002), 325--353.

\bibitem{Mon} Montgomery, R.: The N-body problem, the braid group, and action-minimizing periodic solutions, {\em Nonlinerity} {\bf 11} (1998), 363--376.

\bibitem{Ou1} Ouyang, T., Xie, Z.: Star pentagon and many stable choreographic solutions of the Newtonian 4-body problem, {\em Physica D} {\bf 307} (2015), 61--76. 

\bibitem{Ou2} Kuang, W., Ouyang, T., Xie, Z., Yan, D.: The Broucke-H\'{e}non orbit and the Schubart orbit in the planar three-body problem with equal masses, (2016), arXiv:1607.00580. 

\bibitem{Simo} Sim\'{o}, C.: New families of solutions in N-body problems, {\em European Congress of Mathematics, Vol. I (Barcelona, 2000), 101Ð115, Progr. Math.}  {\bf 201}, Birkh\"{a}user, Basel, 2001. 

\bibitem{Van} Vanderbei, R.: New orbits for the n-body problem, {\em Ann. NY Acad. Sci.} {\bf 107} (2004), 422--433.

\bibitem{Ven} Venturelli, A.: Application de la minimisation de l'action au probl\`{e}me des N corps dans le plan et dans l'espace. {\em Thesis, Universit\'{e} de Paris 7}, 2002.

\bibitem{Yan} Yan, D., Ouyang, T., Xie, Z.: Classification of periodic orbits in the planar equal-mass four-body problem, {\em Discrete Contin. Dyn. Syst. 2015, Dynamical systems, differential equations and applications. 10th AIMS Conference. Suppl.} 1115--1124, 2015.

\bibitem{Yan1} Han, S., Huang, A.,  Ouyang, T., Yan, D.: New periodic orbits in the planar equal-mass five-body problem, {\em Commun. Nonlinear Sci. Numer. Simul.} {\bf 48} (2017), 425--438.

%\bibitem{Yan2} W. Kuang, Y. Long, D. Yan, Existence of prograde double-double orbits, preprint, 2017.

\bibitem{Yu} Yu, G.: Periodic solutions of the planar N-center problem with topological constraints, {\em Disc. Cont. Dyn. Sys.} {\bf 36} (2016), 5131--5162.

\bibitem{Yu1} Yu, G.: Simple choreography solutions of the Newtonian N-body problem, {\em Arch. Ration. Mech. Anal.} (2017), doi:10.1007/s00205-017-1116-1.

\bibitem{Yu2} Yu, G.: Spatial double choreographies of the Newtonian $2n$-body problem, {\em arXiv:} 1608.07956.

\bibitem{Zhang} Zhang, S., Zhou, Q.: Nonplanar and noncollision periodic solutions for N-body problems, {\em Dist. Cont. Dyn. Syst.} {\bf 10} (2004), 679--685.

\bibitem{Zh} Wang, Z.,  Zhang, S.: New periodic solutions for Newtonian n-body problems with dihedral group symmetry and topological constraints, {\em Arch. Ration. Mech. Anal.} {\bf 219} (2016), 1185--1206.

\end{thebibliography}
\end{document}